\documentclass[11pt]{article}
\usepackage{amsmath}
\usepackage{amssymb}
\usepackage{amsthm}
\usepackage{enumitem}
\usepackage{setspace}
\usepackage[margin=1in]{geometry}
\usepackage{graphicx,color,caption,subcaption}
\usepackage{wrapfig}
\parskip \medskipamount
\newtheorem{theorem}{Theorem}[section]
\newtheorem{lemma}[theorem]{Lemma}

\newtheorem{proposition}[theorem]{Proposition}

\theoremstyle{definition}
\newtheorem{remark}[theorem]{Remark}
\newtheorem{definition}[theorem]{Definition}

\newcommand{\Q}{\upsilon}

\newcommand{\ind}[1]{{\text{\Large $\mathfrak 1$}}\left(#1\right)}
\newcommand{\PR}{\mathbb{P}}
\newcommand{\PQ}{\mathbb{Q}}

\newcommand{\lr}[1]{\left(#1\right)}
\newcommand{\lrb}[1]{\left[#1\right]}
\newcommand{\lrc}[1]{\left\{#1\right\}}

\newcommand{\ball}[1]{\mathcal{B}\lr{#1}}
\newcommand{\ballQ}[1]{\mathcal{B}_\Q\lr{#1}}
\newcommand{\bddi}{\partial^\mathrm{i}}
\newcommand{\bddo}{\partial^\mathrm{o}}
\newcommand{\bdde}{\partial^\mathrm{e}}

\newcommand{\FPP}{\mathrm{FPP}}
\newcommand{\core}{\mathrm{core}}
\newcommand{\super}{\mathrm{super}}
\newcommand{\neigh}{\mathrm{neigh}}
\newcommand{\neighs}{\mathrm{neigh2}}
\newcommand{\enc}{\mathrm{enc}}
\newcommand{\couter}{{\mathrm{outer}}}
\newcommand{\couters}{{\couter/3}}
\newcommand{\compl}{\textrm{c}}
\newcommand{\eff}{{\mathrm{eff}}}
\newcommand{\bad}{{\mathrm{bad}}}
\newcommand{\crit}{{\mathrm{c}}}

\newcommand{\dir}{{\mathrm{dir}}}

\newcommand{\mdla}{{\small{MDLA}}}
\newcommand{\dla}{{\small{DLA}}}
\newcommand{\fpphe}{{\small{FPPHE}}}

\newcommand{\HDFIG}[1]{#1}

\renewcommand{\phi}{\varphi}
\renewcommand{\hat}{\widehat}

\begin{document}

\title{Multi-Particle Diffusion Limited Aggregation} 

\author{Vladas Sidoravicius\thanks{Supported in part by CNPq grants 308787/2011-0 and 476756/2012-0 and FAPERJ grant E-26/102.878/2012-BBP.}\ \ and Alexandre Stauffer\thanks{Supported by a Marie Curie Career Integration Grant PCIG13-GA-2013-618588 DSRELIS, and EPSRC Early Career Fellowship EP/N004566/1.}\\ $\,$ \\{\small{Courant Institute of Mathematical Sciences, New York}} \\
{\small{NYU-ECNU Institute of Mathematical Sciences at NYU Shanghai}} \\
{\small{CEMADEN, S\~ao Jos\'e dos Campos}} \\
{\small{and}} \\{\small{University of Bath}}}
\date{}

\maketitle

\begin{abstract}
We consider a stochastic aggregation model on $\mathbb{Z}^d$. 
Start with particles distributed according to the product Bernoulli measure with parameter $\mu$. 
In addition, start with an aggregate at the origin. 
Non-aggregated particles move as continuous-time simple random walks obeying the exclusion rule, whereas aggregated particles do not move. 
The aggregate grows by attaching particles to its surface whenever a particle attempts to jump onto it. 
This evolution is called multi-particle diffusion limited aggregation.
Our main result states that if on $d>1$ the initial density of particles is large enough, 
then with positive probability the aggregate has linearly growing arms; 
that is, there exists a constant $c>0$ so that at time $t$ the aggregate contains a point of distance at least $ct$ from the origin, for all $t$.

The key conceptual element of our analysis is the introduction and study of a new growth process. 
Consider a first passage percolation process, called type 1, starting from the origin. 
Whenever type 1 is about to occupy a new vertex, with positive probability, instead of doing it, it gives rise to another first passage percolation process, 
called type 2, which starts to spread from that vertex. 
Each vertex gets occupied only by the process that arrives to it first. 
This process may have three phases: extinction (type 1 gets eventually surrounded by type 2), 
coexistence (infinite clusters of both types emerge), 
and strong survival (type 1 produces an infinite cluster which entraps all type 2 clusters). 
Understanding the various phases of this process is of mathematical interest on its own right. 
We establish the existence of a strong survival phase, and use this to show our main result.

\end{abstract}

\section{Introduction}\label{sec:intro}

In this work we consider one of the classical aggregation processes, introduced  in  \cite{RM} (see also \cite{Voss}) 
with the goal of providing an example  of ``a simple and tractable" mathematical model of dendritic growth, for which theoretical and mathematical 
concepts and tools could be designed and tested on. Almost four decades later we still encounter tremendous  mathematical
challenges studying its geometric and dynamic properties, and understanding the driving mechanism lying behind the formation 
of fractal-like structures.

{\bf Multi-particle diffusion limited aggregation (\mdla).} We consider the following stochastic aggregation  model on $\mathbb{Z}^d, \; d\ge 1$. 
Start with an infinite collection of particles located at the vertices of the lattice, with at most one particle 
per vertex, and initially distributed according to the product Bernoulli measure with parameter $\mu\in(0,1)$. 
In addition, there is an aggregate, which initially consists of
only one special particle, placed at the origin. The system
evolves in continuous time.
Non-aggregated particles move as simple symmetric random walks 
obeying the exclusion rule, i.e.\ particles jump at rate $2d$ to a uniformly random neighbor, but
if the chosen neighbor already contains another non-aggregated particle, such jump is suppressed and the particle waits for the next attempt to jump. 
Aggregated particles do not move. 
Whenever a non-aggregated particle attempts to jump on a 
vertex occupied by the aggregate, the jump of 
this particle is suppressed, the particle becomes part of the aggregate, and never moves from that moment onwards.
Thus the aggregate grows by attaching particles to its surface whenever a particle attempts to 
jump onto it.  This evolution will be referred to as {\emph{multi-particle diffusion limited 
aggregation}}, \mdla; examples for different values of $\mu$ are shown in Figure~\ref{fig:mdla}.
\begin{figure}[htbp]
   \begin{center}
      \includegraphics[width=.45\textwidth]{\HDFIG{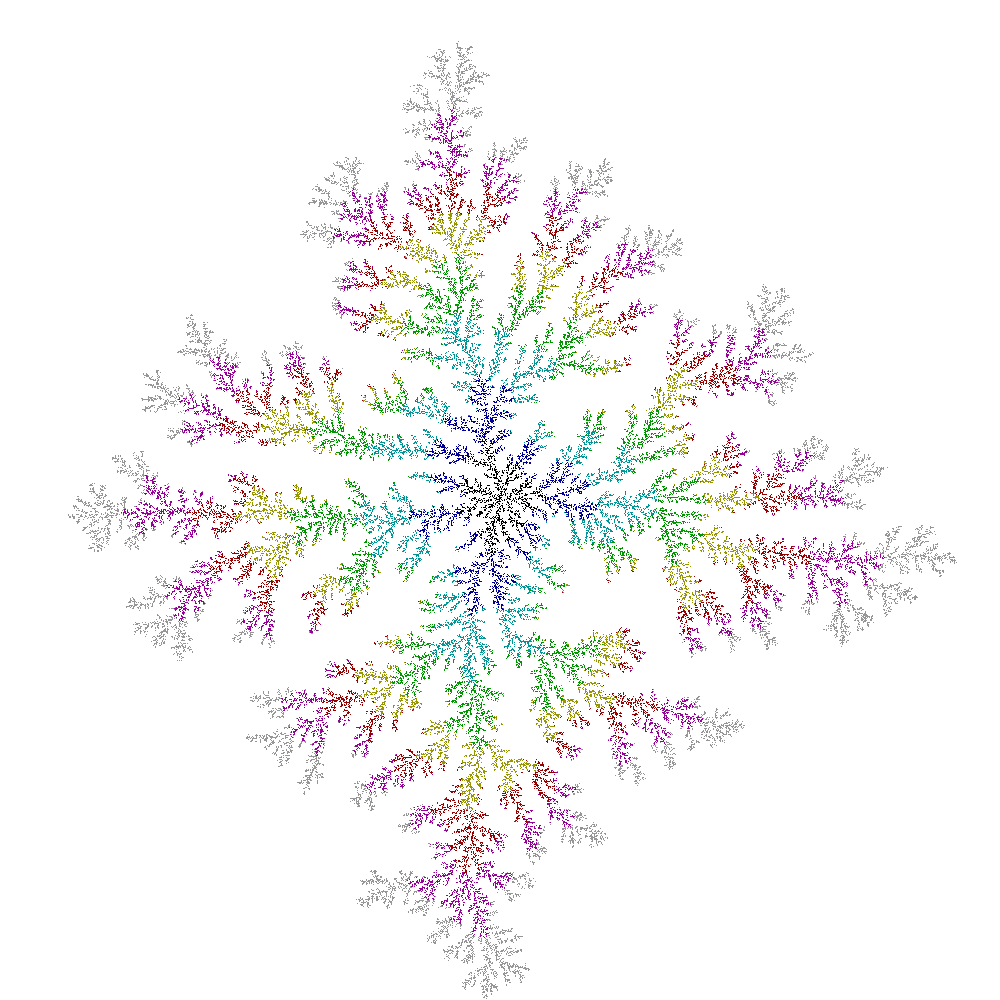}}
      \hspace{\stretch{1}}
      \includegraphics[width=.45\textwidth]{\HDFIG{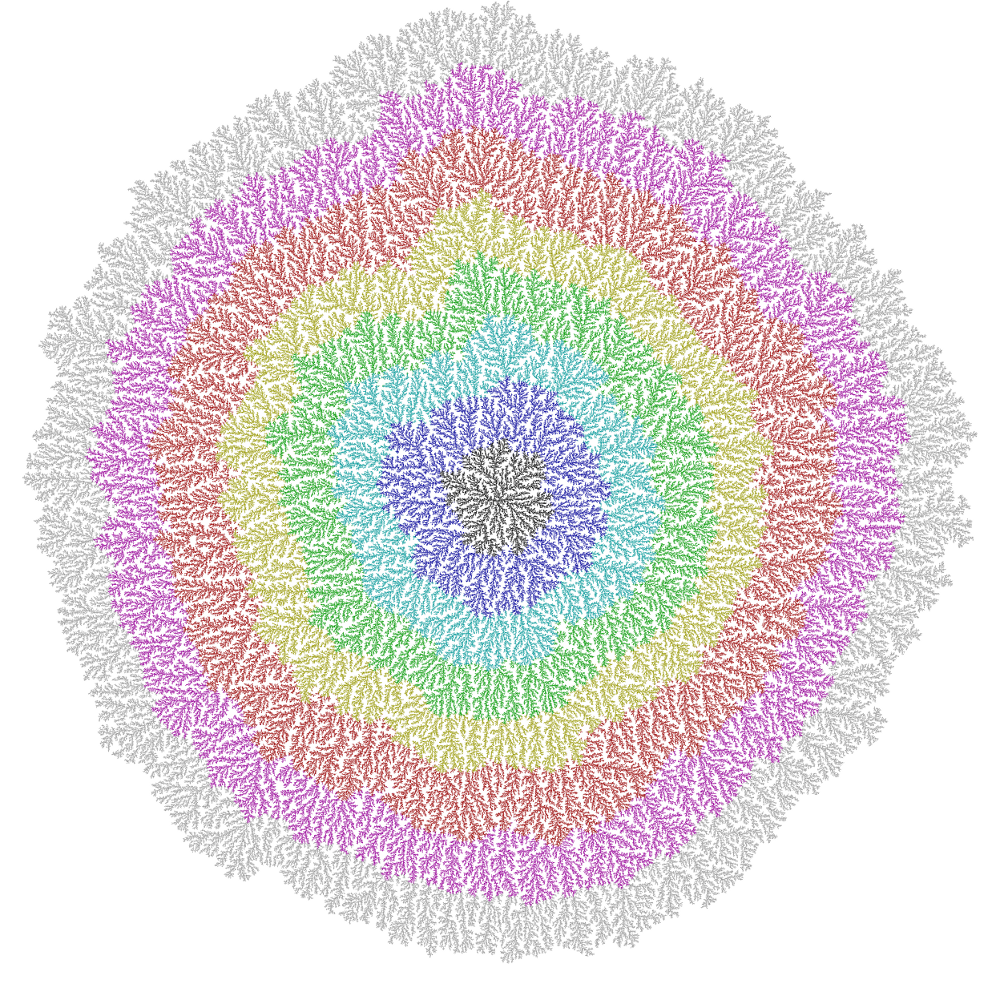}}
   \end{center}\vspace{-.5cm}
   \caption{\mdla{} with $\mu=0.1$ and $0.3$, respectively. Colors represent different epochs of the evolution of the process.}
   \label{fig:mdla}
\end{figure}

Characterizing the behavior of \mdla{} is a widely open and challenging problem. Existing mathematical results 
are limited to one dimension~\cite{KS_dla,CSwindle}. In this case, it is known that
the aggregate has almost surely \emph{sublinear} growth for any $\mu\in(0,1)$, having size of order $\sqrt{t}$ by time $t$. 
The main obstacle preventing the aggregate to grow with positive speed is 
that, from the point of view of the front (i.e., the rightmost point) of the aggregate, the density of particles decreases
since the aggregate grows by forming a region of density 1, larger than the initial density of particles.

In dimensions two and higher, \mdla{} seems to present a much richer and complex behavior, which changes substantially depending on the value of $\mu$;
refer to Figure~\ref{fig:mdla}. 
For small values of $\mu$, the low density of particles affects the rate of growth of the aggregate, as it needs to wait particles 
that move diffusively to find their way to its boundary. This suggests that the growth of the aggregate at small scales is governed by evolution of the 
``local'' harmonic measure of its boundary. This causes the aggregate to grow by protruding long fractal-like arms, similar to dendrites. 
On the other hand, when $\mu$ is large enough, the situation appears to be different. In this case, the aggregate is immersed in a very dense cloud of 
particles, and its growth follows random, dynamically evolving geodesics that deviate from 
occasional regions without particles.
Instead of showing dendritic type of growth, the aggregate forms a dense region characterized by the appearance of a limiting shape, 
similar to a first passage percolation process~\cite{Richardson,CoxDurrett}.
These two regimes do not seem to be exclusive. For intermediate values of $\mu$, the process shows 
the appearance of a limiting shape at macroscopic scales, while 
zooming in to mesoscopic and microscopic scales reveals rather complex ramified structures similar to dendritic growth, as in Figure~\ref{fig:mdla2}.

\begin{figure}[htbp]
   \begin{center}
      \includegraphics[width=.8\textwidth]{\HDFIG{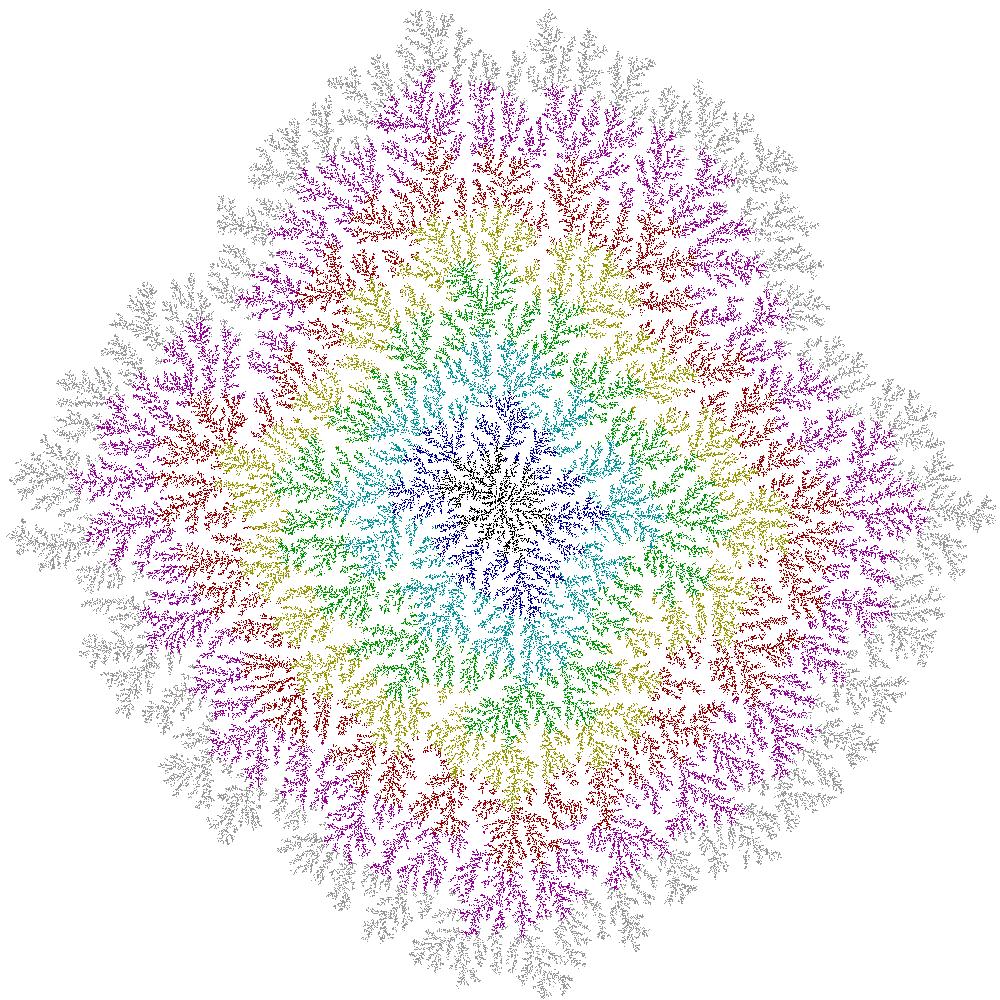}}
   \end{center}\vspace{-.5cm}
   \caption{\mdla{} with $\mu=0.2$, showing a limiting shape at macroscopic scales, but rather complex ramified structures at mesoscopic and microscopic scales.}
   \label{fig:mdla2}
\end{figure}

The main result of this paper is to establish that, unlike in dimension one, in dimensions $d\geq 2$ \mdla{} has a phase of \emph{linear} growth.
We actually prove a stronger result, showing that the aggregate grows with positive speed in \emph{all} directions.
For $t\geq0$, let $\mathcal{A}_t\subset\mathbb{Z}^d$ be the set of vertices occupied by the aggregate by time $t$, 
and let $\bar {\mathcal{A}}_t\supseteq \mathcal{A}_t$ be the set of vertices 
of $\mathbb{Z}^d$ that are not 
contained in the infinite component of $\mathbb{Z}^d\setminus \mathcal{A}_t$. 
Note that $\bar {\mathcal{A}}_t$ comprises all vertices of $\mathbb{Z}^d$ that either belong to the aggregate or are separated from infinity by the aggregate.
For $x\in\mathbb{Z}^d$ and $r\in\mathbb{R}_+$, we denote by $B(x,r)$ the ball of radius 
$r$ centered at $x$.
\begin{theorem}\label{thm:mdla}
   There exists $\mu_0\in(0,1)$ such that, for all $\mu>\mu_0$, 
   there are positive constants $c_1=c_1(\mu,d)$ and $c_2=c_2(\mu,d)$ for which
   $$
     \mathbb{P}\lr{\bar {\mathcal{A}}_t \supset B(0,c_1 t) \text{ for all $t\geq 0$}}>c_2.
   $$
\end{theorem}
\begin{remark}~\label{rmk:mdla}
   It is not difficult to see that the aggregate cannot grow faster than linearly. That is, there exists a constant $c_3$ such that the probability that 
   $\bar {\mathcal{A}}_t \subset B(0,c_3 t)$ for all $t\geq t_0$ goes to 1 with $t_0$. This is the case because the growth of the aggregate is slower than the growth 
   of a first passage percolation with exponential passage times of rate $1$, which has linear growth; see, for example,~\cite{Kingman,ADH}. 
\end{remark}

We believe that Theorem~\ref{thm:mdla} holds in a stronger form, 
with $\mathbb{P}\lr{\bar {\mathcal{A}}_t \supset B(0,c_1 t) \text{ for all $t\geq t_0$}}$ going to $1$ with $t_0$.
However, with positive probability, it happens that there is no particle within a large distance
to the origin at time 0. In this case, in the initial stages of the process, the aggregate will grow very slowly as if in a system with a small density of particles. 
We expect that the density of particles near the boundary of the aggregate will become close to $\mu$ after particles have moved for a large enough time, 
allowing the aggregate to start having positive speed of growth.
However, particles perform a non-equilibrium dynamics due to their interaction with the aggregate,
and the behavior and the effect of this initial stage of low density is not yet understood mathematically.
This is related to the problem of describing the behavior of \mdla{} for small values of $\mu$, which is still far from reach, and raises the challenging question 
of whether the aggregate has positive speed of growth for any $\mu>0$. 
Even in a heuristic level, it is not at all clear 
what the behavior of the aggregate should be for small $\mu$. 
On the one hand, the low density of particles causes the aggregate to grow slowly since particles move diffusively 
until they are aggregated. 
On the other hand, 
since the aggregate is immersed in a dense cloud of particles, 
this effect of slow growth could be restricted to small scales only, because
at very large scales the aggregate could simultaneously grow in many different directions. 

We now describe the ideas of the proof of Theorem~\ref{thm:mdla}.
For this we use the language of the dual representation of the exclusion process, where
vertices without particles are regarded as hosting another type of particles, called \emph{holes}, which perform among themselves
simple symmetric random walks obeying the exclusion rule.
When $\mu$ is large enough, at the initial stages of the process, the aggregate grows without encountering any hole. 
The growth of the aggregate is then equivalent to a first passage percolation process 
with independent exponential passage times. 
This stage is well understood: it is known that first passage percolation not only grows with positive speed, but also has a limiting shape~\cite{Richardson,CoxDurrett}. 
However, at some moment, the aggregate will start encountering holes. 
We can regard the aggregate as a solid wall for holes, as they can neither jump onto the aggregate nor be attached to the aggregate. 
In one dimension, holes end up accummulating at the boundary of the aggregate, and this is enough to prevent positive speed of growth. 
The situation is different in dimensions $d\geq 2$, 
since the aggregate is able to deviate from any hole it encounters, 
advancing through the particles that lie 
in the neighborhood of the hole until
it completely surrounds and entraps the hole. 
The problem is that the aggregate will find regions of holes of arbitrarily large sizes, which require a long time for the aggregate to go around them.
When $\mu$ is large enough, the regions of holes will be typically well spaced out, 
giving sufficient room for the aggregate to grow in-between the holes. 
One needs to show that the delays caused by deviation from holes are not large enough to prevent positive speed.
A challenge is that as
holes cannot jump onto the aggregate, their motion gets a drift whenever they are neighboring the aggregate. 
Hence, holes move according to a non-equilibrium dynamics, which creates difficulties in controlling the location of the holes. 
In order to overcome this problem, we introduce a new process to model the interplay between 
the aggregate and holes.

{\bf First passage percolation in a hostile environment (\fpphe).}
This is a two-type first passage percolation process.
At any time $t\geq0$, let $\eta^1(t)$ and $\eta^2(t)$ denote the vertices of $\mathbb{Z}^d$ occupied by type~1 and type~2, respectively.
We start with $\eta^1(0)$ containing only the origin of $\mathbb{Z}^d$, and $\eta^2(0)$ being a random set obtained by selecting each vertex of 
$\mathbb{Z}^d\setminus\{0\}$ with probability $p\in(0,1)$, independently of one another.
Both type~1 and type~2 are growing processes; i.e., for any times $t<t'$ we have $\eta^1(t)\subseteq\eta^1(t')$ and $\eta^2(t)\subseteq\eta^2(t')$.
Type~1 spreads from time $0$ throughout $\mathbb{Z}^d$ at rate $1$. Type~2 does not spread at first, and we denote $\eta^2(0)$ as \emph{type~2 seeds}.
Whenever the type~1 process attempts to occupy a vertex hosting a type~2 seed, the occupation is suppressed and that type~2 seed is \emph{activated} and starts to 
spread throughout $\mathbb{Z}^d$ at rate $\lambda\in(0,1)$. The other type~2 seeds remain inactive until type~1 or already activated type~2 attempts to occupy their location.
A vertex of the lattice is only occupied by the type that arrives to it first, so $\eta^1(t)$ and $\eta^2(t)$ are disjoint sets for all $t$;
this causes the two types to compete with each other for space. 
Note that type~2 spreads with smaller rate than type~1, but type~2 starts with a density of seeds while type~1 starts only from a single location.

We show that it is possible to analyze \mdla{} via a coupling with this process by showing that a hole that has been in contact 
with the aggregate will remain contained inside a cluster of type~2. Since the aggregate grows in the same way as type~1, 
establishing that the type~1 process grows with positive speed allows us to show that \mdla{} has linear growth.
Besides its application to studying \mdla{}, we believe that \fpphe{}
is an interesting process to analyze on its own right, as it shows fascinating different phases of behavior depending on the choice of $p$ and $\lambda$. 
An illustration of the behavior of this process is shown in Figure~\ref{fig:speed07}.
\begin{figure}[tbp]
   \begin{center}
      \includegraphics[width=.32\textwidth]{\HDFIG{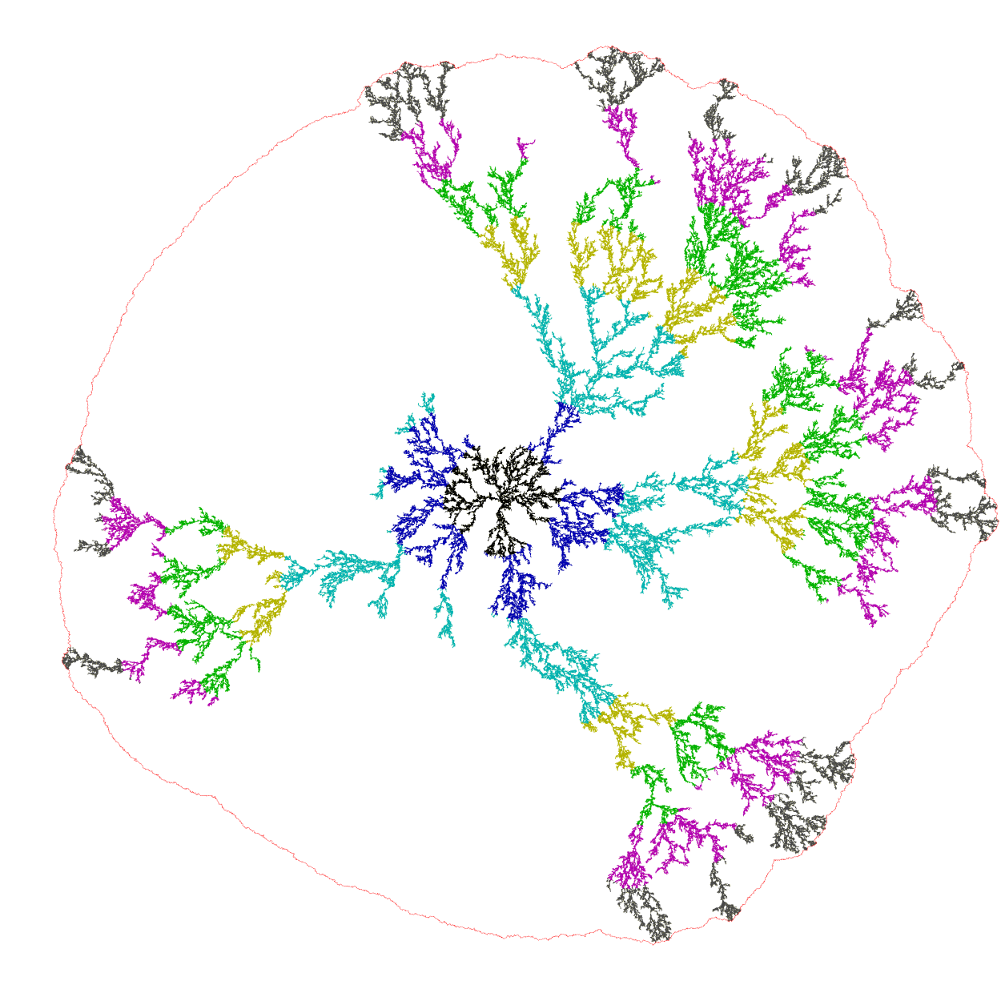}}
      \hspace{\stretch{1}}
      \includegraphics[width=.32\textwidth]{\HDFIG{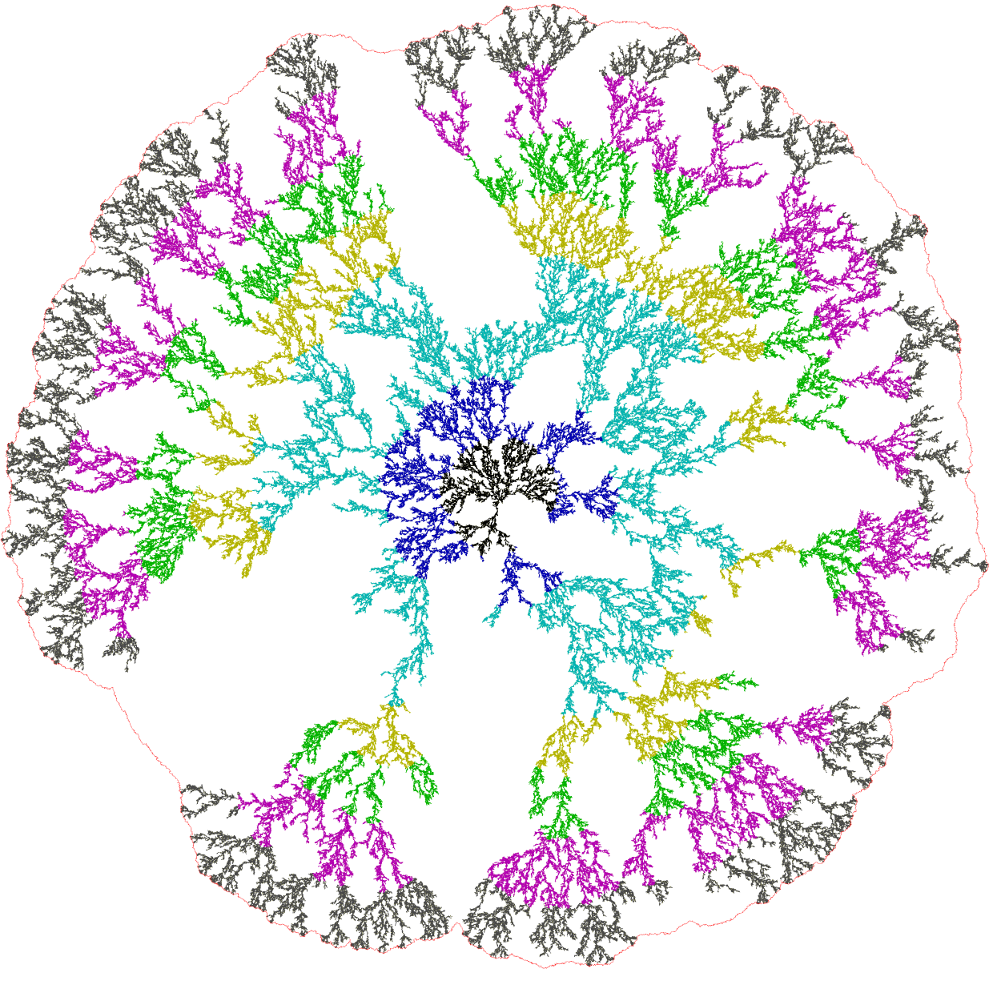}}
      \hspace{\stretch{1}}
      \includegraphics[width=.32\textwidth]{\HDFIG{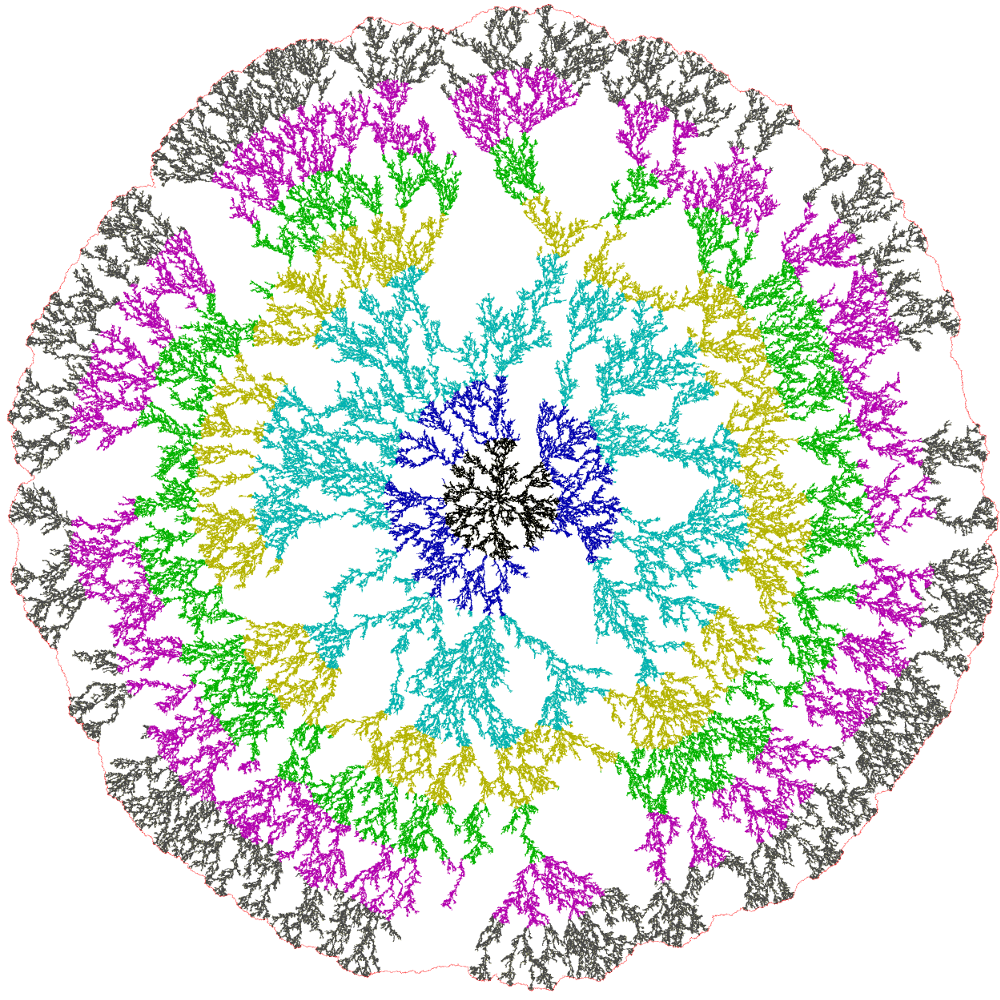}}
   \end{center}\vspace{-.5cm}
   \caption{\fpphe{} with $\lambda=0.7$ and $p=0.030, 0.029$ and $0.027$, respectively. Colors represent different epochs of 
   the growth of $\eta^1$, while the thin curve at the boundary represents the boundary between $\eta^2$ and vertices that are either unoccupied or host an inactive type~2 seed. 
   The whole white region within this boundary 
   is occupied by activated type~2.}
   \label{fig:speed07}
\end{figure}

The first phase is the \emph{extinction phase}, 
where type~1 stops growing in finite time with probability~$1$. 
This occurs, for example, when $p> 1-p_\crit$, with $p_\crit=p_\crit(d)$ being the critical probability for independent site percolation on $\mathbb{Z}^d$. 
In this case, with probability $1$, the origin is contained in a finite cluster of vertices not occupied by type~2 seeds, and hence type~1 will eventually 
stop growing. 
This extinction phase for type~1 also arises when $p\leq 1-p_\crit$ but $\lambda$ is close enough to $1$ so that type~2 clusters grow quickly enough to surround type~1 and confine it to a finite set.
 
We show in this work that another phase exists, called the \emph{strong survival phase}, and which is characterized by a positive probability of appearance 
of an infinite cluster of type~1, while type~2 is confined to form only finite clusters. 
Note that type~1 cannot form an infinite cluster with probability 1, 
since with positive probability all neighbors of the origin contain seeds of type~2.
Unlike the extinction phase, whose existence is quite trivial to show, 
the existence of a strong survival phase for some value of $p$ and $\lambda$ is far from
obvious.
Here we not only establish the existence of this phase, but we show that such a phase exists for \emph{any} $\lambda<1$ provided that $p$ is small enough. We also 
show that type~1 has positive speed of growth.
For any $t$, we define $\bar\eta^1(t)$ as the set of vertices of $\mathbb{Z}^d$ that are not contained in the infinite component of $\mathbb{Z}^d\setminus \eta^1(t)$, 
which comprises $\eta^1(t)$ and all vertices of $\mathbb{Z}^d\setminus \eta^1(t)$ that are separated from infinity by $\eta^1(t)$.
The theorem below will be proved in Section~\ref{sec:proof}, as a consequence of a more general theorem, Theorem~\ref{thm:fpp2}.
\begin{theorem}\label{thm:fpp}
   For any $\lambda<1$, there exists a value $p_0\in(0,1)$ such that, for all $p\in(0,p_0)$, there are positive constants 
   $c_1=c_1(p,d)$ and $c_2=c_2(p,d)$ for which
   $$
      \PR\lr{\bar\eta^1(t) \supseteq B(0,c_1t) \text{ for all $t\geq0$}}>c_2.
   $$
\end{theorem} 

There is a third possible regime, 
which we call the \emph{coexistence phase}, and is characterized by 
type~1 and type~2 simultaneously forming infinite clusters with positive probability. 
(We regard the coexistence phase as a regime of weak survival for type~1, in the sense that type~1 survives but leaves enough space for type~2 to produce 
at least one infinite cluster.)
Whether this regime actually occurs for some value of $p$ and $\lambda$ is an open problem, and even simulations do not seem to give 
good evidence of the existence of this regime. 
For example, in the rightmost picture of Figure~\ref{fig:speed07}, 
we observe a regime where $\eta^1$ survives, while $\eta^2$ seems to produce only finite clusters, but of quite long sizes. 
This also seems to be the behavior of the central picture in Figure~\ref{fig:speed07}, though it is not as clear whether each cluster of $\eta^2$ will be
eventually confined to a finite set.
However, the behavior in the leftmost picture of Figure~\ref{fig:speed07} is not at all clear. 
The cluster of $\eta^1$ has survived until the simulation was stopped, but produced a very thin set. 
It is not clear whether coexistence will happen in this situation, whether $\eta^1$ will eventually stop growing, or even whether 
after a much longer time the ``arms'' produced by $\eta^1$ will eventually find one another, 
constraining $\eta^2$ to produce only finite clusters. 

Establishing whether a coexistence phase exists for some value of $p$ and $\lambda$ is an 
interesting open problem.
We can establish that a coexistence phase occurs in a particular example of \fpphe{}, where 
type~1 and type~2 have deterministic passage times, with all randomness coming from the locations of the seeds.
In this example, all three phases occur. 
We discuss this in Section~\ref{sec:det}.
See also the recent paper~\cite{CSFPPHE}, where coexistence is established when $\mathbb{Z}^d$ is replaced by a hyperbolic non-amenable graph.

{\bf Historical remarks and related works.} \mdla{} belongs to a class of models, introduced firstly in the physics and chemistry literature
(see~\cite{Kassner} and references therein), 
and later in the mathematics literature as well, 
with the goal of studying geometric and dynamic  properties of static formations produced by aggregating randomly moving 
colloidal particles. 
Some numerically established quantities, such as fractal dimension, showed striking similarities 
between clusters produced by aggregating particles and
clusters produced in other growth processes of entirely different nature, such as dielectric breakdown cascades and Laplacian growth
models (in particular, Hele-Shaw cell~\cite{ST}). These similarities were further investigated by the introduction of 
the Hastings-Levitov growth model~\cite{HL}, which is represented as a sequence of 
conformal mappings. 
Nonetheless, it is still debated in the physics literature whether some of these models belong to the same universality class or not~\cite{Procaccia}. 
 
In the mathematics literature, 
the diffusion limited aggregation model (\dla), introduced in~\cite{WS} following the introduction of \mdla{} in~\cite{RM}, 
became a paradigm object of study among aggregation models driven by diffusive particles. 
However, 
progress on understanding \dla{} and \mdla{} mathematically has been relatively modest. 
The main results known about \dla{} are bounds on its rate of growth, derived by Kesten~\cite{Kesten1,Kesten2} (see also~\cite{Barlow}), but 
several variants have been introduced and studied~\cite{BPP,BY,Eldan,Martineau,CM,Silvestri}.
Regarding \mdla{}, it was rigorously studied only in the one-dimensional case~\cite{CSwindle,KS_dla,KS_dla2}, for which 
sublinear growth has been proved for all densities $p\in(0,1)$ in~\cite{KS_dla}.

{\bf Structure of the paper.} 
We start in Section~\ref{sec:det} with a discussion of an example of \fpphe{} where the passage times 
are deterministic, and show that this process has a coexistence phase. 
Then, in preparation for the proof of strong survival of \fpphe{} (Theorem~\ref{thm:fpp}),
we state in Section~\ref{sec:prelim} existing results on first passage percolation, 
and discuss in Section~\ref{sec:hp} a result due to H\"aggstrom and Pemantle regarding non-coexistence of a two-type 
first passage percolation process. This result plays a fundamental role in our analysis of \fpphe{}. 
Then, in Section~\ref{sec:proof}, we state and prove Theorem~\ref{thm:fpp2}, which is a more general version of Theorem~\ref{thm:fpp}. 
In Section~\ref{sec:mdla} we relate \fpphe{} with \mdla, giving the proof of Theorem~\ref{thm:mdla}.

\section{Example of coexistence in \fpphe{}}\label{sec:det}
In this section we consider \fpphe{} with \emph{deterministic} passage times. 
That is, whenever type~1 (resp., type~2) occupies a vertex $x\in\mathbb{Z}^d$, then after time $1$ (resp., $1/\lambda$) 
type~1 (resp., type~2) will occupy all unoccupied neighbors of $x$. If both type~1 and type~2 try to occupy a vertex at the same time, 
we choose one of them uniformly at random.
Recall that we denote by $\eta^i(t)$, $i\in\{1,2\}$, the set of vertices occupied by type $i$ by time $t$.
For simplicity, we restrict this discussion to dimension $d=2$.

\begin{figure}[htbp]
   \begin{center}
      \includegraphics[width=.32\textwidth]{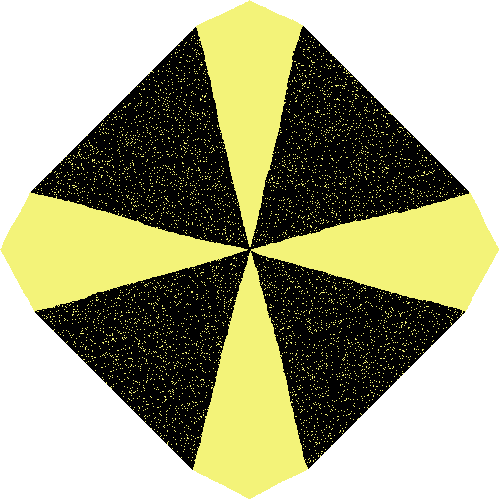}
      \hspace{\stretch{1}}
      \includegraphics[width=.32\textwidth]{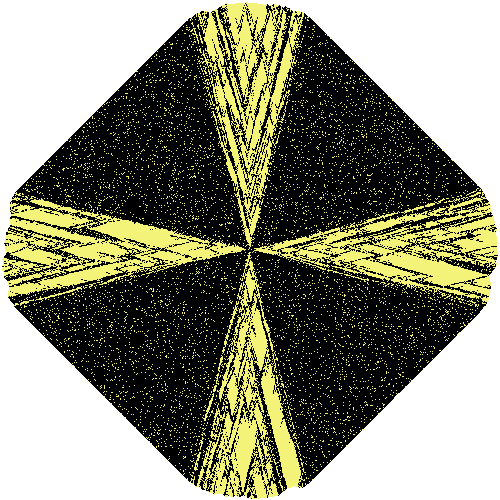}
      \hspace{\stretch{1}}
      \includegraphics[width=.32\textwidth]{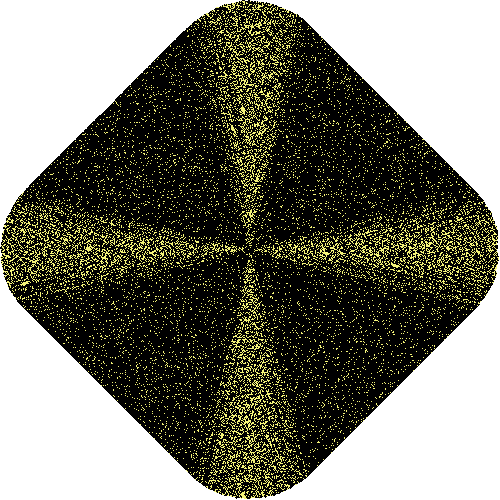}
   \end{center}\vspace{-.5cm}
   \caption{Simulation of \fpphe{} with deterministic passage times, and parameters $p=0.2$ and $\lambda=0.9, 0.8$ and $0.7$, respectively. 
      Black vertices are occupied by $\eta^1$ and yellow vertices are occupied by $\eta^2$.}
   \label{fig:speed09d}
\end{figure}
Figure~\ref{fig:speed09d} shows a simulation of this process for $p=0.2$ and different values of $\lambda$.
In all the three pictures in Figure~\ref{fig:speed09d}, $\eta^1$ seems to survive. 
However, note that the leftmost picture in 
Figure~\ref{fig:speed09d} differs from the other two since $\eta^2$ also seems to give rise to an infinite cluster, characterizing a regime of 
coexistence. See Figure~\ref{fig:detspeed09d} for more details.
\begin{figure}[htbp]
      \hspace{\stretch{1}}
      \includegraphics[width=.35\textwidth]{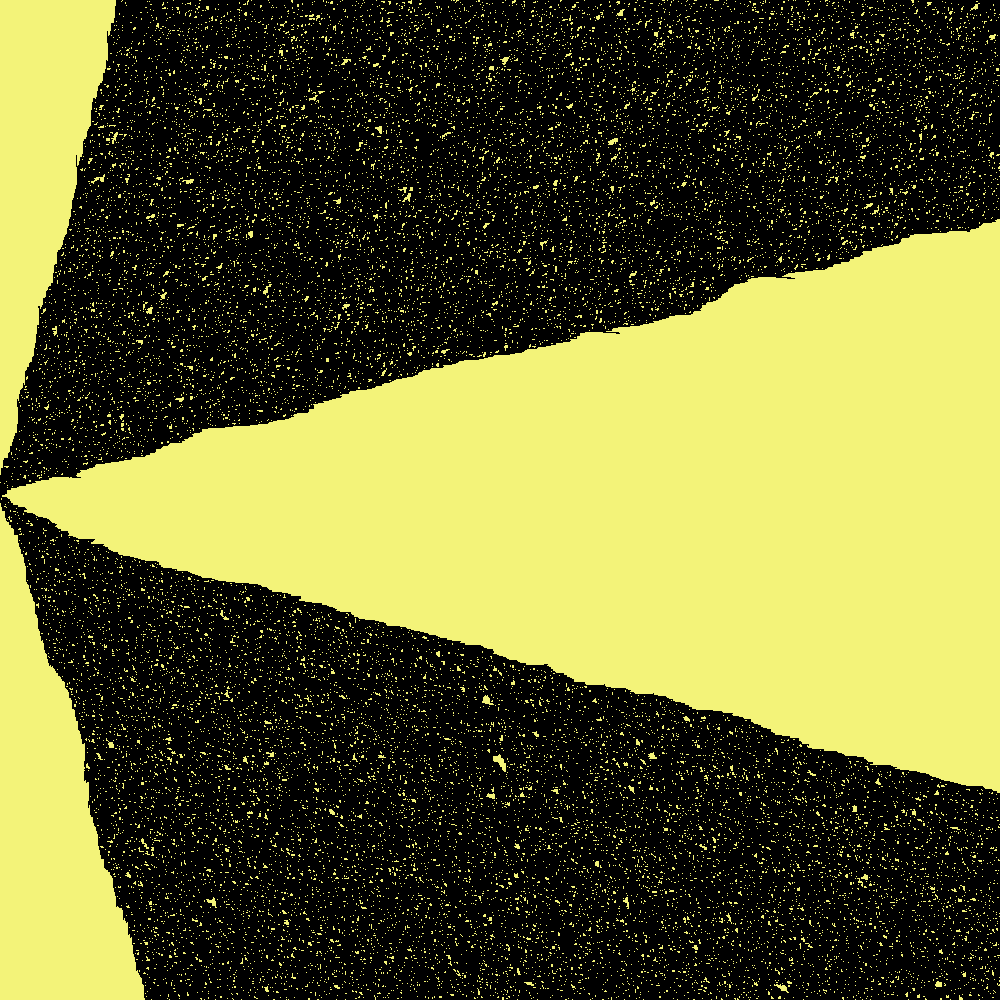}
      \hspace{\stretch{1}}
      \includegraphics[width=.35\textwidth]{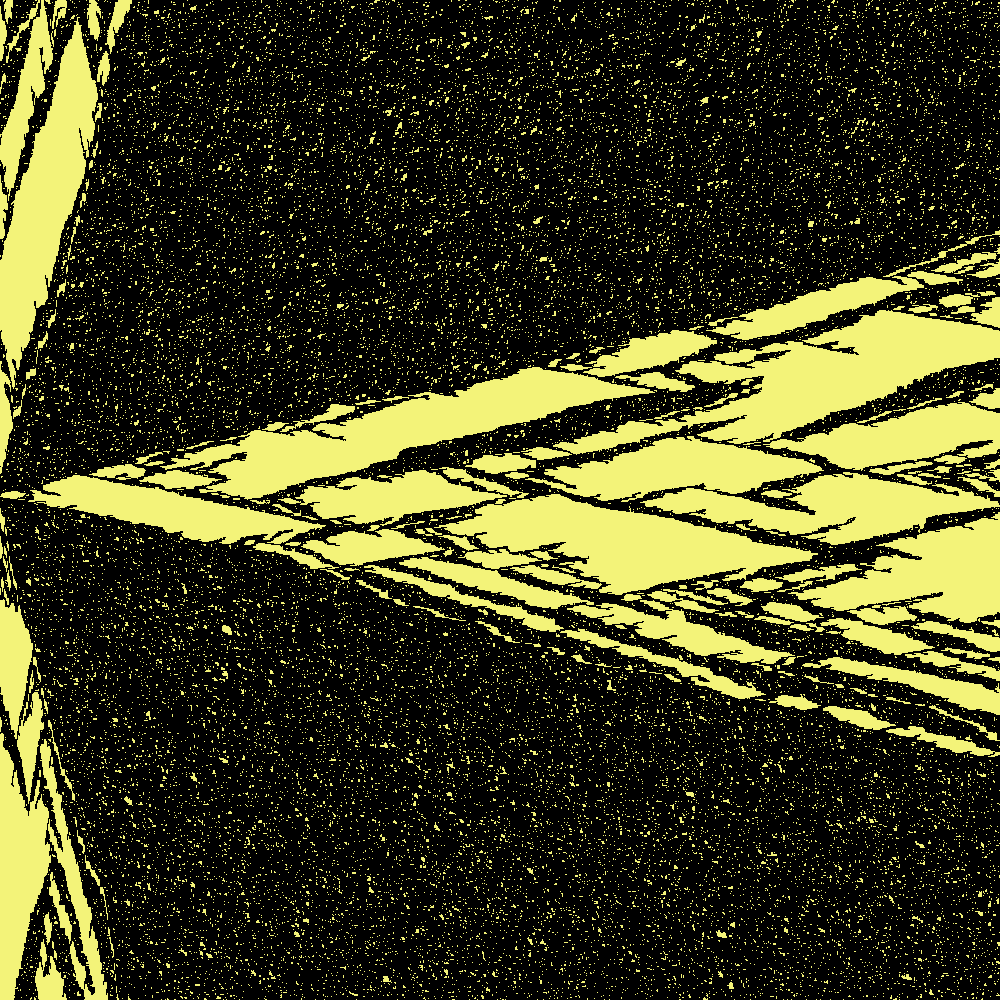}
      \hspace{\stretch{1}}
   \caption{Detailed view of Figure~\ref{fig:speed09d} for $p=0.2$ and $\lambda=0.9$ and $0.8$, respectively.}
   \label{fig:detspeed09d}
\end{figure}

Our theorem below establishes the existence of a coexistence phase. 
We note that here the phase for survival for $\eta^1$ is stronger than 
that shown in Theorem~\ref{thm:fpp}. Here we show that for some small enough $p$, $\eta^1$ survives for \emph{any} $\lambda<1$. The actual value of 
$\lambda$ plays a role only on determining whether coexistence happens. 
In the theorem below and its proof, a directed path in $\mathbb{Z}^d$ is defined to be a path whose jumps are only along the \emph{positive} 
direction of the coordinates.
\begin{theorem}\label{thm:discrete}
   For any $\lambda\in(0,1)$ and any $p\in(0,1-p_\crit^\dir)$, 
   where $p_\crit^\dir=p_\crit^\dir(\mathbb{Z}^d)$ denotes the critical probability for \emph{directed} site percolation in $\mathbb{Z}^d$, 
   we have 
   \begin{equation}
      \PR\lr{\eta^1\text{ produces an infinite cluster}}>0.
      \label{eq:discrete1}
   \end{equation}
   Furthermore, for any $\lambda\in(0,1)$, there exists a positive $p_0<1-p_\crit^\dir$ such that for any $p\in(p_0,1-p_\crit^\dir)$ we have
   \begin{equation}
      \PR\lr{\eta^1\text{ and }\eta^2\text{ both produce infinite clusters}}>0.
      \label{eq:discrete2}
   \end{equation}
\end{theorem}
\begin{proof}
   Consider a directed percolation process on $\mathbb{Z}^d$ where a vertex is declared to be open if it is not in $\eta^2(0)$, otherwise the vertex is closed.
   For any $t\geq0$, let $C_t$ be the vertices reachable from the origin by a directed path of length at most $t$ where all vertices
   in the path are open. 
   We will prove~\eqref{eq:discrete1} by showing that
   \begin{equation}
      \eta^1(t)\supseteq C_t \quad \text{for all $t\geq 0$}.
      \label{eq:induction}
   \end{equation}
   Let $x\in\eta^2(0)$ be the vertex of $\eta^2(0)$ that is the closest to the origin, in $\ell_1$ norm. 
   Clearly, for any time $t<\|x\|_1$, we have that 
   $\eta^1(t)$ has not yet interacted with $\eta^2(0)$, giving that $\eta^1(t) = \{y\in\mathbb{Z}^d \colon \|y\|_1 \leq t\}= C_t$.
   See Figure~\ref{fig:determ}(a) for an illustration.
   Then, at time $\|x\|_1$, $\eta^1$ tries to occupy all vertices at distance $\|x\|_1$ from the origin, leading to the configuration in 
   Figure~\ref{fig:determ}(b) and activating the seed $x$ of $\eta^2(0)$, illustrated in pink in the picture.
   Since $\eta^1$ is faster than $\eta^2$, $\eta^1$ is able to ``go around'' $x$, traversing the same path as in a directed percolation process.
   This leads to the configuration in~\ref{fig:determ}(c). Note that the same behavior occurs when $\eta^1$ finds a larger set of consecutive seeds of $\eta^2$
   at the same $\ell_1$ distance from the origin. For example, see what happens with the three red seeds in Figure~\ref{fig:determ}(d--f). 
   In this case, a directed percolation 
   process does not reach any vertex inside the red triangle in Figure~\ref{fig:determ}(f), as those vertices are shaded by the three red seeds.
   Since $\eta^2$ is slower than $\eta^1$, the cluster of $\eta^2$ that starts to grow when the three red seeds are activated
   cannot occupy any vertex outside of the red triangle. 
   \begin{figure}[htbp]
      \begin{center}
         \includegraphics[width=.22\textwidth]{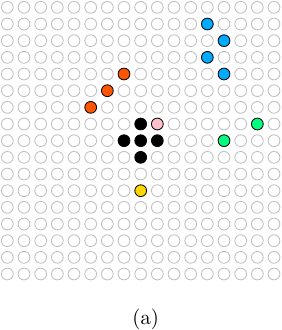}
         \hspace{\stretch{1}}
         \includegraphics[width=.22\textwidth]{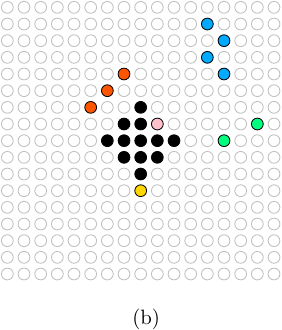}
         \hspace{\stretch{1}}
         \includegraphics[width=.22\textwidth]{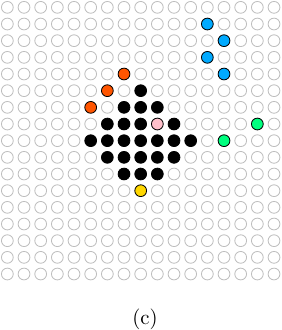}
         \hspace{\stretch{1}}
         \includegraphics[width=.22\textwidth]{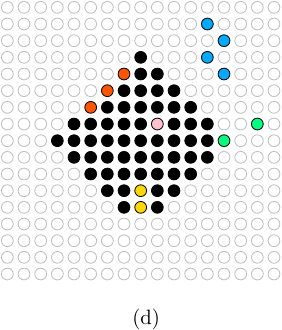}\\\vspace{2ex}
         \includegraphics[width=.22\textwidth]{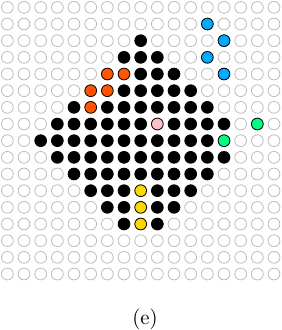}
         \hspace{\stretch{1}}
         \includegraphics[width=.22\textwidth]{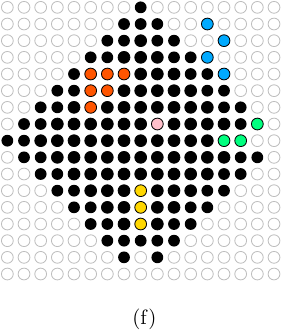}
         \hspace{\stretch{1}}
         \includegraphics[width=.22\textwidth]{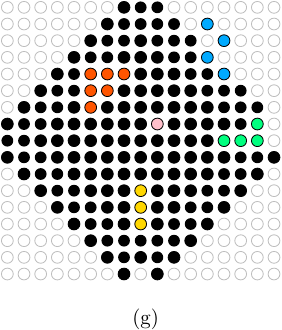}
         \hspace{\stretch{1}}
         \includegraphics[width=.22\textwidth]{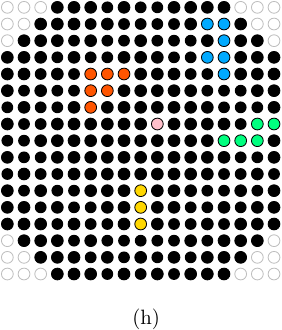}
      \end{center}\vspace{-.5cm}
      \caption{Evolution of \fpphe{} with deterministic passage times and $\lambda=0.59$. 
         Black vertices represent $\eta^1$ and white vertices represent unocuppied vertices. 
         All other colors represent clusters of $\eta^2$.}
      \label{fig:determ}
   \end{figure}
   
   A different situation occurs when $\eta^1$ finds a vertex of $\eta^2(0)$ in the axis, as with the yellow vertex of Figure~\ref{fig:determ}(c). 
   Note that, in a directed percolation process, all vertices below the yellow seed will not be reachable from the origin. 
   In our two-type process, something similar occurs, but only for a finite number of steps. 
   When $\eta^1$ activates the yellow seed at $x\in\eta^2(0)$, 
   $\eta^1$ cannot immediately go around $x$ as explained above. For $\lambda$ close enough to $1$, 
   $\eta^2$ occupies the successive vertex in the axis before 
   $\eta^1$ can go around $x$. This continues for some steps, with $\eta^2$ being able to grow along the axis; 
   see Figure~\ref{fig:determ}(d,e). 
   However, at each step $\eta^1$ will be $1-\lambda$ faster than $\eta^2$. 
   This will accumulate for roughly $\frac{1}{1-\lambda}$ steps, when $\eta^1$ will finally be 
   able to go around $\eta^2$; as in Figure~\ref{fig:determ}(f). 
   This happens unless $\eta^2(0)$ happens to have a seed at a vertex neighboring one of the vertices on the axis occupied by the growth of $\eta^2$. 
   This is illustrated by the green vertices of Figure~\ref{fig:determ}(f--h). When the first green vertex out of the axis is activated, $\eta^1$ will not 
   be able to occupy the vertex to the right of the green vertex, 
   and will encounter the next green seed before it can go around the first green seed found at the axis.
   The crucial fact to observe is that the clusters of $\eta^2$ that start to grow after the activation of each green seed 
   can only occupy vertices located to the right of the seeds, 
   and at the same vertical coordinate. This is a subset of the vertices that are shaded by the green seeds in a directed percolation process.
   Therefore,~\eqref{eq:discrete1} follows since the vertices occupied by $\eta^2$ are a subset of the following set: 
   take the union of all triangles obtained from
   sets of consecutive seeds away from the axis (as with the pink, red and blue seeds in Figure~\ref{fig:determ}), and take the union of semi-lines 
   starting at seeds located at the axis or at seeds neighboring semi-lines starting from seeds of smaller $\ell_1$ distance to the origin
   (as with the yellow and green seeds in Figure~\ref{fig:determ}). This set is exactly the set of vertices not reached by a directed path from the origin.
   
   Now we turn to~\eqref{eq:discrete2}. 
   First notice that, from the first part, we have that $\eta^1(t)\supseteq C_t$ for all $p$ and $\lambda$. 
   Since $C_t$ does not depend on $\lambda$, once we fix $p\in(0,1-p_\crit^\dir)$, we can take $\lambda$ as close to $1$ as we want, and $\eta^1$ will
   still produce an infinite component. Now we consider one of the axis. For example, the one containing the green vertices in Figure~\ref{fig:determ}. 
   Let $(x,0)$ be the first vertex occupied by $\eta^2$ in that axis.
   For each integer $k$, we will define $X_k$ as the smallest non-negative integer such that $(k,X_k)$ will be occupied by $\eta^1$. Similarly, 
   $Y_k$ is the smallest non-negative integer such that $(k,-Y_k)$ will be occupied by $\eta^1$. Now we analyze the evolution of $X_k$; the one of $Y_k$ will be 
   analogous. Assume that $X_1,X_2,\ldots,X_{k-1}=0$. Then, with probability at least $p$ we have that
   $X_{k+1}\geq 1$. When this happens, $\eta^1$ will need to do at least $\frac{1}{1-\lambda}$ steps before being able to occupy the axis again. 
   However, for each $s\geq 2$, the probability that $X_{k+s}> X_{k+1}$ is at least $p$. This gives that the probability that the random variable $X$
   reaches value above $1$ before going back to zero is at least $1-(1-p)^{\frac{1}{1-\lambda}}$. 
   Once we have fixed $p$, by setting $\lambda$ close enough to $1$ we 
   can make this probability very close to $1$. This gives that $X_k$ has a drift upwards. Since the downwards jumps of $X_k$ are of size at 
   most $1$, this implies that at some time $X_k$ will depart from $0$ and will never return to it. 
   A similar behavior happens for $Y_k$, establishing~\eqref{eq:discrete2}.
\end{proof}

\section{Preliminaries on first passage percolation}\label{sec:prelim}
Let $\Q$ be a probability distribution on $(0,\infty)$ with no atoms and with a finite exponential moment.
Consider a first passage percolation process $\{\xi(t)\}_t$, which starts from the origin and spreads according to $\Q$. More precisely, 
for each pair of neighboring vertices $x,y\in\mathbb{Z}^d$, let 
$\zeta_{x,y}$ be an independent random variables of distribution $\Q$. 
The value $\zeta_{x,y}$ is regarded as the time that $\xi$ needs to spread throughout the edge $(x,y)$. 
Note that $\zeta$ defines a random metric on $\mathbb{Z}^d$, where the distance between two vertices is the length of the shortest path between them, and the length of a path is 
the sum of the values of $\zeta$ over the edges of the path.
Hence, given any initial configuration $\xi(0)\subset\mathbb{Z}^d$, the set $\xi(t)$ comprises all vertices of $\mathbb{Z}^d$ that are within distance 
$t$ from $\xi(0)$ according to the metric $\zeta$. We assume throughout the paper that $d\geq 2$. 

For $X\subset\mathbb{Z}^d$, let $\PQ_X^\Q$ be the probability measure induced by the process $\xi$ when $\xi(0)=X$.
When the value of $\xi(0)$ is not important, we will simply write $\PQ^\Q$, and when $\Q$ is the exponential distribution of rate $1$, we write $\PQ$.

Let $\tilde\xi(t)\subset \mathbb{R}^d$ be defined by 
$\tilde\xi(t)=\bigcup_{x\in\xi(t)}\lr{x+[-1/2,1/2]^d}$; that is, $\tilde\xi(t)$ is obtained by adding a unit cube centered at each point of $\xi(t)$. 
A celebrated theorem of Richardson~\cite{Richardson}, extended by Cox and Durrett~\cite{CoxDurrett}, establishes that the rescaled set 
\begin{equation}
   \text{$\frac{\tilde \xi(t)}{t}$ converges as $t\to\infty$ to a deterministic set,
   which we denote by $\mathcal{B}_\Q\subset\mathbb{R}^d$.}
   \label{eq:shapethm}
\end{equation}
Such a result is now widely referred to as a \emph{shape theorem}, and $\mathcal{B}_\Q$ is referred to as the \emph{limit set}.
The set $\mathcal{B}_\Q$ defines a norm $|\cdot|_\Q$ on $\mathbb{R}^d$ via
$$
   |x|_\Q = \inf\big\{r\in\mathbb{R}_+ \colon x\in r\mathcal{B}_\Q\big\}, \quad x\in\mathbb{Z}^d.
$$
We abuse notation and define, for any $t\geq0$, $\ballQ{t}$ as the ball of radius $t$ according to the norm above: 
$\ballQ{t}=\lrc{x\in\mathbb{Z}^d \colon |x|_\Q\leq t}$.
As before, we drop the subscript $\Q$ when $\Q$ is the exponential distribution of rate $1$.

In~\cite[Theorem 2]{Kesten93}, Kesten derived upper bounds on the fluctuations of $\xi(t)$ around $\ballQ{t}$. 
We state Kesten's result in Proposition~\ref{pro:kesten} below, in a form that is more suitable to our use later\footnote{Actually we will only need large deviations results, similar to those derived already in~\cite{GK}. However, 
we will state and use the finer result of Kesten~\cite{Kesten93} as it controls deviations not only in a given direction but from the whole limit shape $\mathcal{B}_\Q$.}.
Before, we need to introduce some notation.
Given any set of positive values $\{\zeta'_{x,y}\}_{x,y}$ to the edges of the lattice, 
which we from now on refer to as \emph{passage times}, 
and given any two vertices $x,y \in \mathbb{Z}^d$, 
let 
\begin{equation}
   D(x,y; \zeta') \text{ be the distance between $x$ and $y$ according to the metric $\zeta'$}.
   \label{eq:defd}
\end{equation}
We extend this notion to subsets by writing 
$$
   D(X,Y;\zeta') = \inf_{x\in X, y\in Y}D(x,y;\zeta'), \quad \text{for any $X,Y\subset \mathbb{Z}^d$}.
$$
For two vertices $x,y\in\mathbb{Z}^d$ we use the notation 
$$
   x\sim y \text{ if $x$ and $y$ are neighbors in $\mathbb{Z}^d$}.
$$ 
Furthermore, for any set $A\subset\mathbb{Z}^d$, define 
\begin{align*}
   \text{the \emph{inner boundary} of $A$ by } \bddi A=\{x\in A \colon \exists y\in\mathbb{Z}^d\setminus A \text{ and }x\sim y\},\\
   \text{the \emph{outer boundary} of $A$ by }
   \bddo A=\{x\in \mathbb{Z}^d\setminus A \colon \exists y\in A \text{ and }x\sim y\},\\
   \text{and the \emph{edge boundary} of $A$ by }
   \bdde A=\{(x,y)\in \mathbb{Z}^d\times\mathbb{Z}^d, x\sim y \colon x\in A \text{ and } y\not \in A\}.
\end{align*}
Given a set $A\subset\mathbb{Z}^d$, we say that an event is measurable with respect to passage times inside $A$ if the event is measurable with respect 
to the passage times of the edges whose both endpoints are in $A$.

For any $t>0$, any $\delta\in(0,1)$, and any set of passage times $\zeta'$, define the event 
$$
   S_t^\delta(\zeta') = \lrc{\inf_{x\in\bddi \ballQ{(1+\delta)t}}D(0,x; \zeta' )\leq t} \cup \lrc{\sup_{x\in\bddo\ballQ{(1-\delta)t}} D(0,x; \zeta')\geq t}.
$$
Disregarding some discrepancies in the choice of the boundary, $S_t^\delta(\zeta')$ is the event that $\xi(t)$ is either not contained in $\ballQ{(1+\delta)t}$ or does not contain 
$\ballQ{(1-\delta)t}$.
\begin{proposition}\label{pro:kesten}
   Let $\Q$ be a probability distribution on $(0,\infty)$, with no atoms, and with a finite exponential moment.
   There exist constants $c_1,c_2,c_3>0$ depending on $d$ and $\Q$ such that, for all $t\geq1$ and all $\delta > c_1 t^{-\frac{1}{2d+4}}(\log t)^\frac{1}{d+2}$, 
   \begin{equation}
      \PQ^\Q\lr{S_t^\delta(\zeta)} \leq c_2 \exp\lr{-c_3 t^\frac{d+1}{2d+4}}.
      \label{eq:kestenbound}
   \end{equation}
   Moreover, we have that 
   \begin{equation}
      S_t^\delta(\zeta') \text{ is measurable with respect to the passage times $\lrc{\zeta'_{x,y} \colon x\sim y \text{ and } x, y\in \ball{(1+\delta)t}}$}.
      \label{eq:measurability}
   \end{equation}
\end{proposition}
\begin{proof}
   First we establish~\eqref{eq:measurability}. Note that 
   the event $\lrc{\inf_{x\in \bddi\ballQ{(1+\delta)t}}D(0,x; \zeta)\leq t}$ is measurable with respect to the passage times inside $\ballQ{(1+\delta)t}$. 
   Then, if this event does not hold, that is under $\lrc{\inf_{x\in \bddi\ballQ{(1+\delta)t}}D(0,x; \zeta)>t}$, 
   the event $\lrc{\sup_{x\in \bddo\ballQ{(1-\delta)t}}D(0,x; \zeta)\leq t}$ is also measurable with 
   respect to the passage times inside $\ballQ{(1+\delta)t}$, establishing~\eqref{eq:measurability}. 
   The bound in~\eqref{eq:kestenbound} follows directly from Kesten's result~\cite[Theorem 2]{Kesten93}.
\end{proof}

\section{Encapsulation of competing first passage percolation}\label{sec:hp}
Here we consider two first passage percolation processes that compete for space as they grow through $\mathbb{Z}^d$. 
One of the processes spreads throughout $\mathbb{Z}^d$ at rate $1$, while the other spreads according to a distribution $\Q$ such that its limit shape is contained in $\ball{\lambda}$ for some 
$\lambda<1$, with $\lambda$ being a 
parameter of the system. 
We will say that $\lambda$ is the \emph{rate of spread} of the second process.
We assume that the starting configuration of each process comprises only a finite set of vertices. In this case, 
one expects that both processes cannot simultanenously grow indefinitely; that is, one of the processes will eventually surround the other, 
confining it to a finite subset of $\mathbb{Z}^d$.
This was studied by H{\"a}ggstr{\"o}m and Pemantle~\cite{HP2000}.
In the proof of our main result, we will employ a refined version of a result in their paper. In particular, 
we will give a lower bound on the probability that the faster 
process surrounds the slower one within some fixed time. 

First we define the processes precisely. 
Let $\xi^1$ denote the faster process so that, for each time $t\geq0$, $\xi^1(t)$ gives the set of vertices occupied by the faster process at time $t$.
Similarly, let $\xi^2$ denote the slower process.
For each neighbors $x,y \in \mathbb{Z}^d$, 
let $\zeta^1_{x,y}$ be an independent exponential random variables of rate $1$, and let $\zeta^2_{x,y}$ be an independent random variable of distribution $\Q$. 
For $i\in\{1,2\}$, $\zeta^i_{x,y}$ represents the passage time of process $\xi^i$ through the edge $(x,y)$.

The processes start at disjoint sets $\xi^1(0),\xi^2(0)\subset \mathbb{Z}^d$. Then they spread throughout $\mathbb{Z}^d$ according to the passage times 
$\zeta^1$ and $\zeta^2$ with the constraint that, whenever a vertex is occupied by either $\xi^1$ or $\xi^2$, the other process cannot occupy that vertex afterwards.
Therefore, for any $t\geq0$, we obtain that $\xi^1(t)$ and $\xi^2(t)$ are disjoint sets.
To define $\xi^1,\xi^2$ more precisely, we will iteratively set $s^k(x)$, for each $x\in\mathbb{Z}^d$ and $k\in\{1,2\}$, so that at the end $s^k(x)$ is the time $x$ is occupied by process $k$, or 
$s^k(x)=\infty$ if $x$ is not occupied by 
process $k$.
Start setting $s^1(x)=0$ for all $x\in\xi^1(0)$, 
$s^2(x)=0$ for all $x\in\xi^2(0)$, 
and $s^k(x)=\infty$ for all $k\in\{1,2\}$ and $x\not\in\xi^k(0)$. 
Then, choose the value of $k\in\{1,2\}$ and the pair of neighboring vertices $x,y$ with $s^k(x)<\infty$ and $s^k(y)=\infty$ that minimizes $s^k(x)+\zeta^k_{x,y}$, and set $s^k(y)=s^k(x)+\zeta^k_{x,y}$.
Then,
$$
   \xi^1(t) = \lrc{x\in\mathbb{Z}^d \colon s^1(x)\leq t}
\quad \text{and} \quad
   \xi^2(t) = \lrc{x\in\mathbb{Z}^d \colon s^2(x)\leq t}.
$$ 
Given two sets $X_1,X_2,\subset \mathbb{Z}^d$, let $\PQ_{X_1,X_2}^\Q$ denote the probability measure induced by the processes $\xi^1,\xi^2$ with initial 
configurations $\xi^1(0)=X_1$ and $\xi^2(0)=X_2$. 

The proposition below is a more refined version of a result of H\"aggstr\"om and Pemantle~\cite[Proposition~2.2]{HP2000}. 
It establishes that if $\xi^2$ starts from inside $\ball{r}$ for some $r\in\mathbb{R}_+$, and $\xi^1$ starts from a single vertex outside of a larger ball $\ball{\alpha r}$, for some $\alpha>1$, then 
there is initially a large separation between $\xi^1$ and $\xi^2$, allowing $\xi^1$ to surround $\xi^2$ with high probability. 
Moreover, we obtain that $\xi^1$ will eventually confine $\xi^2$ to some set $\ball{R}$ for some given $R$, and the probability that this happens goes to $1$ with $\alpha$.
We need to state this result in a high level of detail, as we will apply it at various scales later in our proofs.
We say that an event is \emph{increasing} (resp., \emph{decreasing}) with respect to some passage times $\zeta$ if whenever the event holds for 
$\zeta$ it also holds for any passage times $\zeta'$ that satisfies 
$\zeta'_{x,y}\geq \zeta_{x,y}$ (resp., $\zeta'_{x,y}\leq \zeta_{x,y}$) for all neighboring $x,y\in\mathbb{Z}^d$.
\begin{proposition}\label{pro:encapsulate}
   There exist positive constants $c_1,c_2$ depending only on $d$ so that, for any $\lambda\in(0,1)$, any $r>1$, 
   and any $\alpha>\left(\frac{1}{\lambda(1-\lambda)}\right)^{c_1}$, if $\Q$ is such that $\mathcal{B}_\Q \subset \ball{\lambda}$,
   we can define deterministic values $R$ and $T=T(R)$ satisfying $R\leq \alpha r \exp\lr{\frac{c_1}{1-\lambda}}$ and  
   $T\leq R\lr{\frac{11-\lambda}{10}}^2$ such that the following holds. 
   For any $x\in\bddo \ball{\alpha r}$,
   there is an event $F$ that is 
   measurable with respect to the passage times $\zeta^1,\zeta^2$ inside $\ball{R \lr{\frac{11-\lambda}{10}}^2}$ 
   and is increasing with respect to $\zeta^2$ and 
   decreasing with respect to $\zeta^1$,
   whose occurrence implies 
   $$
      \xi^1(T) \supset \ball{R}\setminus \ball{\frac{10R}{11-\lambda}},
   $$
   and whose probability of occurrence satisfies
   $$
      \PQ_{x,\ball{r}}^\Q\lr{F}
      \geq 1-\exp\lr{-c_2 (\lambda(1-\lambda)\alpha r)^{\frac{d+1}{2d+4}}}.
   $$
   In particular, when $F$ occurs, within time $T$, $\xi^1$ encapsulates $\xi^2$ inside $\ball{R}$.
\end{proposition}
We defer the proof of the proposition above to Appendix~\ref{sec:HPproof}.
The proof will follow along the lines of~\cite[Proposition~2.2]{HP2000}, but we need to perform some steps with more care, as 
we need to obtain bounds on the probability 
that $F$ occurs, to establish bounds on $R$ and $T$, 
to derive that $F$ is increasing with respect to $\eta^2$ and decreasing with respect to $\zeta^1$, and to obtain the measurability 
constraints on $F$. 

We will need to apply the above proposition in a more complex setting.
For this, it is important to keep in mind the process \fpphe{} defined in Section~\ref{sec:intro}, where a cluster of type~2 starts spreading from each type~2 seed when that seed is activated, and type~2 seeds are initially 
distributed in $\mathbb{Z}^d$ as a product measure. 
We will apply the encapsulation procedure of Proposition~\ref{pro:encapsulate} for each different cluster of type~2 growing out of its seed. This means that we will apply 
Proposition~\ref{pro:encapsulate} at several scales (that is, with different values of $r$) 
and at several places of $\mathbb{Z}^d$. 
The encapsulation happening in one place may end up interfering in the spread of type~1 and type~2 in the
other places. 

In order to have a version of Proposition~\ref{pro:encapsulate} that can handle this situation, we will focus in one such encapsulation. 
For that encapsulation, we represent type~1 as $\xi^1$, and assume that $\xi^1(0)$ contains at least one vertex from $\bddo\ball{\alpha r}$.
For the cluster of type 2 whose encapsulation we are considering, we let it start from $\xi^2(0)\subseteq \ball{r}$. 
Here $\xi^2$ will only represent the cluster of type~2 that spreads from $\xi^2(0)$. 
For the other clusters of type~2, we will not refer to them as $\xi^2$ but simply as type~2. 

To model the spread of the other clusters of type~2, we introduce a positive number $\gamma$ and two sequences of simply connected subsets $(\Pi_\iota)_\iota$ and $(\Pi'_\iota)_\iota$ of $\mathbb{Z}^d$, 
such that 
\begin{equation}
   \text{the sets $\Pi_\iota$ are all disjoint}
   \label{eq:pi1}
\end{equation}
and, for each $\iota\geq1$, we have
\begin{equation}
\begin{aligned}
   \Pi_\iota'\subset\Pi_\iota \subset y_\iota+ \ball{\gamma r}
   \text{ for some }y_\iota\in\mathbb{Z}^d,
   \text{ and moreover, }\\
   \Pi_\iota\setminus \Pi'_\iota 
   \text{ is simply connected and contains $\bddi \Pi_\iota$}. 
   \label{eq:pi2}
\end{aligned}
\end{equation}
The sets $\Pi_\iota$ represent the other regions of $\mathbb{Z}^d$ (of smaller scale) where the encapsulation of a cluster of type-2 of \fpphe{} may be happening while $\xi^1,\xi^2$ spread from $\xi^1(0),\xi^2(0)$, 
whereas the sets $\Pi_\iota'$ represent the regions inside which each type~2 cluster gets confined to.
(The value of $\gamma$ will be quite small, so that all $\Pi_\iota$ are of scale smaller than $r$ because clusters of scale larger than $r$ will be treated afterwards: 
in the proof we will consider the clusters essentially in order of their sizes.)
Outside of the set $\bigcup_\iota \Pi_\iota$, the spread of $\xi^1,\xi^2$ will follow the passage times $\zeta^1,\zeta^2$, respectively. 
However, the spread of $\xi^1,\xi^2$ inside $\bigcup_\iota \Pi_\iota$ may be different and quite complicated. 
We will not need to specify this precisely, we will only require the following properties:
\begin{enumerate}[label=(P\arabic*)]
    \item\label{it:pi1} For each $\iota$, if $\xi^2$ does not enter $\Pi_\iota$ from $\xi^2(0)$, then $\Pi_\iota\setminus \Pi_\iota'$ becomes entirely occupied by $\xi^1$. 
    \item\label{it:pi2} For any $\iota,x\in\bddi\Pi_\iota,$ and $y\in\Pi_\iota$, either the time that $\xi^1$ takes to spread from $x$ to $y$ within $\Pi_\iota$ is smaller than 
	   that given by the passage times $\zeta^1$, or $y$ is occupied by type 2.
	\item\label{it:pi3} For any $\iota,x\in\bddi\Pi_\iota,$ and $y\in\Pi_\iota$, either the time that $\xi^2$ takes to spread from $x$ to $y$ within $\Pi_\iota$ is larger than 
	   that given by the passage times $\zeta^2$, or $y$ is occupied by type 1.
\end{enumerate}
Above when we refer to the time given by the passage times $\zeta^k$, $k\in\{1,2\}$, we mean the time given by the passage times $\zeta^k$ when we completely ignore the presence of the
cluster of type~2 that grows from within $\Pi_\iota$.
Regarding properties~\ref{it:pi2} and~\ref{it:pi3} above, in our application of the proposition below, 
we will do some scaling of the passage times so that, within each $\Pi_\iota$, type~1 will actually spread at a rate faster than 1 while type~2 will spread at a rate slower than $\lambda$. 
Since within $\Pi_\iota$ type~1 needs to do a detour around the growing cluster of type~2, 
we will use a coupling argument to say that even with this detour type~1 will spread inside $\Pi_\iota$ faster than the passage times given by $\zeta^1$.
Similarly, within $\Pi_\iota$ type~2 may benefit from $\xi^2$ entering from outside and blocking type~1 as it attempts to encapsulate type~2. 
We will use a coupling argument to say that, even with the help from $\xi^2$, type~2 will spread inside $\Pi_\iota$ slower than the passage times given by $\zeta^2$.
This will become clearer in Section~\ref{sec:overview}, where we present a high-level description of the whole proof.
At this point, we do not need much detail of how the spread of type~1 and type~2 happen inside each $\Pi_\iota$. 

The goal of the proposition below, which is a refinement of Proposition~\ref{pro:encapsulate}, is to argue that with high probability the passage times $\zeta^1,\zeta^2$ are such that 
$\xi^1$ encapsulates $\xi^2$ inside a ball surrounding $\ball{r}$ unless there exists a set $\Pi_\iota$ that does not satisfy one of the properties \ref{it:pi1}--\ref{it:pi3}.
%
Given three sets $S_1,S_2,S_3\subset \mathbb{Z}^d$, we say that $S_1$ separates $S_2$ from $S_3$ if any path in $\mathbb{Z}^d$ from $S_2$ to $S_3$ intersects $S_1$.
\begin{proposition}\label{pro:encapsulate2}
   There exist positive constants $c_1,c_2,c_3$ depending only on $d$ so that, for any $\lambda\in(0,1)$, any $r>1$, 
   and any $\alpha>\left(\frac{1}{\lambda(1-\lambda)}\right)^{c_1}$, if $\Q$ satisfies $\mathcal{B}_\Q\subset \mathcal{B}(\lambda)$, 
   we can define deterministic values $R$ and $T=T(R)$ satisfying $R\leq \alpha r \exp\lr{\frac{c_1}{1-\lambda}}$ and  
   $T\leq R\lr{\frac{11-\lambda}{10}}^2$ such that the following holds. 
   For any $\gamma \leq c_3 \lambda(1-\lambda)\alpha$ and  
   any $x\in\bddo \ball{\alpha r}$,
   there is an event $F$ that is 
   measurable with respect to the passage times $\zeta^1,\zeta^2$ inside $\ball{R \lr{\frac{11-\lambda}{10}}^2}$ 
   and is increasing with respect to $\zeta^2$ and 
   decreasing with respect to $\zeta^1$, such that its occurrence implies that either
   $$
      \xi^1(T) \text{ separates } \bddi\ball{R} \text{ from } \ball{\frac{10R}{11-\lambda}}
   $$
   or there exist sets $\{\Pi_\iota\}_\iota,\{\Pi_\iota'\}_\iota$ satisfying~\eqref{eq:pi1} and~\eqref{eq:pi2}, 
   such that there exists $\iota$ for which $\Pi_\iota \cap \ball{R\left(\frac{11-\lambda}{10}\right)^2}\neq \emptyset$ and 
   $\Pi_\iota$ does not satisfy at least one of the 
   properties~\ref{it:pi1}--\ref{it:pi3}. 
   Moreover, we obtain that 
   $$
      \PQ_{x,\ball{r}}^\Q\lr{F}
      \geq 
      1-\exp\lr{-c_2 (\lambda(1-\lambda)\alpha r)^{\frac{d+1}{2d+4}}}.
   $$
\end{proposition}

\section{Proof of Theorem~\ref{thm:fpp}}\label{sec:proof}
Theorem~\ref{thm:fpp} will follow directly from Theorem~\ref{thm:fpp2} below, which we will prove in this section.
The proof is quite long, so we start with an overview. For clarity's sake, we discuss the proof overview under the setting of 
Theorem~\ref{thm:fpp}, and only state Theorem~\ref{thm:fpp2} in Section~\ref{sec:restate}.

\subsection{Proof overview}\label{sec:overview}
\begin{figure}[htbp]
   \begin{center}
      \includegraphics[width=.18\textwidth]{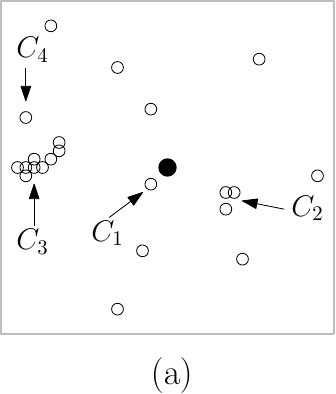}
      \hspace{\stretch{1}}
      \includegraphics[width=.18\textwidth]{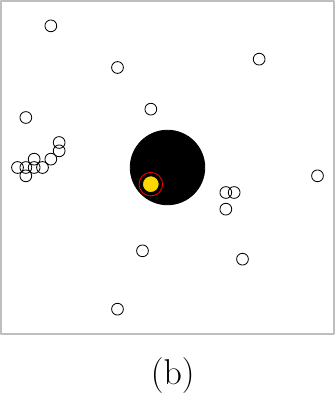}
      \hspace{\stretch{1}}
      \includegraphics[width=.18\textwidth]{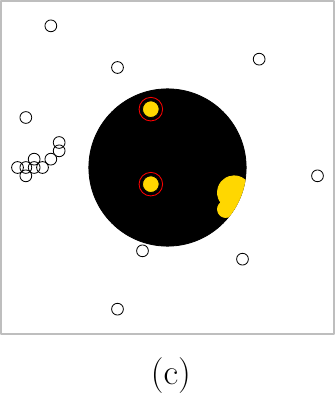}
      \hspace{\stretch{1}}
      \includegraphics[width=.18\textwidth]{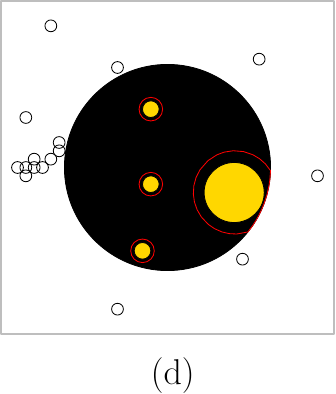}
      \hspace{\stretch{1}}
      \includegraphics[width=.18\textwidth]{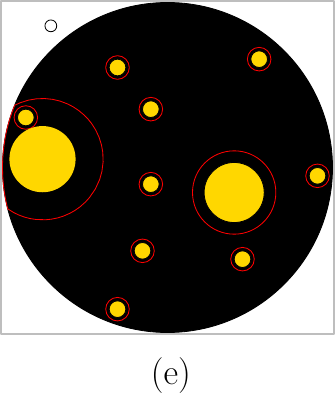}
   \end{center}\vspace{-.5cm}
   \caption{Illustration of the proof strategy of Theorem~\ref{thm:fpp}, 
      with the application of the encapsulation procedure of H{\"a}ggstr{\"o}m and Pemantle~\cite{HP2000} (cf.\ Proposition~\ref{pro:encapsulate}). 
      White balls indicate seeds of
      $\eta^2$ that were not yet activated, 
      black represents the growth of $\eta^1$, and  
      yellow balls represent the regions inside which the activated clusters of $\eta^2$ got trapped. 
      The red circle around each yellow ball corresponds to a larger area in which the passage times need to be observed in order to ensure the success of the encapsulation of the 
      $\eta^2$ cluster; for example, only to ensure the measurability requirement of the event in Proposition~\ref{pro:encapsulate} we need to observe the passage times in 
      $\ball{R\left(\tfrac{11-\lambda}{10}\right)^2}$, whereas the $\eta^2$ cluster is trapped within $\ball{\tfrac{10R}{11-\lambda}}$.}
   \label{fig:poverview}
\end{figure}
We start with a high-level overview of the proof. 
Below we refer to Figure~\ref{fig:poverview}.
Since $p$ is small enough, initially $\eta^1$ will grow without finding any seed of $\eta^2(0)$, as in Figure~\ref{fig:poverview}(a). 
When $\eta^1$ activates a seed of $\eta^2$, then we will apply Proposition~\ref{pro:encapsulate} to 
establish that $\eta^1$ will go around $\eta^2$, encapsulating it inside a small ball (according to the norm $|\cdot|$). 
This is illustrated by the encapsulation of $C_1$
in Figure~\ref{fig:poverview}(b).
The yellow ball in the picture marks the region inside which the cluster of $\eta^2$ will be trapped; 
in Proposition~\ref{pro:encapsulate} this corresponds to the ball $\ball{\frac{10R}{11-\lambda}}$.
To ensure the encapsulation of a cluster of $\eta^2$, we will need to observe the passage times inside a larger ball; for example,
only to ensure the measurability requirement of the event in Proposition~\ref{pro:encapsulate} we need to observe the passage times in 
$\ball{R\left(\frac{11-\lambda}{10}\right)^2}$. This larger ball is represented by the red circles in Figure~\ref{fig:poverview}.
As long as different red circles do not intersect one another, the encapsulation of different clusters of $\eta^2$ will happen
independently. However, when $\eta^1$ encounters a large cluster of $\eta^2$ seeds, as 
it happens with the cluster $C_3$ in Figure~\ref{fig:poverview}(d), 
the encapsulation procedure will require a larger region to succeed. We will carry this out by developing a multi-scale analysis of 
the encapsulation procedure, where the size of the region will depend, among other things, on the size of the clusters of $\eta^2(0)$. 
After the encapsulation
takes place, as in Figure~\ref{fig:poverview}(e), we are left with a larger yellow ball and a larger red circle. 
Also, whenever two clusters of $\eta^2(0)$ are close enough 
such that their corresponding red circles intersect, as it happens with $C_2$ in Figure~\ref{fig:poverview}(c), then the encapsulation cannot be guaranteed to succeed.
In this case, we see these clusters as if they were a larger cluster, and perform the encapsulation procedure over a slightly larger region, as 
in Figure~\ref{fig:poverview}(c,d). 

There is one caveat in the above description. Suppose $\eta^1$ encounters a very large cluster of $\eta^2$, for example $C_3$ in 
Figure~\ref{fig:poverview}(d).
It is likely that during the encapsulation of 
$C_3$, inside the red circle of this encapsulation, we will find smaller clusters of $\eta^2$. 
This happens in Figure~\ref{fig:poverview}(d) with $C_4$.
This does not pose a big problem, 
since as long as the red circle of the encapsulation of the small clusters do not intersect one another and do not intersect the 
yellow ball produced by the encapsulation of $C_3$, 
the encapsulation of $C_3$ will succeed. 
This is illustrated in Figure~\ref{fig:poverview}(e), where the encapsulation of $C_4$ happened inside the encapsulation of $C_3$.
There is yet a subtlety. During the encapsulation of $C_4$, 
the advance of $\eta^1$ is slowed down, 
as it needs to make a detour around the growing cluster of $C_4$. This slowing down could cause the encapsulation of $C_3$ to fail. 
Similarly, as $\eta^2$ spreads from $C_3$, 
$\eta^2$ may find vertices that have already been occupied by $\eta^2$ due to the spread of $\eta^2$ from other non-encapsulated seeds. 
This would happen, for example, if the yellow ball that grows from $C_3$ were to intersect the yellow ball that grows from $C_4$.
If this happens before the encapsulation of $C_4$ ends, then the spread of $C_3$ gets a small advantage. 
The area occupied by the spread of $\eta^2$ from $C_4$ can in this case be  
regarded as being absorbed by the spread of $\eta^2$ from $C_3$, causing $C_3$ to spread faster than if $C_4$ were not present.
We will need to show that 
$\eta^1$ is not slowed down too much by possible detours around smaller clusters, 
and $\eta^2$ is not sped up too much by the absorption of smaller clusters.


To do this, we will define a sequence of scales $R_1,R_2,\ldots$, with $R_k$ increasing with $k$.
The value of $R_k$ represents the radius of the region inside which encapsulation takes place at scale $k$. (Later when making this argument rigorous, for each scale $k$ we will need to introduce several radii, 
but to simplify the discussion here we 
can think for the moment that $R_k$ gives the radius of the red circles in Figures~\ref{fig:poverview},
and that the radius of the yellow circles at scale $k$ is just a constant times $R_k$.) 
The larger the cluster of seeds of $\eta^2$, the larger $k$ must be. 
We will treat the scales in order, starting from scale $1$. 
This procedure is illustrated in Figure~\ref{fig:proof} for the encapsulation of the configuration in Figure~\ref{fig:poverview}(a).
Once all clusters of scale $k-1$ or below have been treated, we look at all remaining 
(untreated) clusters that are not too big to be encapsulated at scale $k$. 
If two clusters of scale $k$ are too close to each other, so that their corresponding red circles intersect, we will not carry out the encapsulation and 
will treat these clusters as if they were one cluster from a larger scale, as illustrated in Figure~\ref{fig:proof}(a). 
After disregarding these, all remaining clusters of scale $k$ are disjoint
and can be treated independently using the more refined Proposition~\ref{pro:encapsulate2}. The $\Pi_\iota$ will be the clusters of scale smaller than $k$ that happen to fall inside the red circle of the cluster of scale $k$. 
Although small and going very fast to zero with $k$, 
the probability that the encapsulation procedure fails is still positive. So it will happen that some encapsulation will fail, 
as illustrated by the vertex at the top of Figure~\ref{fig:proof}(a). 
If this happens for some 
cluster of scale $k$, which is an event measurable with respect to the passage times inside a red circle of scale $k$ containing the $\eta^2$ seeds of that cluster, 
we then take the whole area inside this red circle 
and consider it as a larger cluster of $\eta^2(0)$ seeds, leaving it to be treated at a larger scale, as 
in Figure~\ref{fig:proof}(b). Then we turn to the next scale, as in Figure~\ref{fig:proof}(c,d).
\begin{figure}[htbp]
   \begin{center}
      \includegraphics[width=.22\textwidth]{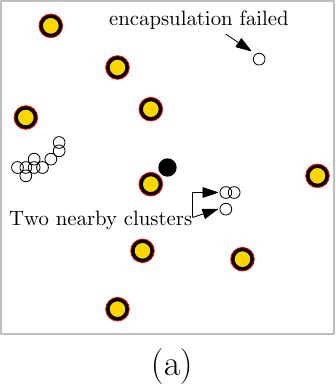}
      \hspace{\stretch{1}}
      \includegraphics[width=.22\textwidth]{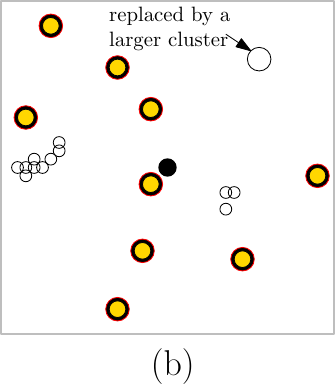}
      \hspace{\stretch{1}}
      \includegraphics[width=.22\textwidth]{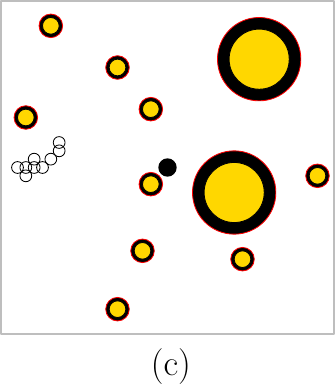}
      \hspace{\stretch{1}}
      \includegraphics[width=.22\textwidth]{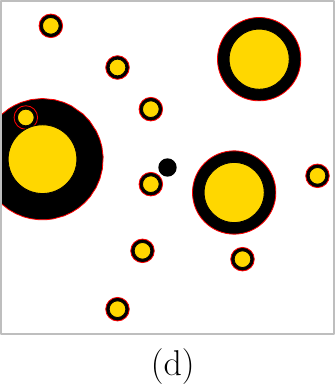}
   \end{center}\vspace{-.5cm}
   \caption{Illustration of the multi-scale encapsulation procedure. 
      (a) The encapsulation of all small clusters are analyzed, skipping clusters that are near other clusters.
      (b) Clusters whose encapsulation in the previous step failed are replaced by larger clusters. (c,d) Clusters of the next scale, or bundles of small clusters 
      that occupy a region of the size of the next scale, are evaluated and so on.}
   \label{fig:proof}
\end{figure}

In order to handle the slow down of $\eta^1$ due to detours imposed by smaller scales, and the sped up of $\eta^2$ due to absorption of smaller scales, we will 
introduce a decreasing sequence of positive numbers $\epsilon_1,\epsilon_2,\ldots$, as follows. In the encapsulation of a cluster $C$ of scale $k$, 
we will show not only that $\eta^1$ is able to encapsulate $C$, but also that $\eta^1$ does that sufficiently fast. We do this by coupling
the spread of $\eta^1$ inside the red circle of $C$ with a slower first passage percolation process of rate $\prod_{i=1}^ke^{-\epsilon_i}$ that evolves 
independently of $\eta^2$. In other words, this slower first passage percolation process does not need to do a detour around $C$, but pay the price 
by having slower passage times. We show that the spread of $\eta^1$ around $C$ is faster than that of this slower first passage percolation process.  
Similarly, we show that, even after absorbing smaller scales, $\eta^2$ still spreads slow enough inside the red circle of $C$,
so that we can couple it with a faster first passage percolation process of rate 
$\lambda\prod_{i=1}^ke^{\epsilon_i}$, which evolves independently of everything else. 
We show using this coupling that the spread of $\eta^2$ is slower than that of the faster first passage percolation
process.
Thus at scale $k$, $\eta^1$ is spreading at rate at least $\prod_{i=1}^ke^{-\epsilon_i}$ while $\eta^2$ is spreading at rate at most
$\lambda\prod_{i=1}^ke^{\epsilon_i}$, regardless of what happened at smaller scales. By adequately setting $\epsilon_k$, we can ensure that 
$\prod_{i=1}^ke^{-\epsilon_k}> \lambda \prod_{i=1}^ke^{\epsilon_k}$ for all $k$, allowing us to apply Proposition~\ref{pro:encapsulate2} at all scales.

The final ingredient is to develop a systematic way to argue that $\eta^1$ produces an infinite cluster. 
For this we introduce two types of regions, which we call \emph{contagious} and \emph{infected}. 
We start at scale $1$, where all vertices of $\eta^2(0)$ are contagious. 
Using the configuration in Figure~\ref{fig:poverview}(a) as an example, all white balls there are contagious.
The contagious vertices that do not belong to large clusters or are not close to other contagious
vertices, are treated at scale $1$. The other contagious vertices remain contagious for scale $2$. 
Then, for each cluster treated at scale $1$, either
the encapsulation procedure is successful or not. 
If it is successful, then the yellow balls produced by the encapsulation of these clusters are declared infected, and 
the vertices in these clusters are removed from the set of contagious vertices. In Figure~\ref{fig:proof}(b), the yellow area represents the infected vertices
after clusters of scale $1$ have been treated.
Recall that when an encapsulation is successful, all vertices reached by $\eta^2$ from that cluster must be contained inside the yellow area. 
On the other hand, if the encapsulation is not successful, then all vertices inside the red circle become contagious and go to scale 2, together with the 
other preselected vertices. 
An example of this situation is given by the cluster at the top-right corner of 
Figure~\ref{fig:proof}(b).
We carry out this procedure iteratively until there are no more contagious vertices or the origin has been disconnected from infinity 
by infected vertices. The proof is concluded by showing that $\eta^2$ is confined to the set of infected vertices, 
and that with positive probability the infected vertices will not
disconnect the origin from infinity.

{\bf Roadmap of the proof}. 
We now proceed to the details of the proof. We split the proof in few sections. 
In Section~\ref{sec:restate}, we state Theorem~\ref{thm:fpp2}, the more general version of Theorem~\ref{thm:fpp}.
Then in Section~\ref{sec:multiscale} we 
set up the multi-scale analysis, specifying the 
sizes of the scales and some parameters. This will define boxes of multiple scales, and we will classify boxes as being either \emph{good}
or \emph{bad}. Roughly speaking, a box will be good if the encapsulation procedure inside the box is successful. The concrete definition of good boxes
is done in Section~\ref{sec:good}. In Section~\ref{sec:prgood} we estimate the probability that a box is good, independent of what happens outside the 
box. We then introduce contagious and infected sets in Section~\ref{sec:ci}, and show that $\eta^2$ is confined to the set of infected vertices.
At this point, it remains to show that the set of infected vertices does not disconnect the origin from infinity. For this, we need to control the set of 
contagious vertices, which can actually grow as we move to larger scales (for example, this happens when some encapsulation procedure fails). 
The event that a vertex is contagious at some scale $k$ depends on what happens at previous scale. We estimate the probability of such event by
establishing a recursion over scales, which we carry out in Section~\ref{sec:recursion}. With this we have a way to control whether a vertex is infected. 
In order to show that the origin is not disconnected from infinity by infected vertices, we apply the first moment method. We sum, over all contours around the origin,
the probability that this contour contains only infected vertices. Since infected vertices can arise at any scale, we need to look at multi-scale paths and contours of 
infected vertices, which we do in Section~\ref{sec:multiscalepaths}. We then put all ingredients together and complete the proof of Theorem~\ref{thm:fpp}
in Section~\ref{sec:completing}.

\subsection{General version of Theorem~\ref{thm:fpp}}\label{sec:restate}
In this section we will consider a generalization of \fpphe, where the passage times of $\eta^2$ can be given by any distribution, while the passage times of $\eta^1$ are exponential random variables of rate $1$.

Let $\Q$ be a probability distribution on $(0,\infty)$, with no atoms, and such it has a finite exponential moment.
It holds by~\cite[Theorem~2.16]{ADH} that a first passage percolation with passage times given by i.i.d.\ random variables with distribution $\Q$ has a limit shape 
$\mathcal{B}_\Q$, as in~\eqref{eq:shapethm}.
Recall that $\ball{r}=r \mathcal{B}$ denotes the ball of radius $r$ according to the norm induced by the shape theorem of first passage percolation with passage times that are 
exponential random variables of rate 1.

For any edge $(x,y)$ of the lattice, 
let $\zeta^1_{x,y}$ be an independent exponential random variable of rate 1, and let $\zeta^2_{x,y}$ be an independent random variable distributed according to $\Q$.
For $i\in\{1,2\}$, $\zeta^i_{x,y}$ is regarded as the passage time of $\eta^i$ through $(x,y)$; that is, when $\eta^i$ occupies $x$, 
then after time $\zeta^i_{x,y}$ we have that $\eta^i$ will occupy $y$ provided that $y$ has not been occupied by the other type.

Recall that, for any $t$, we define $\bar\eta^1(t)$ as the set of vertices of $\mathbb{Z}^d$ that are not contained in the infinite component of $\mathbb{Z}^d\setminus \eta^1(t)$, 
which comprises $\eta^1(t)$ and all vertices of $\mathbb{Z}^d\setminus \eta^1(t)$ that are separated from infinity by $\eta^1(t)$.
Theorem~\ref{thm:fpp} follows immediately from the theorem below by taking $\Q$ to be the exponential distribution of rate $\lambda$.
\begin{theorem}\label{thm:fpp2}
   For any $\lambda<1$, there exists a value $p_0\in(0,1)$ such that the following holds.
   For all $p\in(0,p_0)$
   and all $\Q$ satisfying $\mathcal{B}_\Q \subseteq \ball{\lambda}$, there are positive constants 
   $c_1=c_1(p,d,\Q)$ and $c_2=c_2(p,d,\Q)$ so that 
   $$
      \PR\lr{\bar\eta^1(t) \supseteq B(0,c_1t) \text{ for all $t\geq0$}}>c_2.
   $$
\end{theorem}

\subsection{Multi-scale setup}\label{sec:multiscale}
Let $\epsilon\in(0,1/2)$ be fixed and small enough so that all inequalities below hold:
\begin{equation}
   \lambda e^{2x}<\lambda(1+3x) < 1-2x 
   \quad \text{for all $x\in(0,\epsilon]$}.
   \label{eq:delta}
\end{equation}
We can define positive constants $C_\FPP<C_\FPP'$, depending only on $d$, such that for all $r>0$ we have
\begin{equation}
   [-C_\FPP r,C_\FPP r]^d \subset \ball{r} \subset [-C_\FPP' r,C_\FPP' r]^d.
   \label{eq:cfpp}
\end{equation}
Set $C_\FPP$ to be the largest constant and $C_\FPP'$ to be the smallest constant satisfying~\eqref{eq:cfpp}.
Since $\ball{r}$ is convex and has all the symmetries of the lattice $\mathbb{Z}^d$, we not only obtain that $\ball{r}$ is contained in the $\ell_\infty$-ball of radius $C'_\FPP r$ but it contains 
the $\ell_1$-ball of radius $C'_\FPP r$. Using that the latter contains the $\ell_\infty$-ball of radius $\frac{C'_\FPP r}{d}$, we obtain that
\begin{equation}
   \frac{C'_\FPP}{C_\FPP} \leq d.
   \label{eq:cfpprel}
\end{equation}

Given $\Q$, we can define $\Delta_\Q\geq 1$ as the smallest number such that 
\begin{equation}
   \mathcal{B}_\Q \Delta_\Q \supseteq \ball{\lambda}.
   \label{eq:deltaq}
\end{equation}
Equivalently, we have $\Delta_\Q = \sup_{x \in \ball{\lambda}} |x|_\Q$.
If $\Q$ is an exponential distribution of rate $\lambda$, we have $\Delta_\Q=1$.

Let $L_1$ be a large number, and fix $\alpha>1$ so that it satisfies the conditions in Proposition~\ref{pro:encapsulate2}. 
We let $k$ be an index for the scales.
For $k\geq 1$, once $L_k$ has been defined, we set 
\begin{align}
   R_k = \inf\lrc{r \geq L_k \colon \ball{r} \supset \left[-10d^2 \tfrac{C_\FPP'}{C_\FPP} L_k,10d^2 \tfrac{C_\FPP'}{C_\FPP} L_k\right]^d}.
   \label{eq:rk}
\end{align}
Also, for $k\geq1$, define
\begin{align}
   R_k^\enc = 2\alpha R_k \exp\lr{\frac{1+c_1}{2\epsilon}}
   \quad\text{and}\quad
   R_k^\couter = \frac{72k^2\Delta_\Q}{\epsilon}R_k^\enc,
   \label{eq:rencdef}
\end{align}
where $c_1$ is the constant in Proposition~\ref{pro:encapsulate2}. 
Since $1-\lambda > 2\epsilon$, we have that 
$\ball{R_k^\enc}$ contains all the passage times 
according to which the event in Proposition~\ref{pro:encapsulate2} with $r=R_k$ is measurable.
For $k\geq2$, let 
\begin{equation}
   L_k = \inf\lrc{\ell\geq 12 C_\FPP'R_{k-1}^\couter :  [-\ell/2,\ell/2]^d\supseteq \ball{100k^d R_{k-1}^\couter}}.
   \label{eq:lk}
\end{equation}
We then obtain the following bounds for $L_k$:
\begin{equation}
   200 C_\FPP k^d R_{k-1}^\couter 
   \leq L_k
   \leq 200 C_\FPP' k^d R_{k-1}^\couter 
   \label{eq:rl}
\end{equation}
and
\begin{equation}
   \frac{C^2_\FPP R_k}{10d^2 C'_\FPP}
   \leq L_k
   \leq \frac{C_\FPP R_k}{10d^2}.
   \label{eq:rl2}
\end{equation}
The first bound follows from~\eqref{eq:lk} and~\eqref{eq:cfpp}, and the fact that in~\eqref{eq:lk} $L_k$ is obtained via an infimum, so any cube containing
$\ball{100 k^d R_{k-1}^\couter}$ must have side length at least $L_k$.
The second bound follows from similar considerations, but applying~\eqref{eq:rk} and~\eqref{eq:cfpp}.

The intuition is that $L_k$ is the size of scale $k$, and $R_k$ is the radius of the clusters of $\eta^2(0)$ to
be treated at scale $k$.
The value of $R_k^\enc$ gives the radius inside which the encapsulation takes place; in the overview in Section~\ref{sec:overview} and in Figures~\ref{fig:poverview} and~\ref{fig:proof},
$R_k^\enc$ will be larger than the radius of each yellow ball so that each $\eta^2$ cluster treated at scale $k$ will be contained inside a ball of radius $R_k^\enc$. 
Regarding $R_k^\couter$, it represents a larger radius, which will be needed for 
the development of some couplings between scales; 
in the overview in Section~\ref{sec:overview} and in Figures~\ref{fig:poverview} and~\ref{fig:proof}, $R_k^\couter$ gives the radius of the red circles.

With the definitions above we obtain 
\begin{equation}
   R_k^\couter
   = \frac{144\alpha\exp\lr{\frac{1+c_1}{2\epsilon}}}{\epsilon}\Delta_\Q k^2R_k
   \leq \frac{1440 C_\FPP'\,d^2\,\alpha\exp\lr{\frac{1+c_1}{2\epsilon}}}{\epsilon C_\FPP^2}\Delta_\Q  k^2 L_k 
   \leq ck^{d+2} R_{k-1}^\couter,
   \label{eq:rkrel1}
\end{equation}
for some constant $c=c(d,\epsilon,\alpha,\Q)>\frac{288000 d^2 \alpha}{\epsilon}\exp\left(\frac{1+c_1}{2\epsilon}\right)\Delta_\Q$.
Iterating the above bound, we obtain
\begin{equation}
   R_k^\couter 
   \leq c^{k-1} \lr{k!}^{d+2} R_{1}^\couter
   \leq \lr{\frac{1440 C_\FPP'\,d^2\,\alpha\exp\lr{\frac{1+c_1}{2\epsilon}}\Delta_\Q}{\epsilon C_\FPP^2}}c^{k-1} \lr{k!}^{d+2} L_{1}.
   \label{eq:rkrel2}
\end{equation}
Using similar reasons we can see that 
\begin{equation}
   R_k \geq \frac{10d^2 L_k}{C_\FPP} 
   \geq 2000 d^2 k^d R_{k-1}^\couter
   = \frac{288000 d^2 \alpha k^{d+2}\exp\lr{\frac{1+c_1}{2\epsilon}}\Delta_\Q}{\epsilon} R_{k-1},
   \label{eq:rkrecursion}
\end{equation}
which allows us to conclude that 
$$
   R_k 
   \geq \tilde c^{k-1} (k!)^{d+2} R_1
   \geq \lr{100\frac{k^2}{\epsilon}}^{3d+6}
   \quad \text{ for all $k\geq 1$},
$$
where $\tilde c$ is a positive constant depending on $\alpha,\epsilon,d$ and $\Q$, and the last step follows for all $k\geq1$ by setting $L_1$ large enough. 

At each scale $k\geq 1$, tessellate $\mathbb{Z}^d$ into cubes of side-length $L_k$, producing a collection of disjoint cubes
\begin{equation}
   \{Q_k^\core(i)\}_{i\in\mathbb{Z}^d},
   \text{ where } Q_k^\core(i) = L_ki + [-L_k/2,L_k/2)^d.
   \label{eq:qcore}
\end{equation}
Whenever we refer to a \emph{cube} in $\mathbb{Z}^d$, we will only consider cubes of the form $\prod_{i=1}^d[a_i,b_i]$ for reals $a_i<b_i$, $i\in\{1,2,\ldots,d\}$. 
We will need cubes at each scale to overlap. We then define the following collection of cubes
$$
   \{Q_k(i)\}_{i\in\mathbb{Z}^d},
   \text{ where } Q_k(i) = L_ki + [-10dL_k,10dL_k]^d.
$$
We refer to each such cube $Q_k(i)$ of scale $k$ as a \emph{$k$-box}, and note that $Q_k(i)\supset Q_k^\core(i)$. 
One important property is that 
\begin{align*}
   \text{if a subset $A\subset \mathbb{Z}^d$ is completely contained inside a cube of side length 
   $18dL_k$,}\\
   \text{then $\exists i\in\mathbb{Z}^d$ such that $A\subset Q_k(i)$.}
\end{align*}

As described in the proof overview (see Section~\ref{sec:overview}), when going from scale $k$ to scale $k+1$, we will need to consider a slowed down version of $\eta^1$ and a 
sped up version of $\eta^2$.
For this reason we set $\epsilon_1=0$ and define for $k\geq 2$
$$
   \epsilon_k = \frac{\epsilon}{k^2}.
$$
Set $\lambda^1_1 = 1$ and $\lambda_1^2=\lambda$, 
and let $\zeta^1_1=\zeta^1$ and $\zeta^2_1=\zeta^2$ be the passage times used by $\eta^1$ and $\eta^2$, respectively.
For $k\geq 2$, 
define
$$
   \lambda^1_k = \exp\lr{-\sum\nolimits_{i=2}^k \epsilon_i}
   \quad\text{and}\quad
   \lambda^2_k = \lambda\,\exp\lr{\sum\nolimits_{i=2}^k \epsilon_i}.
$$
We have that $\lambda^1_k > \lambda^1_{k+1}$ and 
$\lambda^2_k < \lambda^2_{k+1}$ for all $k\geq 1$.
Also, note that 
$$
   \sum_{k=2}^{\infty} \epsilon_k 
   < \epsilon \int_1^\infty x^{-2}\,dx
   = \epsilon,
$$
which gives
$$
   \lambda^1_k 
   > e^{-\epsilon}
   > 1-\epsilon
   > \lambda e^\epsilon
   > \lambda^2_k
   \quad\text{for all $k\geq 1$},
$$
where the third inequality follows from the bound on $\epsilon$ via~\eqref{eq:delta}.

For each $k\geq 2$, consider two 
collections of passage times $\zeta_k^1$ and $\zeta_k^2$ on the edges of $\mathbb{Z}^d$, which are given by 
$\frac{\zeta^1}{\lambda^1_k}$ and $\frac{\zeta^2\lambda}{\lambda^2_k}$, respectively. 
These will be the passage times we will use in the analysis at scale $k$.
Note that, for any given $k$, the passage times of $\zeta_k^1$ are independent exponential random variables of 
parameter $\lambda_k^1$, while for the passage times of $\zeta_k^2$ we obtain that its limit shape is contained in $\ball{\lambda_k^2}$.

Moreover, up to a time scaling, having passage times $\zeta_k^1,\zeta_k^2$ is equivalent to having type 1 spreading at rate $1$, while type 2 spreads according to a random variable 
whose limit shape is contained in $\ball{\frac{\lambda_k^2}{\lambda_k^1}}$. Therefore, let $\lambda_k^\eff=\frac{\lambda_k^2}{\lambda_k^1}$ be the \emph{effective rate} of 
type~2 in comparison with that of type~1 at scale $k$. From now on, we will refer to the $\lambda^2_k$ as the rate of spread of type 2 at scale $k$, even if type~2 may not have exponential passage times.

We obtain that 
\begin{equation}
   \lambda 
   < \lambda_k^\eff
   \leq \frac{\lambda e^\epsilon}{e^{-\epsilon}}
   = \lambda e^{2\epsilon}
   < 1-2\epsilon.
   \label{eq:ratiolambda}
\end{equation}
Thus the effective rate of spread of type~2 is smaller than 1 at all scales.
We can also define the effective passage time of type~2 at scale $k$ as 
\begin{equation}
   \zeta_k^\eff = \frac{\zeta^2 \lambda \lambda_k^1}{\lambda_k^2}=\frac{\zeta^2 \lambda}{\lambda_k^\eff};
   \label{eq:zetaeff}
\end{equation}
in this way, at scale $k$, when employing Proposition~\ref{pro:encapsulate2}, we will take the passage times $\zeta^1_k,\zeta^2_k$ and scale time by a factor of 
$\lambda_k^1$, so that type~1 spreads according to the passage times $\zeta^1$ and type~2 spreads according to the passage times $\zeta_k^\eff$. 
Finally, for $k\geq 1$, define
\begin{equation}
   T_k^1 = R_k^\enc\lr{\frac{11-\lambda}{10}}^2\frac{1}{\lambda_k^1}
   \geq R_k^\enc\lr{\frac{11-\lambda_k^\eff}{10}}^2\frac{1}{\lambda_k^1}.
   \label{eq:tk1}
\end{equation}
Note that, using the passage times $\zeta^1_k,\zeta^2_k$, we have that $T_k^1$ represents the time required to run each encapsulation procedure at scale $k$ (before time is scaled by a factor of 
$\lambda_k^1$ as mentioned above).

\subsection{Definition of good boxes}\label{sec:good}

For each $Q_k(i)$, we will apply Proposition~\ref{pro:encapsulate2} to handle the situation where $Q_k(i)$ entirely contains a cluster of 
$\eta^2(0)$. 
At scale $k$ we will only handle the clusters that have not already been handled at a scale smaller than $k$.
By the relation between $L_k$ and $R_k$ in~\eqref{eq:rl2}, 
the cluster of $\eta^2$ inside $Q_k(i)$ will not start growing before
$\eta^1$ reaches the boundary of $L_ki+\ball{R_k}$. By the time $\eta^1$ reaches the boundary of $L_ki+\ball{R_k}$, 
$\eta^1$ must have crossed the boundary of $L_ki+\ball{\alpha R_k}$. 
(For the moment we assume that $L_ki+\ball{\alpha R_k}$ does not contain the origin, otherwise we will later consider that the origin has 
already been disconnected from infinity by $\eta^2$.) 
At this point we
apply Proposition~\ref{pro:encapsulate2} with $r=R_k$ and $\lambda=\lambda_k^\eff$, obtaining values $R$ and $T$ such that
\begin{equation}
   R\leq \alpha R_k \exp\lr{\frac{c_1}{1-\lambda_k^\eff}}
   \leq \alpha R_k \exp\lr{\frac{c_1}{2\epsilon}}
   \leq R_k^\enc,
   \label{eq:bdr}
\end{equation}
where the second inequality follows from~\eqref{eq:rencdef} and the last inequality follows from~\eqref{eq:ratiolambda},
and 
\begin{equation}
   T\leq R \left(\frac{11-\lambda_k^\eff}{10}\right)^2
   \leq R_k^\enc \left(\frac{11-\lambda_k^\eff}{10}\right)^2
   \leq \lambda_k^1 T_k^1,
   \label{eq:tr}
\end{equation}
where the
last two inequalities follow from~\eqref{eq:bdr} and~\eqref{eq:tk1}, respectively.
Note that in our application of Proposition~\ref{pro:encapsulate2} above time has been scaled by $\lambda_k^1$, since we apply it with type~1 (resp., type~2) spreading at rate 1 (resp., $\lambda_k^\eff$) 
instead of the actual rate $\lambda_k^1$ (resp., $\lambda_k^2$).
This is the reason why the term $\frac{1}{\lambda_k^1}$ appears in the definition of $T_k^1$ in~\eqref{eq:tk1}.
With this we get $\lambda_{k}^1 T_k^1$ in the right-hand side of~\eqref{eq:tr}, and the actual time that the encapsulation procedure takes is 
$\frac{1}{\lambda_k^1}T\leq T_k^1$.
At the moment we have not yet defined the sets $\{\Pi_\iota\}_\iota,\{\Pi'_\iota\}_\iota$; they will only be defined precisely in Section~\ref{sec:ci}.

Now let $E_k(i,x)$ with $x\in L_ki+\ball{\alpha R_k}\setminus \ball{\alpha R_k/2}$,
be the event in the application of Proposition~\ref{pro:encapsulate2} with the origin at $L_ki$,
$r=R_k$, passage times given by $\zeta^1,\zeta^\eff_k$, and $\eta^1$ starting from $x$.
Here $x$ represents the first vertex of $\bddo\lr{L_ki+ \ball{\alpha R_k}}$ occupied by $\eta^1$, from where the encapsulation of the cluster of $\eta^2$ inside 
$L_ki+\ball{R_k^\enc}$ will start.
Letting $B_k(i) =  \lr{L_ki+\ball{\alpha R_k}\setminus \ball{\alpha R_k/2}} \cup \bddo\lr{L_ki+\ball{\alpha R_k}}$, define
\begin{align*}
   G_k^\enc(i) \text{ to be the event that $E_k(i,x)$ holds for all $x\in B_k(i)$}.
\end{align*}
The event $G_k^\enc(i)$ implies that $\eta^1$ encapsulates $\eta^2$ inside $L_ki+\ball{R_k^\enc}$ during a time interval of length $T_k^1$,
unless $\eta^2$ ``invades'' $L_ki+\ball{R_k^\enc}$ from outside, that is, unless another cluster of $\eta^2$ starts growing and reaches the boundary of 
$L_ki+\ball{R_k^\enc}$ before $\eta^1$ manages to encapsulate $\eta^2$ inside $L_ki+\ball{R_k^\enc}$. 
(When we apply the above argument later in the proof, we will only try to encapsulate a cluster of $\eta^2$ at scale $k$ if the ball $L_ki+\ball{R_k^\couter}\supset L_ki+ \ball{R_k^\enc}$ does not 
intersect other balls being treated at the same scale. If there is another ball being treated at the same scale $k$ and intersecting  $L_ki+\ball{R_k^\couter}$, then these balls will be only treated at a larger scale, not allowing 
different clusters of $\eta^2$ of the same scale to interfere in each other's encapsulation.)

For each $i\in \mathbb{Z}^d$, define
\begin{align*}
   Q_k^\alpha(i) = L_ki + \ball{\alpha R_k},
   \quad
   Q_k^\enc(i) = L_ki + \ball{R_k^\enc},\\
   Q_k^\couter(i) = L_ki + \ball{R_k^\couter}
   \quad\text{and}\quad
   Q_k^\couters(i) = L_ki + \ball{R_k^\couter/3}.
\end{align*}
We will also define two other events $G_k^{1}(i)$ and $G_k^{2}(i)$, which will be measurable with respect to $\zeta^1_k,\zeta^2_k$ inside 
$Q_k^\couter(i)$.
For any $X\subset\mathbb{Z}^d$, let $\zeta^1_{k}|_X$ be the passage times that are equal to $\zeta^1_k$ inside 
$X$ and are equal to infinity everywhere else; define $\zeta^2_{k}|_X$ analogously.
Define the event $G_k^{1}(i)$ as
\begin{align*}
   \Big\{&D\big(\bddo Q_k^\enc(i),\bddi Q_k^\couters(i); \zeta_{k+1}^1|_{Q_k^\couter(i)}\big) 
   \geq T_k^1 + \sup\nolimits_{x\in \bddi Q_k^\couters(i)} D\big(\bddo Q_k^\alpha(i),x; \zeta_{k}^1|_{Q_k^\couter(i)}\big)\Big\}.
\end{align*}
The main intuition behind this event is that, during the encapsulation of 
a $(k+1)$-box, $\eta^1$ will need to perform some small local detours when encapsulating clusters of scale $k$ or smaller. We can capture this by 
using the slower passage times $\zeta_{k+1}^1$. If $G_k^{1}$ holds for the $k$-boxes that are traversed during the encapsulation of a 
$(k+1)$-box, then using the slower passage times $\zeta_{k+1}^1$ but ignoring the actual detours around $k$-boxes will only slow down $\eta^1$.

We also need to handle the case where the growth of $\eta^2$ is sped up by absorption of smaller scales. 
For $i\in\mathbb{Z}^d$, define 
\begin{align*}
   G_k^{2}(i)
   =\bigcap_{x\in \bddi Q_k^\couters(i)}\Big\{& D\big(x,\bddo Q_k^\enc(i); \zeta_{k}^2|_{Q_k^\couter(i)}\big)
   \geq \sup\nolimits_{y\in Q_k^\enc(i)}  D\big(x,y; \zeta_{k+1}^2|_{Q_k^\couter(i)}\big)\Big\}.
\end{align*}
Note that the event $G_k^2(i)$ implies the following. 
Let $x\in \bddi Q_k^\couters(i)$ be the first vertex of $Q_k^\couters(i)$ reached by $\eta^2$ from outside 
$Q_k^\couters(i)$. While $\eta^2$ travels from $x$ to $Q_k^\enc(i)$, the encapsulation of $Q_k^\enc(i)$
may start taking place. Then, $\eta^2$ can only get a sped up inside $Q_k^\enc(i)$ if 
$\eta^2$ enters $Q_k^\enc(i)$ before the encapsulation of $Q_k^\enc(i)$ is completed. However, under $G_k^2(i)$ and the passage times 
$\zeta_k^2$, the time that $\eta^2$ takes to go from $x$ to $Q_k^\enc(i)$ is larger than the time, under $\zeta_{k+1}^2$, that 
$\eta^2$ takes to go from $x$ to all vertices in $Q_k^\enc(i)$. Therefore, under $G_k^2(i)$, we can use the faster passage times 
$\zeta_{k+1}^2$ to absorb the possible sped up that $\eta^2$ may get by the cluster growing inside $Q_k^\enc(i)$.

For $i\in\mathbb{Z}^d$ and $k\geq1$, we define 
$$
   G_k(i)= G_k^\enc(i) \cap G_k^{1}(i) \cap G_k^{2}(i),
$$
and say that 
$$
   Q_k(i) \text{ is \emph{good} if } G_k(i) \text{ holds}.
$$
Hence, intuitively, $Q_k(i)$ being good means that $\eta^1$ successfully encapsulates the growing cluster of $\eta^2$ inside $Q_k(i)$, and this happens 
in such a way that the detour of $\eta^1$ during this encapsulation is faster than 
letting $\eta^1$ use passage times $\zeta_{k+1}^1$, and also the possible sped up that $\eta^2$ may get from clusters of 
$\eta^2$ coming from outside $Q_k(i)$ is slower than letting $\eta^2$ use passage times $\zeta_{k+1}^2$.

We now explain why in the definition of $G_k^1(i)$ and 
$G_k^2(i)$ we calculate passage times from $\partial Q_k^\couters(i)$ instead of from $\partial Q_k^\couter(i)$. 
The reason is that we had to define $G_k^1(i)$ and 
$G_k^2(i)$ in such a way that they are measurable with respect to the passage times inside $Q_k^\couter(i)$. 
We do this by forcing to use only passage times inside $Q_k^\couter(i)$. 
By using the distance between 
$\partial Q_k^\couter(i)$ and $\partial Q_k^\couters(i)$, we can ensure that this constraint does not change much the probability that the corresponding events occur.

\subsection{Probability of good boxes}\label{sec:prgood}
In this section we show that the events $G_k^\enc(i)$, $G_k^1(i)$ and $G_k^2(i)$, defined in Section~\ref{sec:good}, are likely to occur.

\begin{lemma}\label{lem:goodbox}
   There exist positive constants $L_0=L_0(d,\epsilon)$ and $c=c(d,\Q)$ 
   such that if $L_1\geq L_0$, then for any $k\geq 1$ and any $i\in \mathbb{Z}^d$ we have
   $$
      \PR\lr{G_k(i)}
      \geq 1 - \exp\lr{-c\lr{\epsilon\lambda R_k^\enc}^\frac{d+1}{2d+4}}.
   $$
   Moreover, the event $G_k(i)$ is measurable with respect to the passage times inside $Q_k^\couter(i)$.
\end{lemma}

Before proving the lemma above, we state and prove two lemmas regarding the probability of the events $G_k^1(i)$ and $G_k^2(i)$.
\begin{lemma}\label{lem:g1}
   There exist positive constants $L_0=L_0(d,\epsilon)$ and $c=c(d)$ 
   such that if $L_1\geq L_0$, then for any $k\geq 1$ and any $i\in \mathbb{Z}^d$ we have
   $$
      \PR\lr{G_k^1(i)}\geq 
      1 - \exp\lr{-c \lr{R_k^\couter}^\frac{d+1}{2d+4}}.
   $$
\end{lemma}
\begin{proof}
   Set $\delta=\frac{\epsilon}{120 k^{2}}$.
   Define 
   $$
      \tau_1 = \lr{\frac{24k^2}{\epsilon}-1} \frac{1}{\lambda_{k+1}^1}R_k^\enc
              = \frac{1}{\lambda_{k+1}^1}\left(\frac{1}{3}R_k^\couter - R_k^\enc\right)
   $$
   and
   $$
      \tau_2 = \frac{24k^2}{\epsilon\lambda_k^1}R_k^\enc
       = \frac{1}{3\lambda_k^1}R_k^\couter.
   $$
   We will show that there exists a constant $c=c(d)>0$ such that 
   \begin{align}
      \PR\lr{
      D\big(\bddo Q_k^\enc(i), \bddi Q_k^\couters(i);\zeta_{k+1}^1|_{Q_k^\couter(i)}\big)
      \leq (1-\delta)\tau_1}
      \leq (R_k^\couter)^{3d} \exp\lr{-c (R_k^\couter)^{\frac{d+1}{2d+4}}},
      \label{eq:pr1}
   \end{align}
   and 
   \begin{align}
      \PR\lr{\sup\nolimits_{x\in \bddi Q_k^\couters(i)}
      D\big( \bddo Q_k^\alpha(i), x;\zeta_{k}^1|_{Q_k^\couter(i)}\big)
      \geq (1+\delta)\tau_2}&\nonumber\\
      \leq  \frac{c}{1-2\epsilon} (R_k^\couter)^{2d} \exp\lr{-c (R_k^\couter)^{\frac{d+1}{2d+4}}}.&
      \label{eq:pr2}
   \end{align}
   Using~\eqref{eq:pr1} and~\eqref{eq:pr2}, it remains to show that 
   $$
      (1-\delta)\tau_1
      \geq T_k^1 + (1+\delta)\tau_2.
   $$
   Note that 
   \begin{align*}
      (1-\delta)\tau_1
      &= (1-\delta)\lr{\frac{24k^2}{\epsilon}-1}\frac{\exp\lr{\epsilon (k+1)^{-2}}}{\lambda_k^1}R_k^\enc\\
      &= T_k^1 + \lr{(1-\delta)\lr{\frac{24k^2}{\epsilon}-1}\exp\lr{\epsilon (k+1)^{-2}}-\lr{\frac{11-\lambda}{10}}^2}\frac{R_k^\enc}{\lambda_k^1}.
   \end{align*}
   Thus we need to show that the last term in the right-hand side above is at least $(1+\delta)\tau_2$, which is equivalent to showing that 
   $$
      (1-\delta)\lr{\frac{24k^2}{\epsilon}-1}\exp\lr{\epsilon (k+1)^{-2}}-\lr{\frac{11-\lambda}{10}}^2
      \geq (1+\delta)\frac{24k^2}{\epsilon}.
   $$
   Rearranging the terms, the inequality above translates to 
   $$
      \frac{24k^2}{\epsilon}\lr{(1-\delta)\exp\lr{\epsilon(k+1)^{-2}}-1-\delta} \geq \lr{\frac{11-\lambda}{10}}^2 + (1-\delta)\exp\lr{\epsilon(k+1)^{-2}}.
   $$
   Using that $\exp\lr{\epsilon(k+1)^{-2}}\geq 1+\epsilon (k+1)^{-2}$ and then applying the value of $\delta$, we obtain that the left-hand side above is at least 
   \begin{align*}
      \frac{24k^2}{\epsilon}\lr{\epsilon(k+1)^{-2}-2\delta-\delta \epsilon(k+1)^{-2}}
      &= 24\lr{\frac{k^2}{(k+1)^{2}}-\frac{1}{60}-\frac{\epsilon}{120(k+1)^{2}}}\\
      &\geq 24\lr{\frac{1}{4}-\frac{1}{60}-\frac{1}{480}}
      >5.
   \end{align*}
   Hence, it now suffices to show that 
   $$
      5 \geq \lr{\frac{11-\lambda}{10}}^2 + (1-\delta)\exp\lr{\epsilon(k+1)^{-2}},
   $$
   which is true since the right-hand side above is at most 
   $\lr{\frac{11-\lambda}{10}}^2 + \exp\lr{\epsilon/4}\leq \lr{\frac{11}{10}}^2 + \exp\lr{1/4} \leq 3$.
   
   Now we turn to establish~\eqref{eq:pr1} and~\eqref{eq:pr2}.
   We start with~\eqref{eq:pr1}.
   First note that 
   $$
      R_k^\couter/3-R_k^\enc
      = \lr{\frac{24k^2}{\epsilon}-1}R_k^\enc
      = \tau_1\lambda_{k+1}^1.
   $$
   Recall the notation $S_t^\delta$ from Proposition~\ref{pro:kesten}, which is the (unlikely) event that at time $t$ first passage percolation of rate 1 
   does not contain $\ball{(1-\delta)t}$ or is not contained in $\ball{(1+\delta)t}$.
   Then using time scaling to go from passage times of rate $\lambda_{k+1}^1$ to passage times of rate $1$, and using the union bound on $x$, we obtain
   \begin{align*}
      &\PR\lr{
      D\big(\bddo Q_k^\enc(i), \bddi Q_k^\couters(i);\zeta_{k+1}^1|_{Q_k^\couter(i)}\big)
      \leq (1-\delta)\tau_1}\\
      &=\PR\lr{
      D\big(\bddo Q_k^\enc(i), \bddi Q_k^\couters(i);\zeta_{1}^1|_{Q_k^\couter(i)}\big)
      \leq (1-\delta)\tau_1\lambda_{k+1}^1}\\
      &\leq \sum\nolimits_{x\in \bddi Q_k^\couters(i)}\PQ\lr{S_{(1-\delta)\tau_1\lambda_{k+1}^1}^\delta}\\
      &\leq c_1 \lr{(1-\delta)\tau_1\lambda_{k+1}^1}^{3d} \exp\lr{-c_2 \lr{(1-\delta)\tau_1\lambda_{k+1}^1}^\frac{d+1}{2d+4}}\\
      &\leq c_1 (R_k^\couter)^{3d} \exp\lr{-c_3 (R_k^\couter)^\frac{d+1}{2d+4}},
   \end{align*}
   where in the first inequality we used that $(1+\delta)(1-\delta)\tau_1\lambda_{k+1}^1\leq \tau_1\lambda_{k+1}^1 = R_k^\couter/3-R_k^\enc$,
   and in the second inequality we applied Proposition~\ref{pro:kesten}.
   
   Now we turn to~\eqref{eq:pr2}. 
   We again use time scaling and the fact that $\tau_2\lambda_k^1=R^\couter_k/3$ to write
   \begin{align*}
      &\PR\lr{
         \sup\nolimits_{x\in \bddi Q_k^\couters(i)} D\Big( \bddo Q_k^\alpha(i), x;\zeta_{k}^1|_{Q_k^\couter(i)}\Big)
         \geq (1+\delta)\tau_2}\\
      &\leq \PR\lr{
         \sup\nolimits_{x\in \bddi Q_k^\couters(i)}D\Big( L_k i, x;\zeta_{k}^1|_{Q_k^\couter(i)}\Big)
         \geq (1+\delta)\tau_2}\\
      &\leq \sum\nolimits_{x\in \bddi Q_k^\couters(i)}\PQ\lr{S_{(1+\delta)\tau_2\lambda_k^1}^{\delta/2}}\\
      &\leq c_1 \lr{R_k^\couter}^{3d} \exp\lr{-c_2 \lr{R_k^\couter}^\frac{d+1}{2d+4}},
   \end{align*}
   where the second inequality follows since 
   $(1-\delta/2)(1+\delta)\tau_2\lambda_k^1\geq R_k^\couter/3$ for all $\delta\in[0,1]$.
   Moreover, 
   $(1+\delta/2)(1+\delta)\tau_2\lambda_k^1< \frac{2 R_k^\couter}{3}$, implying that $S_{(1+\delta)\tau_2\lambda_k^1}^{\delta/2}$ is 
   measurable with respect to the passage times inside $Q_k^\couter(i)$.
   Finally, the last step of the derivation above follows from Propositon~\ref{pro:kesten}.
\end{proof}

The next lemma shows that $G_k^2(i)$ occurs with high probability.
\begin{lemma}\label{lem:g2}
   There exist positive constants $L_0=L_0(d,\epsilon)$ and $c=c(d,\nu)$ 
   such that if $L_1\geq L_0$, then for any $k\geq 1$ and any $i\in \mathbb{Z}^d$ we have
   $$
      \PR\lr{G_k^2(i)}\geq 
      1 - \exp\lr{-c \lr{R_k^\couter}^\frac{d+1}{2d+4}}.
   $$
\end{lemma}
\begin{proof}
   Set $\delta=\frac{\epsilon}{20 k^{2}}$ and fix an arbitrary $x\in \bddi Q_k^\couters(i)$. 
   Define the smallest distance between $x$ and $Q_k^\enc(i)$ with respect to the norm $\Q$ as 
   $$
      m = \min_{y \in \bddi Q_k^\enc(i)} |x-y|_\Q.
   $$
   Since $\mathcal{B}_\Q\subseteq \ball{\lambda}$, we have that 
   $$
      m \geq \frac{1}{\lambda} \lr{\frac{R_k^\couter}{3}- R_k^\enc}.
   $$
   Under the passage times $\zeta_k^2$, the time it takes to reach $Q_k^\enc(i)$ from $x$ is roughly $m \frac{\lambda}{\lambda_k^2}$. Therefore, we define 
   $$
      \tau_1 = \frac{m\lambda}{\lambda_k^2},
   $$
   and will show later that there exists a constant $c'>0$ such that, uniformly over $x$,  
   \begin{align}
      &\PR\lr{
      D\big( x, \bddo Q_k^\enc(i);\zeta_{k}^2|_{Q_k^\couter(i)}\big)
      \leq (1-\delta)\tau_1}
      \leq \exp\lr{-c' (R_k^\couter)^{\frac{d+1}{2d+4}}}.
      \label{eq:pr21}
   \end{align}
   Now, under the faster passage times $\zeta_{k+1}^2$, the time it takes to reach $Q_k^\enc(i)$ from $x$ is roughly $m \frac{\lambda}{\lambda_{k+1}^2}$. 
   Let $x'\in\bddi Q_k^\enc(i)$ be the first vertex of $Q_k^\enc(i)$ reached from $x$. 
   Note that 
   $$
      x'+\mathcal{B}_\Q \Delta_\Q \frac{2 R_k^\enc}{\lambda}
      \supseteq x'+\ball{\lambda} \frac{2 R_k^\enc}{\lambda}
      = x'+\ball{2 R_k^\enc}
      \supseteq Q_k^\enc(i).
   $$
   Under the passage times $\zeta_{k+1}^2$, which is a scale of $\zeta^2$ by a factor of $\frac{\lambda}{\lambda_{k+1}^2}$, 
   the time until $x'+\mathcal{B}_\Q \Delta_\Q \frac{2 R_k^\enc}{\lambda}$ is fully occuppied starting from $y$ is roughly 
   $\Delta_\Q \frac{2 R_k^\enc}{\lambda_{k+1}^2}$. Therefore, we set
   $$
      \tau_2 = m \frac{\lambda}{\lambda_{k+1}^2} + \Delta_\Q \frac{2 R_k^\enc}{\lambda_{k+1}^2},
   $$
   and will show that there exists a constant $c''>0$ such that 
   \begin{align}
      \PR\lr{
      \sup\nolimits_{y\in \bddo Q_k^\enc(i)} D\Big( x, y;\zeta_{k+1}^2|_{Q_k^\couter(i)}\Big)
      \geq (1+\delta)\tau_2}
      \leq  \exp\lr{-c'' (R_k^\couter)^{\frac{d+1}{2d+4}}}.
      \label{eq:pr22}
   \end{align}
   Assuming~\eqref{eq:pr21} and~\eqref{eq:pr22} for the moment, it remains to show that 
   \begin{equation}
      (1-\delta)\tau_1
      \geq (1+\delta)\tau_2.
      \label{eq:pr2122}
   \end{equation}
   Replacing $\lambda_{k+1}^2$ with $\lambda_k^2 \exp(\epsilon(k+1)^{-2})$ in the definition of $\tau_2$,~\eqref{eq:pr2122} follows if we show that 
   \begin{align*}
      (1-\delta)m \lambda 
      \geq (1+\delta)\exp\lr{-\epsilon (k+1)^{-2}}\lr{m\lambda + 2 \Delta_\Q R_k^\enc}.
   \end{align*}
   First note that 
   $$
      2\Delta_\Q R_k^\enc 
      =  \frac{2\epsilon}{24 k^2} \cdot \frac{R_k^\couter}{3}
      \leq \frac{3\epsilon}{24 k^2} \lr{\frac{R_k^\couter}{3} - R_k^\enc}
      \leq \frac{3\epsilon}{24 k^2} m\lambda
      = \frac{5\delta}{2} m\lambda,
   $$
   where the first inequality follows by the definition of $R_k^\couter$ in~\eqref{eq:rencdef}.
   So now it suffices to show that 
   \begin{align*}
      1-\delta
      \geq (1+\delta)\exp\lr{-\epsilon (k+1)^{-2}}\lr{1+ \frac{5\delta}{2}}.
   \end{align*}
   Rearranging gives that $\frac{1-\delta}{(1+\delta)(1+5\delta/2)}\geq \exp\lr{-\epsilon (k+1)^{-2}}$. The left-hand side is at least 
   $(1-\delta)^2(1-5\delta/2)\geq 1- \frac{9\delta}{2}$.
   Using that $e^{-a}\leq 1-a+a^2/2$ for all $a\geq 0$,~\eqref{eq:pr2122} holds if the following is true
   $$
      \frac{9\delta}{2} 
      \leq \frac{\epsilon}{(k+1)^2}\lr{1 - \frac{\epsilon}{2(k+1)^2}}.
   $$
   Using the value of $\delta$, we are left to showing 
   $$
      \frac{9}{40} \leq \frac{k^2}{(k+1)^2}\lr{1 - \frac{\epsilon}{2(k+1)^2}},
   $$
   which is true since the right-hand side is at least $\frac{1}{4} \cdot \lr{1-\frac{\epsilon}{8}} \geq \frac{1}{4} \cdot \frac{15}{16}$.
   This establishes~\eqref{eq:pr2122}.
   
   Now we turn to establish~\eqref{eq:pr21} and~\eqref{eq:pr22}, which essentially follow from Proposition~\ref{pro:kesten}. 
   We start with~\eqref{eq:pr21}.
   Scaling the passage times $\zeta_{k^2}$ by $\frac{\lambda_k^2}{\lambda}$ we obtain passage times distributed as $\Q$. 
   Hence,
   \begin{align*}
      \PR\lr{
      D\Big( x, \bddo Q_k^\enc(i);\zeta_{k}^2|_{Q_k^\couter(i)}\Big)
      \leq (1-\delta)\tau_1}
      &\leq \PQ^{\Q}\lr{S_{\tau_1\lambda_k^2/\lambda}^\delta}\\
      &\leq \exp\lr{-c' (R_k^\couter)^{\frac{d+1}{2d+4}}}.
   \end{align*}   
   The same reasoning holds for~\eqref{eq:pr22}, which gives
   \begin{align*}
      \PR\lr{
      \sup\nolimits_{y\in \bddo Q_k^\enc(i)} D\Big( x, y;\zeta_{k+1}^2|_{Q_k^\couter(i)}\Big)
      \geq (1+\delta)\tau_2}
      &\leq  \PQ^{\Q}\lr{S_{\tau_2\lambda_{k+1}^2/\lambda}^{\delta}}\\
      &\leq \exp\lr{-c'' (R_k^\couter)^\frac{d+1}{2d+4}}.
   \end{align*}
   Then the lemma follows by taking the union bound over $x$, and using the fact that $R_k^\couter$ is very large at all scales so that the extra term obtained from the union bound
   can be absorbed in the constant $c$.
\end{proof}

\begin{proof}[Proof of Lemma~\ref{lem:goodbox}]
   Proposition~\ref{pro:encapsulate2} gives that $G_k^\enc(i)$ can be defined so that it is measurable with respect to the passage times inside 
   \begin{align*}
      L_ki+\ball{\alpha R_k \exp\lr{\frac{c_1}{1-\lambda_k^\eff}} \lr{\frac{11-\lambda_k^\eff}{10}}^2}
      &\subseteq L_ki+\ball{\alpha R_k \exp\lr{\frac{c_1}{2\epsilon}}\lr{\frac{11}{10}}^2}\\
      &\subseteq L_ki+\ball{R_k^\enc}.
   \end{align*}
   Moreover, Proposition~\ref{pro:encapsulate} gives a constant $c_2>0$ so that, for all large enough $L_1$, we have
   \begin{align*}
      \PR\lr{G_k^\enc(i)} 
      &\geq 1- \sum_{x\in B_k(i)} \exp\lr{-c_2 \lr{\lambda_k^\eff\lr{1-\lambda_k^\eff}\alpha R_k}^\frac{d+1}{2d+4}}\\
      &\geq 1- \exp\lr{-c \lr{2\epsilon \lambda\alpha R_k}^\frac{d+1}{2d+4}},
   \end{align*}
   where the last step follows by applying the bounds in~\eqref{eq:ratiolambda} and $c$ is a positive constant.
   By definition, the events $G_k^1$ and $G_k^2(i)$ are measurable with respect to the passage times inside $L_ki+\ball{R_k^\couter}$. So 
   the proof is completed by using the bounds in Lemmas~\ref{lem:g1} and~\ref{lem:g2}.
\end{proof}

\subsection{Contagious and infected sets}\label{sec:ci}
%
As discussed in the proof overview in Section~\ref{sec:overview}, 
for each scale $k$, we will define a set $C_k\subset\mathbb{Z}^d$ as the set of \emph{contagious} vertices at scale $k$, and 
also define a set $I_k\subset\mathbb{Z}^d$ as the set of \emph{infected} vertices at scale $k$. The main intuition behind such sets is that $C_k$ represents 
the vertices of $\mathbb{Z}^d$ that need to be handled at scale $k$ or larger, whereas $I_k$ represents the vertices of $\mathbb{Z}^d$ that may be taken by $\eta^2$
at scale $k$. In particular, we will show that the vertices of $\mathbb{Z}^d$ that will be occupied by $\eta^2$ are contained in $\bigcup_{k\geq 1}I_k$.

At scale $1$ we set the contagious vertices as those initially taken by $\eta^2$; that is, 
$$
   C_1 = \eta^2(0).
$$
All clusters of $C_1$ that belong to good 1-boxes and that are not too close to contagious clusters from other 1-boxes
will be ``cured'' by the encapsulation process 
described in the previous section. 
The other vertices of $C_1$ will become contagious vertices for scale 2, together with the vertices belonging to bad 1-boxes.
Using this, define $C_k^\bad$ as the following subset of the contagious vertices:
\begin{equation}
    \lrc{x \in C_{k} \colon \text{for all $i$ with $x\in Q_k(i)$ we have $iL_k + \ball{3R_k^\couter} \cap \lr{C_k \setminus Q_k(i)}\neq\emptyset$}}.
    \label{eq:cbad}
\end{equation}
Intuitively, $C_k^\bad$ is the set of contagious vertices that cannot be cured at scale $k$ 
since they are not far enough from other contagious vertices in other $k$-boxes. 
Now for the vertices in $C_k\setminus C_k^\bad$, the definition of $C_k^\bad$ gives that we can select a set $\mathcal{I}_k\subset \mathbb{Z}^d$ representing 
$k$-boxes such that 
for each $x\in C_k\setminus C_k^\bad$ there exists a unique $i\in\mathcal{I}_k$ for which $x\in Q_k(i)$, and for each pair 
$i,j\in \mathcal{I}_k$, we have $Q_k^\couter(i)\cap Q_k^\couter(j)=\emptyset$.
Then, given $C_k$, we define $I_k$ as the set of vertices that can be taken by $\eta^2$ during the encapsulation of the good $k$-box, which is more precisely given
by 
$$
   I_k = \lrc{Q_k^\enc(i) \colon Q_k(i) \text{ is good and } i \in \mathcal{I}_k}.
$$
We then define inductively 
\begin{align}
   C_{k+1} = C_{k}^\bad \cup \lrc{Q_{k}^\couter(i) \colon i \in \mathcal{I}_k \text{ and } Q_{k}(i) \text{ is bad}}.
   \label{eq:defck}
\end{align}

The lemma below gives that if the contagious sets of scales larger than $k$ 
are all empty, then $\eta^2$ must be contained inside $\bigcup_{j=1}^{k-1}I_j$.
\begin{lemma}\label{lem:inf}
   Let $A\subset \mathbb{Z}^d$ be arbitrary. Then, for any $k\geq1$, either we have that 
   \begin{equation}
      \text{there exists $j> k$ and $i\in\mathbb{Z}^d$ with $Q_j^\couter(i)\cap A\neq \emptyset$ for which $Q_j^\core(i)\cap C_j \neq \emptyset$},
      \label{eq:imp1}
   \end{equation}
   or
   \begin{equation}
      \eta^2(t)\cap A \subset \bigcup\nolimits_{j=1}^{k} I_j \quad\text{for all $t\geq0$}.
      \label{eq:imp2}
   \end{equation}
\end{lemma}
\begin{proof}
   We will assume that~\eqref{eq:imp1} does not occur; that is, 
   \begin{equation}
      \text{the set } \bigcup\nolimits_{j>k}\bigcup\nolimits_{i\colon Q_j^\core(i)\cap C_j \neq \emptyset} Q_j^\couter(i) 
      \text{ does not intersect $A$}.
      \label{eq:imp1p}
   \end{equation}
   The lemma will follow by showing that the above implies~\eqref{eq:imp2}.
   
%
   We start with scale 1.
   Recall that $C_1$ contains all elements of $\eta^2(0)$.
   Then, all elements of $C_1\setminus C_1^\bad$ are handled at scale 1. 
   Let $i\in\mathcal{I}_1$, so $Q_1(i)$ intersects $C_1\setminus C_1^\bad$. 
   If $Q_1(i)$ is a good box, the passage times inside $Q_1^\enc(i)$ are such that 
   $\eta^1$ encapsulates $\eta^2$ within $Q_1^\enc(i)$ unless another cluster of $\eta^2$ enters $Q_1^\enc(i)$ from outside. 
   When the encapsulation succeeds, we have that the cluster of $\eta^2$ growing inside $Q_1^\enc(i)$ never 
   exits $Q_1^\enc(i)\subset I_1$. 
   
   Before proceeding to the proof for scales larger than $1$, we explain the possibility that the encapsulation above does not succeed because another cluster of $\eta^2$ (say, from $Q_1(j)$) enters $Q_1^\enc(i)$ from outside. 
   Note that if $Q_1^\couter(j)\cap Q_1^\couter(i) \neq \emptyset$, then the two clusters are not handled at scale 1: they will be handled together at a higher scale.
   Now assume that $Q_1^\couter(j)$ and $Q_1^\couter(i)$ are disjoint and do not intersect any other region $Q_1^\couter$ from a contagious site. 
   Thus both $Q_1(i)$ and $Q_1(j)$ are handled at scale 1.
   If they are both good, the encapsulations succeed within $Q_1^\enc(i)$ and $Q_1^\enc(j)$, and do not interfere with each other.
   Assume that $Q_1(i)$ is good, but $Q_1(j)$ is bad. In this case, we will make $Q_1^\couter(j)$ to be contagious for scale $2$, but up to scale $1$ this does not interfere with the encapsulation within
   $Q_1^\enc(i)$ because these two regions are disjoint. The encapsulation of $Q_1^\couter(j)$ will be treated at scale $2$ or higher, and the fact that 
   $Q_1^\couter(j)\cap Q_1^\couter(i)=\emptyset$ will be used to allow a coupling argument between scales. 
  
   We now explain the analysis for a scale $j\in\{2,3,\ldots,k\}$, assuming that we have carried out the analysis until scale $j-1$.
   Thus, we have showed that all contagious vertices successfully handled at scale smaller than $j$ are contained inside $I_1\cup I_2 \cup\cdots\cup I_{j-1}$.
   Consider a cell $Q_j(i)$ of scale $j$ with $i\in \mathcal{I}_j$. 
   During the encapsulation of $\eta^2$ inside $Q_j^\enc(i)$, it may 
   happen that $\eta^1$ advances through a cell $Q_{j-1}(i')$ that was treated at scale $j-1$; 
   that is, $i'\in\mathcal{I}_{j-1}$. (For simplicity of the discussion, we assume here that this cell is of scale $j-1$, but 
   it could be of any scale $j'\leq j-1$.)
   Note that $Q_{j-1}(i')$ must be good for scale $j-1$ because otherwise cell $i$ would not be treated at scale $j$.
   The fact that $Q_{j-1}(i')$ is good implies that the time $\eta^1$ takes to go from $\bddi Q_{j-1}^\couters(i')$ to all points in 
   $\bddi Q_{j-1}^\enc(i')$, therefore encapsulating $Q_{j-1}(i')$, is 
   smaller than the time given by the passage times $\zeta_j^1$. 
   Moreover, $Q_{j-1}(i')$ being good implies that the time $\eta^2$ takes to go from $\bddi Q_{j-1}^\couters(i')$ to any point in 
   $\bddo Q_{j-1}^\enc(i')$ is larger 
   than the time given by the passage times $\zeta_j^2$. This puts us in the context of Proposition~\ref{pro:encapsulate2}, where the sets $\{\Pi_\iota\}_\iota$ are given by the clusters 
   of $\bigcup_{j''=1}^{j-1} \bigcup_{i'' \in \mathcal{I}_{j''}} Q_{j''}^\couter(i'')$, and for each $\iota$, the set $\Pi_\iota'\subset \Pi_\iota$ is given by the union of $Q_{j''}^\enc(i'')$ over all 
   $j'',i''$ for which $Q_{j''}^\couter(i'')\subset \Pi_\iota$.
   Therefore, under the event that all the cells involved in the definition of $\{\Pi_\iota\}_\iota$ are good, Proposition~\ref{pro:encapsulate2} gives 
   $\eta^2$ cannot escape the set $\bigcup_{\iota=1}^j I_\iota$, after all contagious vertices of scale at most $j$ have been analyzed. 
   Therefore, inductively we obtain that $\eta^2(t)\subset \bigcup_{\iota=1}^\infty I_\iota$.
   
   For scales larger than $k$, we will use that~\eqref{eq:imp1p} holds. Since for any scale $j$ and any $i\in\mathbb{Z}^d$ we have that 
   $$
      \bigcup\nolimits_{i': Q_j(i')\cap Q^\core_j(i)\neq \emptyset} Q_j^\enc(i')\subset Q_j^\couter(i),
   $$
   we obtain
   $$
      \bigcup\nolimits_{j>k}I_j 
      = \bigcup\nolimits_{j>k}\bigcup\nolimits_{i \in \mathcal{I}_j}Q_j^\enc(i)
      \subset \bigcup\nolimits_{j>k}\bigcup\nolimits_{i\colon Q_j^\core(i)\cap C_j \neq \emptyset}Q_j^\couter(i).
   $$
   This and~\eqref{eq:imp1p} give that $\bigcup\nolimits_{j>k}I_j$ does not intersect $A$, hence $\eta^2(t)\cap A\subset \bigcup_{\iota=1}^k I_\iota$.
\end{proof}

\subsection{Recursion}\label{sec:recursion}

Define 
$$
   \rho_k(i) = \PR\lr{Q_k^\core(i)\cap C_k\neq \emptyset}.
$$
Recall that $Q_k^\core(i)$ are disjoint for different $i\in\mathbb{Z}^d$, as defined in~\eqref{eq:qcore}.
Define also 
$$
   q_k = \exp\lr{-c\lr{R_k^\couter}^\frac{d+1}{2d+4}},
$$
where $c$ is the constant in Lemma~\ref{lem:goodbox} so that for any $k\in\mathbb{N}$ and $i\in\mathbb{Z}^d$, we have $\PR\lr{G_k(i)}\geq 1-q_k$.

By the definition of $C_k$ from~\eqref{eq:defck}, 
in order to have $Q_k^\core(i)\cap C_k\neq \emptyset$ it must happen that either 
\begin{equation}
   \exists j\in\mathbb{Z}^d \colon Q_{k-1}(j) \text{ is bad and } Q_{k-1}^\couter(j)\cap Q_k^\core(i)\neq\emptyset
   \label{eq:cond1}
\end{equation}
or
\begin{align}
   \exists x,y\in\mathbb{Z}^d\colon 
   x,y \in C_{k-1},\;
   x\in Q_{k}^\core(i),\;
   y \not\in Q_{k-1}(\iota),\;
   y \in \iota L_{k-1} + \ball{3R_{k-1}^\couter},
   \label{eq:cond2}
\end{align}
where $\iota$ is the unique number such that $x\in Q_{k-1}^\core(\iota)$.
The condition above holds by the following. If~\eqref{eq:cond1} does not happen, 
then there must exist a $x\in C_{k-1}\cap Q_k^\core(i)$ that was not treated at scale $k-1$; that is, $x\in C_{k-1}^\bad$. 
Then, by the definition of $C^\bad_{k-1}$, it must be the case that there exists a $y$ satisfying the conditions in~\eqref{eq:cond2}.
The values $x,y$ as in~\eqref{eq:cond2} must satisfy
\begin{equation}
   \frac{(10d^2-1/2)}{C_\FPP'}L_{k-1}\leq |x-y|\leq \frac{L_{k-1}}{2 C_\FPP } + 3 R_{k-1}^\couter.
   \label{eq:xy}
\end{equation}

\begin{lemma}\label{lem:measure}
   For any $k\geq 2$ and any $i\in \mathbb{Z}^d$, 
   define the super cell
   \begin{equation}
      Q_k^\super(i) = \bigcup_{x\in Q_k^\core(i)} \lr{x+\ball{3R_{k-1}^\couter+R_{k-1}}};
      \label{eq:qksuper}
   \end{equation}
   for $k=1$, set $Q_1^\super(i) = Q_1^\core(i)$.
   Then the event $\lrc{Q_k^\core(i)\cap C_k\neq \emptyset}$ is measurable with respect to 
   $Q_k^\super(i)$.
\end{lemma}
\begin{proof}
   The theorem is true for $k=1$ since $\lrc{Q_1^\core(i)\cap C_1\neq \emptyset}$ is equivalent to 
   $\lrc{Q_1^\core(i)\cap \eta^2(0)\neq \emptyset}$. 
   Our goal is to apply an induction argument to establish the lemma for $k>1$. 
   First note that, since the event that a box of scale $k-1$ is good is measurable with respect to passage times inside a ball of diameter $R_{k-1}^\couter$, 
   we have that 
   condition~\eqref{eq:cond1} is measurable with respect to the passage times inside 
   $$
      \bigcup_{x\in Q_k^\core(i)} \lr{x+\ball{2R_{k-1}^\couter}}.
   $$
   It remains to establish the measurability result for condition~\eqref{eq:cond2}. 
   Note that condition~\eqref{eq:cond2} gives the existence of a point $y$ in 
   $\bigcup_{x\in Q_k^\core(i)} \lr{x+\ball{3R_{k-1}^\couter}}$ such that $y\in C_{k-1}$.
   Let $j$ be the integer such that $y\in Q_{k-1}^\core(j)$. 
   Then the induction hypothesis gives that $\{Q_{k-1}^\core(j) \cap C_{k-1}\neq \emptyset\}$ is measurable with respect to the passage times 
   inside
   $$
      \bigcup_{x\in Q_{k-1}^\core(j)} \lr{x+\ball{3R_{k-2}^\couter+R_{k-2}}}
      \subset y+\ball{R_{k-1}}
      \subset Q_k^\super(i).
   $$
   Therefore, condition~\eqref{eq:cond2} is measurable with respect to the passage times inside $Q_k^\super(i)$.
\end{proof}

\begin{lemma}\label{lem:recursion}
   There exists a constant $c=c(d,\epsilon,\alpha,\Q)>0$ such that, 
   for all $k\in\mathbb{N}$ and all $i\in \mathbb{Z}^d$, 
   we have 
   $$
      \rho_k(i) \leq c k^{2d(d+2)}\sup_{j}\rho_{k-1}^2(j)+c k^{d(d+2)}q_{k-1}.
   $$
\end{lemma}
\begin{proof}
   From the discussion above, we have that $\rho_k(i)$ is bounded above by the probability that condition~\eqref{eq:cond1} occurs plus the probability that 
   condition~\eqref{eq:cond2} occurs. We start with condition~\eqref{eq:cond1}. Note that $Q_{k-1}^\core(j)$, for $j$ defined as in~\eqref{eq:cond1},
   must be contained inside 
   $$
      \bigcup_{x\in Q_k^\core(i)} \ball{R_{k-1}^\couter+R_{k-1}}.
   $$
   Therefore, there is a constant $c_0$ depending only on $d$ such that 
   the number of options for the value of $j$ is at most 
   $$
      c_0 \lr{\frac{R_k+R_{k-1}^\couter+R_{k-1}}{R_{k-1}}}^d 
      \leq c_0 \lr{\frac{2 R_k}{R_{k-1}}}^d 
      \leq c_1 k^{d(d+2)},
   $$
   for some constant $c_1$, where the first inequality comes from~\eqref{eq:rkrecursion} and the last inequality follows from~\eqref{eq:rkrel1}.
   Then, taking the union bound on the value of $j$, we obtain that the probability that condition~\eqref{eq:cond1} occurs is at most
   $$
      c_1 k^{d(d+2)} q_{k-1}.
   $$
   Now we bound the probability that condition~\eqref{eq:cond2} happens. 
   For any $z\in \mathbb{Z}^d$, let $\phi(z)\in\mathbb{Z}^d$ be such that 
   $z\in Q_{k-1}^\core(\phi(z))$. 
   We will need to estimate the number of 
   different values that $\phi(x)$ and $\phi(y)$ can assume. 
   Since $x\in Q_k^\core(i)$, we have that $\phi(x)$ can assume at most $\lr{\frac{L_k+2L_{k-1}}{L_{k-1}}}^d$ values.
   For $\phi(y)$, 
   first note that any point in $Q_{k-1}^\core(\phi(y))$ 
   must be contained inside $\phi(x)L_{k-1} + [-3C_\FPP' R_{k-1}^\couter - L_{k-1},3C_\FPP' R_{k-1}^\couter + L_{k-1}]^d$.
   Therefore, $Q_{k-1}^\core(\phi(y))$ 
   must be contained inside a cube of side length 
   $$
      L_k + 6C_\FPP' R_{k-1}^\couter + 2L_{k-1},
   $$
   and consequently there are at most 
   $$
      \lr{\frac{L_k + 6C_\FPP' R_{k-1}^\couter + 2L_{k-1}}{L_{k-1}}}^d
   $$
   possible values for $\phi(y)$.
   Letting $A_k$ be the number of ways of choosing the $Q_{k-1}^\core$ boxes containing $x,y$ according to condition~\eqref{eq:cond2}, 
   we obtain
   $$
      A_k \leq \lr{\frac{L_k+2L_{k-1}}{L_{k-1}}}^d\lr{\frac{L_k + 6C_\FPP' R_{k-1}^\couter + 2L_{k-1}}{L_{k-1}}}^d
      \leq c_2 k^{2d(d+2)},
   $$
   for some constant $c_2=c_2(d,\epsilon,\alpha,\Q)>0$, where the inequality follows from~\eqref{eq:rkrel1}.
   Now, given $x,y$, we want to give an upper bound for 
   $$
      \PR\lr{\lrc{Q_{k-1}^\core(\phi(x)) \cap C_{k-1}\neq \emptyset} \cap \lrc{Q_{k-1}^\core(\phi(y)) \cap C_{k-1}\neq \emptyset}}.
   $$
   From Lemma~\ref{lem:measure} we have that the event $\lrc{Q_{k-1}^\core(\phi(x)) \cap C_{k-1}}$ is measurable with respect to the passage times inside 
   $Q_{k-1}^\super(\phi(x))$.
   By the definition of $Q_{k-1}^\super$ in~\eqref{eq:qksuper}, for any $z\in Q_{k-1}^\super(\phi(x))$ we have
   \begin{align*}
      |x-z| 
      &\leq |x-\phi(x)L_{k-1}| + \lr{3R_{k-2}^\couter+R_{k-2}}\\
      &\leq \frac{L_{k-1}}{2C_\FPP} + 4R_{k-2}^\couter\\
      &\leq \frac{R_{k-1}}{20d^2} + \frac{R_{k-1}}{500d^2(k-1)^d},\\
      &\leq \frac{R_{k-1}}{19d^2},
   \end{align*}
   where we related $R_{k-2}^\couter$ and $R_{k-1}$ via~\eqref{eq:rkrecursion}.
   Since by~\eqref{eq:xy} and~\eqref{eq:rl2} we have 
   $$
      |x-y| > \frac{(10d^2-1/2)}{C'_\FPP}L_{k-1}
         \geq \left(\frac{10d^2-1/2}{C'_\FPP}\right)\left(\frac{C^2_\FPP R_{k-1}}{10d^2C'_\FPP}\right)
         \geq \left(1-\frac{1}{20d^2}\right)\left(\frac{R_{k-1}}{d^2}\right),
   $$
   where the last inequality follows from~\eqref{eq:cfpprel},
   we then obtain 
   $Q_{k-1}^\super(\phi(x)) \cap Q_{k-1}^\super(\phi(y))= \emptyset$. 
   This gives that the events $\lrc{Q_{k-1}^\core(\phi(x)) \cap C_{k-1}}$ and $\lrc{Q_{k-1}^\core(\phi(y)) \cap C_{k-1}}$ are independent, yielding
   $$
      \rho_k(i) \leq c_1 k^{d(d+2)} q_{k-1} + A_k \sup_j \rho_{k-1}^2(j).
   $$
\end{proof}

In the lemma below, recall that $\eta^2(0)$ is given by adding each vertex of $\mathbb{Z}^d$ with probability $p$, independently of one another.
Also let $\bar\rho$ be such that $\bar\rho\geq \sup_{j}\rho_{1}(j)$.
\begin{lemma}\label{lem:rhoexpanded}
   Fix any positive constant $a$. We can set $L_1$ large enough and then $p$ small enough, both depending on $a,\alpha,\epsilon, d$ and $\Q$,  such that
   for all $k\in\mathbb{N}$ and all $i\in \mathbb{Z}^d$, 
   we have 
   $$
      \rho_k(i) \leq \exp\lr{- a 2^{k}}.
   $$
\end{lemma}
\begin{proof}
   For $k=1$, then $\rho_k(i)$ is bounded above by the probability that 
   $\eta^2(0)$ intersects $Q_k^\core(i)$. Once $L_1$ has been fixed, this probability
   can be made arbitrarily small by setting $p$ small enough.
   
   Now we assume that $k\geq 2$.
   We will expand the recursion in Lemma~\ref{lem:recursion}.
   Using the same constant $c$ as in Lemma~\ref{lem:recursion}, define 
   $$
      \bar q_{k-1} = c k^{d(d+2)}q_{k-1}
      \quad \text{ for all $k\geq 2$}.
   $$
   Now fix $k$, set $A_{-1}=1$, and define for $\ell=0,1,\ldots,k-1$ 
   \begin{equation}
      A_\ell 
         = A_{\ell-1} 2^{2^\ell-1} \lr{c(k-\ell)^{2d(d+2)}}^{2^{\ell}}
         = \prod_{m=0}^\ell 2^{2^m-1} \lr{c(k-m)^{2d(d+2)}}^{2^{m}}.
      \label{eq:al}
   \end{equation}
   With this, the recursion in Lemma~\ref{lem:recursion} can be written as 
   \begin{align*}
      \rho_k(i) 
      & \leq A_0 \sup_{j}\rho_{k-1}^2(j)+\bar q_{k-1}\\
      & \leq A_0\,2\lr{\lr{c (k-1)^{2d(d+2)}}^2 \sup_{j}\rho_{k-2}^4(j)+\bar q_{k-2}^2}+\bar q_{k-1}\\
      & = A_1\sup_{j}\rho_{k-2}^4(j)+2A_0 q_{k-2}^2+\bar q_{k-1},
   \end{align*}
   where in the second inequality we used that $(x+y)^m \leq 2^{m-1}\lr{x^m+y^m}$ for all $x,y\in\mathbb{R}$ and $m\in\mathbb{N}$.
   Iterating the above inequality, we obtain
   \begin{align}
      \rho_k(i) 
      \leq A_{k-2} \sup_{j}\rho_{1}^{2^{k-1}}(j)+\sum_{\ell=1}^{k-1}2^{2^{\ell-1}-1}A_{\ell-2}\bar q_{k-\ell}^{2^{\ell-1}}
      = A_{k-2} \bar \rho^{2^{k-1}}+\sum_{\ell=1}^{k-1}2^{2^{\ell-1}-1}A_{\ell-2}\bar q_{k-\ell}^{2^{\ell-1}}.
      \label{eq:it}
   \end{align}
   We now claim that 
   \begin{equation}
      A_\ell \leq \lr{4 c^2\, \lr{3(k-\ell)}^{5d(d+2)}}^{2^\ell} \quad \text{for all $\ell=0,1,\ldots,k-1$}.
      \label{eq:recclaim}
   \end{equation}
   We can prove~\eqref{eq:recclaim} by induction on $\ell$. Note that $A_0$ does satisfy the above inequality.
   Then, using the induction hypothesis and the recursive definition of $A_\ell$ in~\eqref{eq:al}, we have 
   \begin{align*}
      A_\ell 
      &= A_{\ell-1} 2^{2^\ell-1} \lr{c(k-\ell)^{2d(d+2)}}^{2^{\ell}}\\
      &\leq \lr{4c^2 \lr{3(k-\ell+1)}^{5d(d+2)}}^{2^{\ell-1}}2^{2^\ell-1} \lr{c(k-\ell)^{2d(d+2)}}^{2^{\ell}}\\
      &= \frac{1}{2}\lr{2c \lr{3(k-\ell+1)}^{5d(d+2)/2}\, 2\,c(k-\ell)^{2d(d+2)}}^{2^{\ell}}.
   \end{align*}
   Now we use that $(x+1)^{5/2}\leq 6x^{3}$ for all $x\geq 1$, which yields
   \begin{align*}
      A_\ell 
      &\leq \frac{1}{2}\lr{2c \lr{3^{5/2}6(k-\ell)^3}^{d(d+2)}\, 2\,c(k-\ell)^{2d(d+2)}}^{2^{\ell}}\\
      &= \frac{1}{2}\lr{4c^2 \lr{3^{1/2}6^{1/5}(k-\ell)}^{5d(d+2)}}^{2^{\ell}}
      \leq  \frac{1}{2}\lr{4c^2 \lr{3(k-\ell)}^{5d(d+2)}}^{2^{\ell}},
   \end{align*}
   establishing~\eqref{eq:recclaim}. Plugging~\eqref{eq:recclaim} into~\eqref{eq:it}, we obtain
   \begin{align}
      \rho_k(i) 
      &\leq \lr{4c^2\,6^{5d(d+2)}}^{2^{k-2}}\bar \rho^{2^{k-1}}
         +\sum_{\ell=1}^{k-1}2^{2^{\ell-1}-1}\lr{4c^2\lr{3(k-\ell+2)}^{5d(d+2)}}^{2^{\ell-2}}\bar q_{k-\ell}^{2^{\ell-1}}\nonumber\\
      &\leq \lr{2c\,6^{5d(d+2)/2}\bar \rho}^{2^{k-1}}
         +\frac{1}{2}\sum_{\ell=1}^{k-1}\lr{4c\lr{3(k-\ell+2)}^{5d(d+2)/2}\bar q_{k-\ell}}^{2^{\ell-1}}.
         \label{eq:rhokfinal}
   \end{align}
   Given a value of $L_1$, for all small enough $p$ we obtain that $\bar\rho$ is sufficiently small to yield 
   $$
      \lr{2c\,6^{5d(d+2)/2}\bar \rho}^{2^{k-1}} \leq \frac{1}{2}\exp\lr{-a 2^{k}}.
   $$
   Now we turn to the second term in~\eqref{eq:rhokfinal}.
   Note that for small enough $\epsilon$, we have $\epsilon \lambda R_k^\enc \geq \epsilon \lambda \exp\left(\frac{1+c_1}{2\epsilon}\right)R_k> R_k$. Thus, 
   from Lemma~\ref{lem:goodbox}, we have that $q_{k-\ell}\leq \exp\lr{-c R_{k-\ell}^\frac{d+1}{2d+4}}$, for some constant $c=c(\alpha,\epsilon,d,\Q)>0$.
   We have from the relations~\eqref{eq:rkrel1} and~\eqref{eq:rkrel2} that $R_j\leq c_1 c_2^j (j!)^{d}L_1$ for positive constants 
   $c_1,c_2$. Therefore, for any $k\geq \ell$, we have that 
   \begin{align*}
      \lr{4c\lr{3(k-\ell+2)}^{5d(d+2)/2}\bar q_{k-\ell}}^{2^{\ell-1}}
      &\leq \exp\lr{-c_3 \lr{(k-\ell)!}^\frac{d(d+1)}{2d+4}2^{\ell-1}c_2^\frac{(k-\ell)(d+1)}{2d+4}L_1^\frac{d+1}{2d+4}}\\
      &\leq \exp\lr{-c_3 2^{k}L_1^\frac{d+1}{2d+4}},
   \end{align*}
   where in the last step we use that $c_2^\frac{(d+1)}{2d+4}\geq c_2^{1/3}\geq 2$.
   Hence, for sufficiently large $L_1$ we obtain
   $$
      \rho_k(i) 
      \leq \frac{1}{2}\exp\lr{-a 2^{k}}
         +\frac{k}{2}\exp\lr{-c_3 2^{k}L_1^\frac{d+1}{2d+4}} 
      \leq \exp\lr{-a 2^{k}}.
   $$
\end{proof}

\subsection{Multiscale paths of infected sets}\label{sec:multiscalepaths}

Let $x\in\mathbb{Z}^d$ be a fixed vertex. We say that $\Gamma=(k_1,i_1),(k_2,i_2),\ldots,(k_\ell,i_\ell)$ 
is a \emph{multi-scale path} from $x$ if 
$x\in Q_{k_1}^\enc(i_1)$, and for each $j\in\{2,3,\ldots,\ell\}$ we have $Q_{k_j}^\enc(i_j)\cap Q_{k_{j-1}}^\enc(i_{j-1})\neq\emptyset$. 
Hence, 
$$
   \bigcup\nolimits_{(k,i)\in\Gamma} Q_{k}^\enc(i) \text{ is a connected subset of $\mathbb{Z}^d$ and contains $x$}.
$$
Given such a path, we say that 
the \emph{reach} of $\Gamma$ is given by $\sup\lrc{|z-x| \colon z \in \bigcup\nolimits_{(k,i)\in\Gamma} Q_{k}^\enc(i)}$, 
that is, the distance between $x$ and the furthest away point of $\Gamma$.
We will only consider paths such that $Q_{k_j}^\enc(i_j)\subset I_{k_j}$. 
Recall 
the way the sets $I_\kappa$ are constructed from 
$C_\kappa\setminus C_\kappa^\bad$, which is defined in~\eqref{eq:cbad}.
Then for any two $(k,i),(k',i')\in\Gamma$ with
$k=k'$ we have $Q_{k}^\enc(i)\cap Q_{k'}^\enc(i')=\emptyset$.
Therefore, we impose the additional restriction that on any multi-scale path $\Gamma=(k_1,i_1),(k_2,i_2),\ldots,(k_\ell,i_\ell)$ we have 
$k_j\neq k_{j-1}$ for all $j\in\{2,3,\ldots,\ell\}$. 

Now we introduce a subset $\tilde \Gamma$ of $\Gamma$ as follows. 
For each $k\in\mathbb{N}$ and $i\in\mathbb{Z}^d$, define
$$
   Q_k^\neigh(i) = i L_k + \ball{\tfrac{11}{10}R_k^\couter}
   \quad\text{and}\quad
   Q_k^\neighs(i) = i L_k + \ball{\tfrac{6}{5}R_k^\couter}.
$$
Note that $Q_k^\couter(i)\subset Q_k^\neigh(i)\subset Q_k^\neighs(i)$.
Let $\kappa_1 > \kappa_2>\cdots$ be an ordered list of the scales that appear in cells of $\Gamma$. 
The set $\tilde\Gamma$ will be constructed in steps, one step for each scale. First, add to $\tilde\Gamma$ all cells of $\Gamma$ of scale $\kappa_1$. 
Then, for each $j\geq2$, after having decided which cells of $\Gamma$ of scale at least $\kappa_{j-1}$ we add to $\tilde\Gamma$, 
we add to $\tilde \Gamma$ all cells $(k,i)\in \Gamma$ of 
scale $k=\kappa_j$ such that $Q_k^\neigh(i)$ does not intersect $Q_{k'}^\neigh(i')$ for each $(k',i')$ already added to $\tilde\Gamma$. 
Recall that, from the definition of $C_k^\bad$ in~\eqref{eq:cbad}, two cells 
$(k,j),(k,j')$ of the same scale that are part of $I_k$ must be such that 
$$
   Q_k(j') \not\subset jL_k + \ball{3R_k^\couter}.
$$
This gives that $|jL_k-j'L_k| \geq 3 R_k^\couter-R_k$, which implies that 
$Q_k^\neighs(j)$ and $Q_k^\neighs(j')$ do not intersect.

The idea behind the definitions above is that we will look at ``paths'' of multi-scale cells such that two neighboring cells in the path 
are such that their $Q^\neighs$ regions intersect, and any two cells in the path have disjoint $Q^\neigh$ regions. The first property limits the number of 
cells that can be a neighbor of a given cell, allowing us 
to control the number of such paths, while the second property allows us to argue that the encapsulation procedure 
behaves more or less independently for different cells of the path.
\begin{lemma}\label{lem:paths}
   Let $\Gamma=(k_1,i_1),(k_2,i_2),\ldots,(k_\ell,i_\ell)$ be a multi-scale path starting from $x$. 
   Then, the subset $\tilde \Gamma$ defined above is such that 
   $$
      \bigcup\nolimits_{(k,i)\in\tilde\Gamma}Q_{k}^\neighs(i) \text{ is a connected subset of $\mathbb{Z}^d$}.
   $$
   Furthermore, any point $y\in \bigcup_{(k,i)\in\Gamma}Q_k^\enc(i)$ must belong to $\bigcup_{(k,i)\in\tilde\Gamma}Q_k^\neighs(i)$.
\end{lemma}
\begin{proof}
   Let $\Upsilon$ be an arbitrary subset of $\tilde\Gamma$ with $\Upsilon \neq \tilde\Gamma$. The first part of the lemma follows by showing that 
   \begin{equation}
      \text{there exists $(\kappa,\iota)\in \tilde \Gamma\setminus \Upsilon$ such that 
         $Q_\kappa^\neighs(\iota)$ intersects $\bigcup\nolimits_{(k,i)\in \Upsilon}Q_k^\neighs(i)$}.
      \label{eq:claiming}
   \end{equation}
   Define
   $$
      \Upsilon^\neigh = \lrc{(k,i)\in\Gamma \colon 
      \text{$Q_{k}^\neigh(i)$ intersects $Q_{k'}^\neigh(i')$ for some $(k',i')\in\Upsilon$}}.
   $$
   Clearly, 
   $\Upsilon^\neigh \supset\Upsilon$, 
   and since $\{Q_k^\neigh(i) \colon (k,i)\in\tilde\Gamma\}$ is by definition a collection of 
   disjoint sets, we have that 
   $$
      \text{all elements of $\Upsilon^\neigh\setminus \Upsilon$ do not belong to $\tilde\Gamma$.}
   $$
   Recall that $\tilde\Gamma\neq\Upsilon$, and since no element of $\tilde\Gamma\setminus\Upsilon$ was added to $\Upsilon^\neigh$, 
   we have that $\Upsilon^\neigh\neq \Gamma$.
   Using that $\bigcup_{(k,i)\in\Gamma} Q_{k}^\enc(i)$ is a connected set, we obtain a value 
   $$
      (k,i)\in\Gamma\setminus \Upsilon^\neigh \text{ for which $Q_{k}^\enc(i)$ intersects 
      $Q_{k'}^\enc(i')\subset Q_{k'}^\neigh(i')$ for some $(k',i')\in \Upsilon^\neigh$.}
   $$
   Refer to Figure~\ref{fig:upsilon} for a schematic view of the definitions in this proof.
   \begin{figure}[htbp]
      \begin{center}
         \includegraphics[scale=.6]{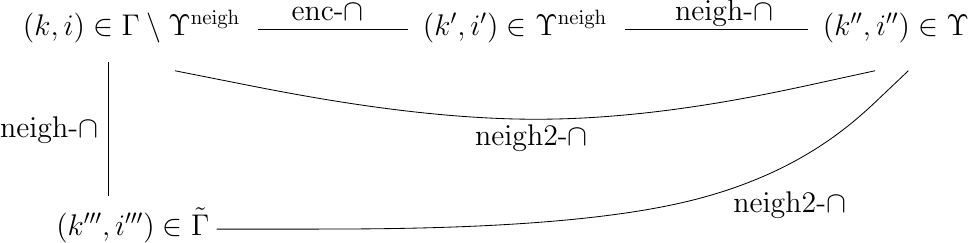}
      \end{center}\vspace{-.5cm}
      \caption{Illustration of the relations between the variables in the proof of Lemma~\ref{lem:paths}. 
         A line from $(\kappa,\iota)$ to $(\kappa',\iota')$ labeled *-$\cap$ indicates that 
         $Q_\kappa^*(\iota)\cap Q_{\kappa'}^*(\iota')\neq\emptyset$.}
      \label{fig:upsilon}
   \end{figure}
   Let 
   $(k',i')$ be the cell of $\Upsilon^\neigh$ for which $Q_k^\enc(i)$ intersects $Q_{k'}^\enc(i')$.
   Since $(k',i')\in \Upsilon^\neigh$, let $(k'',i'')$ be the element of $\Upsilon$ of largest scale for which $Q_{k'}^\neigh(i')$ intersects $Q_{k''}^\neigh(i'')$;
   if $(k',i')\in \Upsilon$, then  $(k'',i'')=(k',i')$.
   We obtain that  
   \begin{equation}
      \text{the distance according to $|\cdot|$ between $Q_{k''}^\neigh(i'')$ and $Q_k^\enc(i)$ }
      \leq R_{k'}^\enc+\tfrac{11}{10}R_{k'}^\couter. 
      \label{eq:dist}
   \end{equation}
   By the construction of 
   $\tilde\Gamma$, and the fact that $(k'',i'')$ was set as the element of largest scale satisfying $Q_{k''}^\neigh(i'')\cap Q_{k'}^\neigh(i')\neq\emptyset$, 
   we must have that 
   $$
      k''> k' \quad\text{or}\quad (k'',i'')=(k',i').
   $$
   In the former case, the distance in~\eqref{eq:dist} is bounded above by $2R_{k''-1}^\couter$, 
   while in the latter case the distance is zero. So we assume that the distance 
   between $Q_{k''}^\neigh(i'')$ and $Q_k^\enc(i)$ is at most $2 R_{k''-1}^\couter$, which yields that 
   $Q_k^\enc(i)$ intersects $Q_{k''}^\neighs(i'')$. Therefore, if $(k,i)\in\tilde\Gamma$, we have~\eqref{eq:claiming} and we are done.
   When $(k,i)\not\in\tilde\Gamma$, 
   take the cell $(k''',i''')\in\tilde\Gamma$ of largest scale such that $Q_k^\neigh(i)$ intersects $Q_{k'''}^\neigh(i''')$ and, by 
   the construction of $\tilde\Gamma$, we have
   $$
      k'''>k.
   $$
   We obtain that $(k''',i''')\not\in\Upsilon$, otherwise it would imply that $(k,i)\in\Upsilon^\neigh$ violating the definition of $(k,i)$.
   The distance between $Q_{k'''}^\neigh(i''')$ and $Q_{k''}^\neigh(i'')$ is at most 
   $$
      \tfrac{6}{5}R_{k}^\couter+R_{k}^\enc+2R_{k''-1}^\couter
      \leq 2 \lr{R_{k'''-1}^\couter+R_{k''-1}^\couter}
      \leq \frac{1}{20}\lr{R_{k''}^\couter+R_{k'''}^\couter}.
   $$
   Therefore, we have that $Q_{k''}^\neighs(i'')$ intersects $Q_{k'''}^\neighs(i''')$, establishing~\eqref{eq:claiming} and concluding the first part of the proof.
   
   For the second part, take $y$ to be a point of $Q_{k}^\enc(i)$ with $(k,i)\in\Gamma$. 
   If $(k,i)\in\tilde\Gamma$, then the lemma follows. 
   Otherwise, let $(\kappa,\iota)$ be the cell of largest scale in $\tilde\Gamma$ such that $Q_{\kappa}^\neigh(\iota)\cap Q_{k}^\neigh(i)\neq\emptyset$. By the 
   construction of $\tilde\Gamma$, we have that $\kappa>k$. The distance between $y$ and $Q_{\kappa}^\neigh(\iota)$ is at most 
   $$
      R_{k}^\enc + R_{k}^\neigh \leq 2R_{\kappa-1}^\neigh \leq \frac{1}{10} R_\kappa,
   $$
   which gives that $y\in Q_\kappa^\neighs(\iota)$.
\end{proof}

Now we define the type of multi-scale paths we will consider.
\begin{definition}
   Given $x\in\mathbb{Z}^d$ and $m>0$, 
   we say that $\Gamma=(k_1,i_1),(k_2,i_2),\ldots,(k_\ell,i_\ell)$ is a \emph{well separated path} of reach $m$ starting from $x$ if 
   all the following hold:
   \begin{align*}
      (i)\;&\text{$x \in Q_{k_1}^\neighs(i_1)$,} \\
      (ii)\;&\text{for any $j\in\{2,3,\ldots,\ell\}$ we have that $Q_{k_j}^\neighs(i_j)$ intersects 
         $Q_{k_{j-1}}^\neighs(i_{j-1})$,} \\
      (iii)\;&\text{for any $j,\iota\in\{1,2,\ldots,\ell\}$ with $|j-\iota|\geq 2$ we have $Q_{k_j}^\neighs(i_j)$ does not intersect 
         $Q_{k_\iota}^\neighs(i_\iota)$,} \\
      (iv)\;&\text{for any distinct $j,\iota\in\{1,2,\ldots,\ell\}$ we have that $Q_{k_j}^\neigh(i_j)$ does not intersect 
         $Q_{k_\iota}^\neigh(i_\iota)$,} \\
      (v)\;&\text{for any $j\in\{2,3,\ldots,\ell\}$ , we have $k_j\neq k_{j-1}$,}\\
      (vi)\;&\text{and the point of $Q_{k_\ell}^\neighs(i_\ell)$ that is furthest away from $x$ is of distance $m$ from $x$.}
   \end{align*}
\end{definition}

We say that a well separated path $\Gamma$ is \emph{infected} if for all $(k,i)\in\Gamma$ we have $Q_{k}^\enc(i)\subset I_k$.
If the origin is separated from infinity by $\eta^2$, 
then there must exist a multi-scale path for which the union of the $Q_k^\enc(i)$ over the cells $(k,i)$ in the path contains the set occupied by $\eta^2$ that separates the 
origin from infinity. Then Lemma~\ref{lem:paths} gives the existence of a well separated path for which the union of the $Q_k^\neighs(i)$ over $(k,i)$ in the path separates the origin from infinity. 
\begin{lemma}\label{lem:givenscales}
   Fix any positive constant $c$. We can set $L_1$ large enough and then $p$ small enough, 
   both depending only on $c,\alpha,d, \epsilon$ and $\Q$, so that the following holds. 
   For any integer $\ell\geq1$, any given collection of (not necessarily distinct) integer numbers $k_1,k_2,\ldots,k_\ell$,
   and any vertex $x\in\mathbb{Z}^d$, we have
   \begin{align*}
      &\PR\lr{\text{$\exists$ a well separated path $\Gamma=(k_1,i_1),(k_2,i_2),\ldots,(k_\ell,i_\ell)$ from $x$ that is infected}}\\
         &\leq \exp\lr{-c \sum\nolimits_{j=1}^\ell 2^{k_j}}.
   \end{align*}
\end{lemma}
\begin{proof}
   For any $j$, since the path is infected we have $Q_{k_j}^\enc(i_j)\subset I_{k_j}$. This gives that there exists 
   $\tilde i_j$ such that $Q_{k_j}^\core(\tilde i_j)\cap Q_{k_j}(i_j)\cap C_{k_j}\neq \emptyset$. 
   From Lemma~\ref{lem:measure}, we have that the event $\lrc{Q_{k_j}^\core(\tilde i_j)\cap C_{k_j}\neq \emptyset}$
   is measurable with respect to the passage times inside $Q_{k_j}^\super(\tilde i_j)\subset Q_{k_j}^\neigh(i_j)$. 
   Also, the number of choices for $\tilde i_j$ is at most some constant $c_1$, depending only on $d$.
   Since $\{Q_{k_j}^\neigh(i_j)\}_{j=1,\ldots,\ell}$ is a collection of disjoint sets, if we fix 
   the path $\Gamma=(k_1,i_1),(k_2,i_2),\ldots,(k_\ell,i_\ell)$, and take the union bound over the choices of 
   $\tilde i_1, \tilde i_2,\ldots,\tilde i_\ell$, we have from Lemma~\ref{lem:rhoexpanded} that 
   $$
      \PR\lr{\text{$\Gamma$ is an infected path}} \leq c_1^\ell\exp\lr{-a \sum\nolimits_{j=1}^\ell 2^{k_j}},
   $$
   where $a$ can be made as large as we want by properly setting $L_1$ and $p$.
   It remains to bound the number of well separated paths that exist starting from $x$. 
   Since $x \in Q_{k_1}^\neighs(i_1)$, the number of ways to choose the first cell 
   is at most $\lr{\frac{12}{5}C_\FPP' R_{k_1}^\couter}^d$.
   Consider a $j$ such that $k_j>k_{j+1}$. We have that $Q_{k_{j}}^\neighs(i_{j})$ must intersect $Q_{k_{j+1}}^\neighs(i_{j+1})$, 
   which gives that 
   \begin{align*}
      Q_{k_{j+1}}^\core(i_{j+1})
      &\subset i_j L_{k_j}+\ball{\tfrac{6}{5}R_{k_j}^\couter +\tfrac{6}{5}R_{k_{j+1}}^\couter+R_{k_{j+1}}}\\
      &\subset i_j L_{k_j}+\ball{\tfrac{7}{5}R_{k_j}^\couter}\\
      &\subset i_j L_{k_j}+\lrb{-\tfrac{7C_\FPP'}{5}R_{k_j}^\couter,\tfrac{7C_\FPP'}{5}R_{k_j}^\couter}^d.
   \end{align*}
   Hence, the number of ways to choose $i_{j+1}$ given $(k_j,i_j)$ and $k_{j+1}$ is at most 
   $$
      \lr{\frac{14C_\FPP'}{5}\frac{R_{k_j}^\couter}{L_{k_{j+1}}}}^d.
   $$
   Therefore, we have that 
   \begin{align*}
      &\PR\lr{\text{$\exists$ a well separated path $\Gamma=(k_1,i_1),(k_2,i_2),\ldots,(k_\ell,i_\ell)$ from $x$ that is infected}}\\
         &\leq c_1^\ell\lr{\frac{12}{5}C_\FPP' R_{k_1}^\couter}^d \prod_{j=1}^\ell \lr{\frac{14C_\FPP'}{5}\frac{R_{k_j}^\couter}{L_{k_{j+1}}}}^{2d}\exp\lr{-a2^{k_j}}\\
         &\leq c_1^\ell\prod_{j=1}^\ell \exp\lr{c_2 k_j\log(k_j) -a2^{k_j}}\\
         &\leq \prod_{j=1}^\ell \exp\lr{-a2^{k_j-1}},
   \end{align*}
   where the second inequality follows for some $c_2=c_2(d,\alpha,\epsilon,\Q)$ by the value of $R_k^\couter$ from~\eqref{eq:rkrel2},
   and the last inequality follows by setting $a$ large enough and such that $a\geq 2c$.
\end{proof}

For the lemma below, define the event 
\begin{align}
   E_{\kappa,r} = \{&\text{there exists a well separated path from the origin that is infected,}\nonumber\\
   &\text{has only cells of scale smaller than $\kappa$, and has reach at least $C_\FPP r$}\}.
   \label{eq:defe}
\end{align}
Let $E_{\infty,r}$ be the above event without the restriction that all scales must be smaller than $\kappa$.
Below we restrict to $r>3$ just to ensure that $\log\log r >0$.
\begin{proposition}\label{pro:path}
   Fix any positive constant $c$, any $r>3$ and any time $t\geq0$. We can set $L_1$ large enough and then $p$ small enough, 
   both depending only on $c,\alpha,d,\epsilon$ and $\Q$, so that there exists a positive constant $c'$ depending only on $d$ for which 
   $$
      \PR\lr{E_{\infty,r}}\leq \exp\lr{-c r^\frac{c'}{\log\log r}}.
   $$
\end{proposition}
\begin{proof}
   Let $A_r$ be the set of vertices of $\mathbb{Z}^d$ of distance at most 
   $C_\FPP r$ from the origin.
   Set $\delta_r = \frac{1}{(d+3)\log\log r}$ and $\kappa= \delta_r \log r$.
   For any large enough $a$ depending on $L_1$ and $p$, we have
   \begin{align*}
      &\PR\lr{\exists (k,i) \colon k\geq \kappa \text{ and } Q_k^\neighs(i)\cap A_r \neq \emptyset \text{ and } Q_k^\core(i)\cap C_k\neq\emptyset} \\
      &\leq \sum_{k=\kappa}^\infty \sum_{i\colon Q_k^\neighs(i)\cap A_r \neq \emptyset} \PR\lr{Q_k^\core(i)\cap C_k\neq \emptyset}\\
      &\leq \sum_{k=\kappa}^\infty \lr{\frac{2C_\FPP' \lr{r+\frac{6}{5}R_k^\couter+R_k}}{L_k}}^d \exp\lr{-a 2^k},
   \end{align*}
   where in the last inequality we use Lemma~\ref{lem:rhoexpanded}. Since $a$ above can be chosen as large as needed (by requiring that $L_1$ is large enough and $p$ is small enough), 
   we can choose a large enough $a$ so that $\lr{\frac{2C_\FPP' \lr{r+\frac{6}{5}R_k^\couter+R_k}}{L_k}}^d\leq \exp\lr{a2^{k-1}}$ for all $k$, yielding
   \begin{align*}
      &\PR\lr{\exists (k,i) \colon k\geq \kappa \text{ and } Q_k^\neighs(i)\cap A_r \neq \emptyset \text{ and } Q_k^\core(i)\cap C_k\neq\emptyset} \\
      &\leq \sum_{k=\kappa}^\infty \exp\lr{-a 2^{k-1}}
       \leq 2\exp\lr{-a 2^{\kappa-1}}.
   \end{align*}
   If the event above does not happen, then Lemma~\ref{lem:inf} gives that $\eta^2(t)\cap A_r \subset \bigcup_{j=1}^{\kappa-1}I_j$. 
   Hence, 
   $$
      \PR\lr{E_{\infty,r}} \leq 2\exp\lr{-a 2^{\kappa-1}} + \PR\lr{E_{\kappa,r}}.
   $$
   Let $\Gamma$ be a well separated path from the origin, with all cells of scale smaller than $\kappa$, and which has reach at least $r$.
   Define $m_k(\Gamma)$ to be the number of cells of scale $k$ in $\Gamma$. 
   Since $\Gamma$ must contain at least one cell for which its $Q^\neighs$ region is not contained in $A_r$, we have
   $$
      C_\FPP r \leq \sum_{k=1}^{\kappa-1} m_k(\Gamma) \frac{12}{5} R_k^\couter.
   $$
   Because of the type of bounds derived in Lemma~\ref{lem:givenscales}, it will be convenient to rewrite the inequality above so that the term 
   $\sum_{k=1}^{\kappa-1}2^k$ appears. Note that using~\eqref{eq:rkrel2} 
   we can set a constant $c_0\geq2$ such that $R_j^\couter\leq c_0^j (j!)^{d+2}L_1$ for all $j\geq 1$, which gives
   $$
      C_\FPP r \leq  \sum_{k=1}^{\kappa-1} m_k(\Gamma) 2^k  \frac{12 (c_0-2)^k(k!)^{d+2}L_1}{5}
      \leq \frac{12 (c_0-2)^\kappa(\kappa!)^{d+2}L_1}{5} \sum_{k=1}^{\kappa-1} m_k(\Gamma) 2^k.
   $$
   For any $\Gamma$, define $\phi(\Gamma)=\sum_{k=1}^{\kappa-1} m_k(\Gamma) 2^k$.
   We can then split the sum over all paths according to the value of $\phi(\Gamma)$ of the path. Using this, Lemma~\ref{lem:givenscales}, and the fact that 
   $\phi(\Gamma)\geq \frac{5 C_\FPP r}{12 c' (\kappa!)^{d+2}L_1}$, we have
   \begin{align*}
      \PR\lr{E_{\kappa,r}}
      \leq \sum_{m\geq \frac{5 C_\FPP r}{12 c' (\kappa!)^{d+2}L_1}}^\infty \exp\lr{-c'' m} A_m,
   \end{align*}
   where $A_m$ is the number of ways to fix $\ell$ and set $k_1,k_2,\ldots,k_\ell$ such that $\phi(\Gamma)=\sum_{j=1}^\ell 2^{k_j}=m$, and $c''$ is the constant 
   in Lemma~\ref{lem:givenscales}.
   For each choice of $\ell,k_1,k_2,\ldots,k_\ell$, we can define a string from $\{0,1\}^m$ by taking $2^{k_1}$ consecutive 0s, $2^{k_2}$ consecutive 
   1s, $2^{k_3}$ consecutive 0s, and so on and so forth. Note that each string is mapped to at most one choice of $\ell,k_1,k_2,\ldots,k_\ell$. 
   Therefore, $A_m\leq 2^m$, the number of strings in $\{0,1\}^m$. The proof is completed since 
   $c''$ can be made arbitrarily large by setting $L_1$ large enough and then $p$ small enough, and 
   $\frac{5 C_\FPP r}{12 c' (\kappa!)^{d+2}L_1}\geq \frac{5 C_\FPP r}{12 c' \kappa^{(d+2)\kappa}L_1}\geq \frac{5 C_\FPP r^\frac{1}{d+3}}{12 c' L_1}$. 
\end{proof}

\subsection{Completing the proof of Theorem~\ref{thm:fpp2}}\label{sec:completing}
\begin{proof}[Proof of Theorem~\ref{thm:fpp2}]
   We start showing that $\eta^1$ grows indefinitely with positive probability.
   Let $e_1=(1,0,0,\ldots,0)\in\mathbb{Z}^d$. Any set of vertices that separates the origin from infinity must contain a vertex of the form $b e_1$ for 
   some nonnegative integer $b$. 
   For any $b$ and $t\geq0$, let 
   $$
      f_b(t) = \PR\lr{\eta^2(t) \text{ contains $b e_1$ and separates the origin from infinity}}.
   $$
   For the moment, we assume that $b$ is larger than some fixed, large enough value $b_0$.
   Recall that $\ball{r}\subseteq [- C_\FPP' r, C_\FPP' r]^d$, which gives that 
   $b e_1 + \ball{\frac{b}{2C'_\FPP}}$ does not contain the origin.
   Hence, in order for $\eta_2(t)$ to contain $b e_1$ and separate
   the origin from infinity, $\eta_2(t)$ must contain at least a vertex of distance (according to the norm $|\cdot |$) 
   greater than $\frac{b}{2C'_\FPP}$ from $b e_1$. 
   When $\eta_2(t)$ separates the origin from infinity, it must contain a set of sites that form a connected component according to the $\ell_\infty$ norm, which itself separates the origin from infinity and 
   contains a vertex of distance (now according to the norm $|\cdot |$) greater than $\frac{b}{2C'_\FPP}$ from $b e_1$. 
   This connected component implies the occurence of the event in Proposition~\ref{pro:path}, hence
   $$
      f_b(t) \leq \exp\lr{-c \lr{\frac{b}{2C'_\FPP}}^\frac{c'}{\log\log b}}.
   $$
   Note that, as needed, the bound above does not depend on $t$; this will allows us to derive a bound for the survival of $\eta^1$ that is uniformly bounded away from $0$ as $t$ grows to infinity.
   Note also that the constant $c$ from Proposition~\ref{pro:path} can be made arbitrarily large by setting $L_1$ and $p$ properly. Therefore,
   $\sum_{b=b_0}^\infty f_b(t)$ can be made smaller than $1$, and in fact goes to zero with $b_0$. 
   Regarding the case $b\leq b_0$, for each $k\geq 1$ let $\mathcal{K}_k$ be the set of $(k,i)$ such that $R_k^\couter(i)\cap \{e_1,2e_1,\ldots,b_0 e_1\}$.
   Note that there exists a constant $c_b$ depending on $b_0$ such that the cardinality of $\mathcal{K}_k$ is at most $c_b$ for all $k$.
   Then, using Lemma~\ref{lem:rhoexpanded}, we have
   $$
      \PR\lr{\exists k \colon \left(\bigcup\nolimits_{(k,i)\in\mathcal{K}_k}Q_k^\core(i)\right) \cap C_k \neq \emptyset}
      \leq \sum_{k\geq 1} c_b \exp\left(-a 2^k\right),
   $$
   which can be made arbitrarily small since $a$ can be made large enough by choosing $L_1$ large and $p$ small. This concludes this part of the proof, since 
   $$
      \PR\lr{\eta^1 \text{ grows indefinitely}} \geq 1 - \sum_{k\geq 1} c_b \exp\left(-a 2^k\right) - \limsup_{t\to\infty}\sum_{b=b_0}^\infty f_b(t).
   $$
   
   Now we turn to the proof of positive speed of growth for $\eta^1$.
   Note that $\eta^1\cup\eta^2$ is stochastically dominated by a first passage percolation process where the passage times are at least i.i.d.\ exponential
   random variables of rate $2$, because $\eta^2$ is slower than a first passage percolation of exponential passage times of rate $1$. 
   Then, by the shape theorem we have that there exists a constant $c>0$ large enough such that 
   $$
      \PR\lr{\eta^1(t)\cup\eta^2(t) \subset [-ct,ct]^d}\to 1 \quad\text{as $t\to\infty$}.
   $$
   Now fix any $t$, take $c$ as above, and set $\kappa= 1+\frac{\log t}{\lr{\log\log t}^2}$.
   For any large enough $a$ depending on $L_1$ and $p$, we have
   \begin{align*}
      &\PR\lr{\exists (k,i) \colon k\geq \kappa \text{ and } Q_k^\neighs(i)\cap [-ct,ct]^d \neq \emptyset \text{ and } Q_k^\core(i)\cap C_k\neq\emptyset} \\
      &\leq \sum_{k=\kappa}^\infty \sum_{i\colon Q_k^\neighs(i)\cap [-ct,ct]^d \neq \emptyset} \PR\lr{Q_k^\core(i)\cap C_k\neq \emptyset}\\
      &\leq \sum_{k=\kappa}^\infty \lr{\frac{2C_\FPP' \lr{ct+\frac{6}{5}R_k^\couter+R_k}}{L_k}}^d \exp\lr{-a 2^k}
      \leq \sum_{k=\kappa}^\infty \exp\lr{-a 2^{k-1}}
       \leq 2\exp\lr{-a 2^{\kappa-1}},
   \end{align*}
   where in the second inequality we use Lemma~\ref{lem:rhoexpanded}, and the third inequality follows because $a$ can be chosen large enough in Lemma~\ref{lem:rhoexpanded}. 
   The above derivation allows us to restrict to cells of scale smaller than $\kappa$.
   Note that since there are no contagious set of scale $\kappa$ or larger intersecting $[-ct,ct]^d$, the spread of 
   $\eta^1(t)$ inside $[-ct,ct]^d$ stochastically dominates a first passage percolation process of rate $\lambda^1_\kappa$. Thus, disregarding 
   regions taken by $\eta^2$, we can set a sufficiently small 
   constant $c'>0$ so that, at time $t$, $\bar\eta^1$ will 
   contain a ball of radius $2c't$ around the origin with probability at least $1-\exp\lr{-c'' t^\frac{d+1}{2d+4}}$ 
   for some constant $c''$, by Proposition~\ref{pro:kesten}. 
   The only caveat is that, at time $t$, there may be regions of scale smaller than $\kappa$ that are taken by $\eta^2$ and intersects the boundary of 
   $\ball{2c't}$. If we show that such regions cannot intersect $\bddi \ball{c't}$, then we have that 
   the probability that $\eta^1$ survives up to time $t$ but $\bar\eta^1(t)$ 
   does not contain a ball of radius $c't$ around the origin is at most $1-2\exp\lr{-a 2^{\kappa-1}}-\exp\lr{-c'' t^\frac{d+1}{2d+4}}$.
   This is indeed the case, since we can take a constant $c'''$ such that any cell of scale smaller than $\kappa$ has diameter at most 
   $$
      c'''(\kappa!)^{d+2} \leq c''' \exp\lr{(d+2)\kappa\log\kappa}\leq c'''\exp\lr{2(d+2)\frac{\log t}{\log\log t}}<c't,
   $$
   where the inequalities above hold for all large enough $t$, completing the proof.
\end{proof}
\section{From \mdla{} to \fpphe{}}\label{sec:mdla}
Here we show how to use the proof scheme for \fpphe{} from Section~5 to establish Theorem~\ref{thm:mdla}.
The relation between \fpphe{} and \mdla{} is very delicate, and we will need to introduce another process, which we call the \emph{$h$-process}. 
For the sake of clarity, this section is split into a few subsections.

\subsection{Dual representation and Poisson clocks}\label{sec:dual}
We start by recalling the dual representation of the exclusion process. In this dual representation, 
vertices without particles are regarded as hosting another type of particle, called \emph{holes}, while vertices hosting an original particle are seen as unoccupied. 
Using the terminology of the dual representation, in \mdla{}, 
holes perform a simple exclusion process among themselves, 
where they move as simple symmetric random walks obeying the exclusion rule (jumps to vertices already occupied by a hole or by the aggregate are suppressed).
Then the growth of the aggregate is equivalent to a first passage percolation process 
which expands along its boundary edges at rate 1, but with the additional feature that the aggregate does not occupy vertices that are occupied by holes.

To be more precise, we now define \mdla{} in terms of \emph{Poisson clocks}. 
A Poisson clock of rate $\nu$ is a clock that rings infinitely many times, and
such that the time until the first ring, as well as the time between any two consecutive rings, are given by independent exponential random variables of rate 
$\nu$.
Even though edges of $\mathbb{Z}^d$ have so far always been considered as undirected, we will need to assign an independent Poisson clock of rate 1 to each 
\emph{oriented} edge $(x \to y)$.
Then the evolution of \mdla{} is as follows. 
When the clock of $(x \to y)$ rings, if $x$ is occupied by a hole and $y$ is unoccupied, the hole jumps from $x$ to $y$. 
If $x$ is occupied by the aggregate and $y$ is unoccupied, then the aggregate occupies $y$. 
In any other case, nothing is done.
Henceforth, the Poisson clocks used to construct \mdla{} will be referred to as the \emph{\mdla-clocks}. 

\subsection{\mdla{} with discovery of holes}\label{sec:dischole}
We give a different representation of \mdla, which we refer to as \emph{\mdla{} with discovery of holes}. 
Each vertex of $\mathbb{Z}^d$ will either be occupied by the aggregate, be occupied by a hole, or be unoccupied.
As before, the aggregate starts from the origin. 
However, unlike before, each vertex of $\mathbb{Z}^d\setminus\{0\}$ is initially \emph{unoccupied}, and is assigned a non-negative integer value, which 
is given by 
an independent random variable having value $i\geq0$ with probability $(1-\mu)^i\mu$. This value represents the number of holes that can be born at that vertex.

More precisely, when the \mdla-clock of an edge $(x \to y)$ rings, a few things may happen.
\begin{itemize}
   \item If $x$ hosts a hole and $y$ is unoccupied, the hole jumps from $x$ to $y$.
   \item If $x$ belongs to the aggregate, $y$ is unoccupied and the value of $y$ is $0$, then the aggregate occupies $y$. 
   \item If $x$ belongs to the aggregate, $y$ is unoccupied and the value of $y$ is $i\geq 1$, then the value of $y$ is changed to $i-1$, a hole is born at $y$ (so $y$ becomes occupied),
         and the aggregate does not occupy $y$.
  \item In any other case, nothing happens.
\end{itemize}
Note that holes move
independently of the values of the vertices, and perform continuous time, simple symmetric random walks (jumping at the time of the \mdla-clocks) obeying the exclusion rule; that is, 
whenever a hole attempts to jump onto a vertex already occupied by a hole or by the aggregate, the jump is suppressed.
Note that this process is equivalent to the description of \mdla{} with the dual representation of the exclusion process. The only difference is that, instead of placing all holes at time $0$, holes are added 
one by one as the process evolves. More precisely, holes are born as the aggregate tries to occupy 
unoccupied vertices of value at least $1$.  
\subsection{Backtracking jumps and overall strategy}
The main idea we will use to compare \mdla{} and \fpphe{} is the following. 
Regardless of the location of a hole, if the 
hole jumps from a vertex $x$ to a vertex $y$, with positive probability the \mdla-clock of $(y \to x)$ rings before the other $2(2d-1)$ \mdla-clocks involving edges of the form $(\cdot \to x)$ or $(y \to \cdot)$. 
This causes the 
hole to jump back to $x$ before the hole can jump to any other vertex adjacent to $y$ or before the aggregate or another hole can occupy $x$. 
We call this a \emph{backtracking jump}. This type of jump intuitively gives that the rate at which a hole leaves a set of vertices is strictly smaller than $1$. 
In other words, holes are slower than the aggregate.
A natural approach is to set $\lambda<1$ in \fpphe{} to represent the rate at which holes move (taking into consideration backtracking jumps), and then couple \mdla{} and \fpphe{} so that the following properties hold.
\begin{enumerate}
   \item The seeds of $\eta_2$ are the vertices of value at least $1$ in \mdla. 
   \item The aggregate contains $\eta_1$ at all times.
   \item For all $t\geq 0$, the holes that have been discovered by time $t$ are contained inside $\eta_2(t)$. 
\end{enumerate}

Despite the above idea being relatively simple, the following delicate issue prevents this to be made into a rigorous argument.  
Suppose that the above three properties hold up to time $t$, and assume that at time $t$ 
the aggregate of \mdla{} contains vertices that do not belong to 
$\eta_1(t)$. Hence, at a later time the aggregate may discover a hole at a vertex $x$ which is at the boundary of the aggregate but is 
not at the boundary of $\eta_1$. In other words, the aggregate may discover a hole at a vertex $x$ whose seed cannot be activated since $x$ is yet unreachable by $\eta_1$. 
At this time, property 3 would cease to hold. 

To go around the above issue, we will employ a coupling argument to show that
\mdla{} stochastically dominates \fpphe{} \emph{locally}. In particular, we will use that coupling to show that the encapsulation procedure used for \fpphe{} (via Proposition~\ref{pro:encapsulate2}) 
works as well for \mdla. 
Then, the multi-scale machinery developed in Section~\ref{sec:proof} can be used to obtain that 
each cluster of 
holes get encapsulated by the aggregate at some finite (possibly large) scale.
For this, we will use the fact that 
the encapsulation procedure we did for \fpphe{} in Proposition~\ref{pro:encapsulate2} is implied by 
the occurrence of a \emph{monotone} event $F$, which is increasing with respect to the passage times of type~2 and decreasing with respect to the passage times of type~1.
\subsection{Coupling of the initial configuration}\label{sec:initialconfig}
We now formalize the coupling of the initial configurations of \mdla{} and \fpphe, as suggested in the previous section. 
For each vertex $x\in\mathbb{Z}^d\setminus\{0\}$, note that the probability that $x$ is assigned a value at least 1 in \mdla{} with discovery of holes is $\sum_{i=1}^{\infty}(1-\mu)^i \mu = 1-\mu$. 
Then, we set $p=1-\mu$ so that we can couple the vertices with value at least $1$ with the location of the type-2 seeds of \fpphe{}. From now on, for each vertex of value at least $1$, we will refer to it 
as a \emph{seed}, regardless of whether we are talking about \mdla{} or \fpphe.

\subsection{The $h$-process}\label{sec:prepcoupling}
We will not actually couple \mdla{} with \fpphe, but we will couple \mdla{} with
another process $\{h_t\}_t$, which will be a growing subset of $\mathbb{Z}^d$.
We call this process the \emph{$h$-process}.
The $h$-process will be constructed using the \mdla-clocks and the seeds, where the seeds have been coupled with \mdla{} as described in Section~\ref{sec:initialconfig}.

When a vertex $x$ belongs to $h_t$, we will say that $x$ is \emph{infected}. 
To avoid confusion, we will not say that $x$ is occupied by $h_t$ since, as we explain later, a vertex that is occupied by the aggregate can also be infected.
Our goal with the $h$-process is to obtain that the holes that have already been discovered at time $t$ are contained in $h_t \cup \bddo h_t$, and 
the ones in $\bddo h_t$ are the holes that will jump back to $h_t$ (in a backtracking jump).

At time $0$ we set $h_0=\emptyset$, and let the aggregate spread using the \mdla-clocks using the representation with discovery of holes.
The $h$-process will evolve according to three operations: birth, expansion and halting upon encapsulation. 
\begin{itemize}
   \item Birth. If at time $t$ a hole is discovered by the aggregate inside a cluster $C\subset\mathbb{Z}^d$ of seeds\footnote{For the sake of clarify, given any set of vertices $X\subset \mathbb{Z}^d$, 
      when we say that $C$ is a \emph{cluster} of $X$, we mean that $C\subset X$ and $\bddo C \cap X = \emptyset$.},
      then we infect $C$; that is, we add to $h_t$ the whole cluster $C$ of seeds. 
   
   \item Expansion. For each unoriented edge $(x,y)$, we will define a passage time $\tau_{x,y}$ (which we will specify later on and will depend on the evolution of \mdla). 
      So if $x$ gets infected at time $t$, then $x$ infects $y$ at time $t+\tau_{x,y}$;
      note that $y$ could get infected before $t+\tau_{x,y}$ if a neighbor of $y$ different than $x$ infects $y$.
      
   \item Halting upon encapsulation. The $h$-process is allowed to infect vertices that are occupied by the aggregate. 
      However, if at some moment a cluster $C$ of $h_t$ is separated from infinity by the aggregate, which means that any path from $C$ to infinity intersects $\mathcal{A}_t$, then 
      $h_t$ will not infect any vertex of $\bddo C$ that already belongs to the aggregate\footnote{If the aggregate decides to occupy a site $x$ at the same time as the $h$-process decides to infect $x$, then we assume that 
      the aggregate occupies $x$ immediately \emph{before} the $h$-process infects $x$. This is just a convenience to take care of the following situation. 
      Assume that $x$ is a neighbor of a cluster of infected sites which is disconnected from infinity by the aggregate but $x$ itself is not disconnected from infinity by the aggregate. This implies that the aggregate occupies the 
      neighbors of $x$ that belong to the infected cluster. 
      If at this time the aggregate and the $h$-process both decide to occupy $x$, we obtain that the aggregate does so first, and then the $h$-process does not infect $x$ due to the halting upon encapsulation operation.}. 
      This is to guarantee that a cluster of the $h$-process is confined to a finite set when it gets encapsulated by the aggregate.
\end{itemize} 


Now we introduce some notation. 
For each vertex $x$, let $\mathcal{E}_x$ be the set of (unoriented) edges incident to $x$, and let $\mathcal{E}_x^\to$ (resp., $\mathcal{E}_x^\leftarrow$) be the set of oriented edges going out of (resp., coming into) $x$.
For each edge $(x \to y)$, let 
\begin{equation}
   \mathcal{E}_{x\to y}=\mathcal{E}_x^\leftarrow\cup\mathcal{E}_y^\to,
   \label{eq:exy}
\end{equation}
which includes $(y \to x)$ but not $(x\to y)$. 
Let 
\begin{equation}
   \text{$M=4d-1$ and note that, given an edge $(x\to y)$, we have that $M=|\mathcal{E}_{x\to y}|$.}
   \label{eq:defm}
\end{equation}
Later, $\frac{1}{M}$ will be a lower bound on the probability that a hole performs a backtracking jump. 

We will use the convention that if we write $(x \to y)\in\bdde h_t$, we mean that $x \in h_t$ and $y\not\in h_t$.
We will update a set of edges $\mathcal{H}(t)$ as the $h$-process evolves, starting from $\mathcal{H}(0)=\emptyset$. 
Moreover, for each $(x\to y)\in\mathcal{H}(t)$, we will associate an independent Bernoulli random variable $\mathfrak{B}_{x\to y}$ of parameter $1/M$. 
If needed, the random variable $\mathfrak{B}_{x\to y}$ will be redraw independently during the evolution of the $h$-process.
Once $\mathcal{H}(t)$ is specified, we define
\begin{align}
   \mathcal{H}_B(t) = \bigcup\nolimits_{(u \to v)\in \mathcal{H}(t)} \mathcal{E}_{u \to v},
   \quad\text{ and }\quad
   \mathcal{H}_B^*(t) = \bigcup\nolimits_{\substack{(u \to v),(u' \to v')\in \mathcal{H}(t) \\(u\to v)\neq (u'\to v')}} \mathcal{E}_{u \to v}\cap \mathcal{E}_{u' \to v'}.
   \label{eq:defsb}
\end{align}

\subsection{Evolution of the $h$-process}\label{sec:hexp}
Here we will describe how the $h$-process uses the \mdla-clocks to evolve. 
Our description here will be precise, but will not enter in the details needed to define the passage times of the $h$-process. This will be carried out in Section~\ref{sec:coupling}.

Start from time $0$, where we have $h_0=\emptyset$ and $\mathcal{H}(t)=\emptyset$.
From this time, we let \mdla{} evolve using its \mdla-clocks.
If at some time $t$ \mdla{} tries to occupy a seed (that is, \mdla{} discovers a hole) inside some cluster $C\subset\mathbb{Z}^d$ of seeds, the $h$-process undergoes a birth operation and we set $h_t=C$.
At this moment, we continue to let \mdla{} evolve using its \mdla-clocks. If new holes are discovered, new births take place and clusters are added to the $h$-process (that is, new clusters are infected). 

We will now discuss all possibilities that could happen for the expansion of the $h$-process. 
In all cases below, we will assume that the expansion of the $h$-cluster is happening in an 
infected cluster that is not disconnected from infinity by the aggregate. 
Otherwise, the halting upon encapsulation would already have happened to that cluster, which would prevent it from expanding.

The $h$-process only expands when an edge at the boundary of the $h$-process or an edge from $\mathcal{H}_B$ rings.
(The edges in $\mathcal{H}_B$ are needed to verify backtracking jumps, and ``B'' in the subscript actually stands for backtracking.)
If an edge that is internal to the $h$-process (that is, both of its endpoints are already infected) rings, then the $h$-process does not change, even if that causes new holes to be discovered.
Similarly, if an edge that is external to the $h$-process (that is, both endpoints are not infected) and does not belong to $\mathcal{H}_B$ rings, and no new hole get discovered by this operation, then the 
$h$-process does not change.

Now we assume that an edge $(x\to y)$ from the boundary of the $h$-process or from $\mathcal{H}_B$ rings at time $t$, and discuss what occurs with the $h$-process at this time. We split our discussion into 
three cases, and at the end explain two particular situations.
Let $s<t$ be the last time before $t$ that the $h$-process changed.

\subsubsection{Case 1: a hole jumps out of the $h$-process}\label{sec:holeout}
This corresponds to $(x\to y)\in \bddo h_s$ with a hole at $x\in h_s$ and $y\not\in h_s$ unoccupied at time $t-$.
Thus, the ring of $(x\to y)$ causes the hole to jump from $x$ to $y$. 

It could be the case that there is already an edge $(y'\to y)\in \mathcal{H}(t)$ with $y'\neq x$. 
If this is the case, we simply do nothing; this case will be better discussed in Section~\ref{sec:revisit}. 
Otherwise, if there is no such edge, we add $(x\to y)$ to $\mathcal{H}(t)$ to verify whether the hole will do a backtracking jump to $x$. 
Moreover, we draw (independently from previous values that this random variable could have assumed) the Bernoulli random variable $\mathfrak{B}_{x\to y}$ of parameter $1/M$. 
If $\mathfrak{B}_{x\to y}=0$, which means that the hole will not backtrack to $x$, then we infect $y$ at time $t$. 

\subsubsection{Case 2: verification of backtracking jumps}\label{sec:backtrack}
This corresponds to $(x\to y)\in \mathcal{H}_B(s)$. 
Let $(u\to v)$ be the edge from $\mathcal{H}(s)$ such that $(x\to y)\in \mathcal{E}_{u\to v}\subset \mathcal{H}_B(s)$, and assume that $(u\to v)$ was added to $\mathcal{H}$ at time $t'\leq s$.
Note that, from Section~\ref{sec:holeout}, this was done because a hole jumped from $u$ to $v$ at time $t'$.  
Then while no clock from $\mathcal{E}_{u \to v}$ rings, $u$ will remain unoccupied and the hole will remain at $v$.  

When the clock of an edge $(x\to y)$ from $\mathcal{E}_{u \to v}$ rings at time $t$, then the first thing we do is to remove $(u\to v)$ from $\mathcal{H}(t)$. 
We will say that the possibility of a backtracking jump through $(u\to v)$ has been verified.
We then couple the value of $\mathfrak{B}_{u\to v}$ with the \mdla-clocks so that if $\mathfrak{B}_{u\to v}=1$, we have that the first clock to ring among the ones from $\mathcal{E}_{u \to v}$ is $(v \to u)$. 
If this happens, then $(x\to y)=(v\to u)$ so, at time $t$, the hole backtracks to $u$. Note that both $u$ and $v$ are already infected. In this case, nothing else needs to be done. 

On the other hand, if $\mathfrak{B}_{u\to v}=0$, the backtracking will not happen and $(x\to y)\in \mathcal{E}_{u\to v}\setminus (v \to u)$.
Note that the probability that $(x\to y)$ is equal to a given $(w\to z)\in \mathcal{E}_{u\to v}\setminus (v\to u)$ is exactly 
\begin{equation}
   \frac{1}{M-1}.
   \label{eq:case4}
\end{equation}
Note also that $v$ is already infected. 
If $(x\to y)\in\mathcal{E}_u^\leftarrow$, then $u$ could get occupied by the aggregate or by another hole, which could prevent the hole from $v$ to jump back to $u$ when the clock $(v\to u)$ rings. 
In this case, we do not need to do anything else. 

Finally, if $(x \to y)\in \mathcal{E}_v^\rightarrow$, then the hole at $v=x$ may jump to $y$, if $y$ is unoccupied. 
If, in addition, we have that $y\not\in h_s$, then the hole jumped out of the $h$-process, so we perform 
the steps described in Section~\ref{sec:holeout} to $(x\to y)$ so that we can later verify the possibility of a backtracking jump through $(x\to y)$. 
In particular, we add $(x \to y)$ to $\mathcal{H}(t)$, sample $\mathfrak{B}_{x\to y}$, and infect $y$ if $\mathfrak{B}_{x\to y}=0$. 

The purpose of the set $\mathcal{H}$ is to keep track of the edges over which a backtracking jump can happen. 
It will hold that
\begin{align}
   &\text{at all times $s'$, $\mathcal{H}(s')$ is a set of disjoint edges (with no endpoint in common)}\nonumber\\
   &\text{such that for any $(w \to z)\in \mathcal{H}(s')$, $w\in h_s'$ is unocuppied in \mdla{}}. 
   \label{eq:sh}
\end{align} 
The above is quite straightforward from our description above, but we will actually prove it in Lemma~\ref{lem:sh}, 
after we describe precisely how the passage times are constructed. 
We remark that in~\eqref{eq:sh} we do not require $z$ to host a hole. This is because of a corner case that we need to handle carefully, and which we will explain in Section~\ref{sec:hbs}.

\subsubsection{Case 3: expansion without jump of holes}\label{sec:noholes}
Here we assume that $(x\to y)\in \bddo h_t\setminus H_B(s)$ and such that one of the following conditions happens.
\begin{itemize}
\item $x\in h_{s}$ is unoccupied at time $t-$.
\item $x$ is occupied by the aggregate at time $t-$.
\item $x$ is occupied by a hole, but $y$ is occupied by either a hole or the aggregate at time $t-$ (preventing the hole from $x$ to jump to $y$).
\end{itemize}
(The case of $x$ being occupied by a hole and $y$ unoccupied is covered by Section~\ref{sec:holeout}.)
The above three situations bring little trouble to us since it does not cause any hole to jump, so we will simply choose to
infect $y$ with probability $\frac{M-1}{M}<1$, otherwise we do nothing. 
This choice is to guarantee that the passage times $\tau_{\cdot,\cdot}$ stochastically dominate (but are not equal to) exponential random variables of rate $1$ in this case.

\subsubsection{Edges in $\mathcal{H}^*_B$}\label{sec:hbs}
We need to give special attention to the set $\mathcal{H}^*_B$. 
Suppose now that we have carried the above process up to a time $t$, when it happens that $\mathcal{H}^*_B(t)\neq\emptyset$. Note that
\begin{align}
   \text{if $(x \to y)\in \mathcal{H}_B^*(t)$, there exist a neighbor $x'$ of $x$ and a neighbor $y'$ of $y$ such that}\nonumber\\
   \text{$(x' \to x),(y \to y') \in \mathcal{H}(t)$, and hence $y \in h_t$ and $(x\to y) \in \mathcal{E}_{x'\to x} \cap \mathcal{E}_{y\to y'} \subset \mathcal{H}_B(t)$.}
   \label{eq:doubleclock}
\end{align}
Note that if $t$ is the first time that $\mathcal{H}^*_B(t)\neq \emptyset$, then both $x$ and $y'$ host a hole.
The above gives that $(x\to y)$ is involved in the backtracking jumps of both $(x' \to x)$ and $(y \to y')$, which is in conflict with the fact that 
$\mathfrak{B}_{x' \to x}$ and $\mathfrak{B}_{y \to y'}$ are independent.

To solve this, for each $(x\to y)\in\mathcal{H}_B^*(t)$, we will consider two Poisson clocks:  
the actual \mdla-clock, which will be associated to the backtracking jump of $(y \to y')$, and a \emph{fake-clock}, which will be associated to the backtracking jump of $(x' \to x)$. 
So, if $\mathfrak{B}_{x' \to x}=1$, it means that the clock of $(x \to x')$ will ring before the \mdla-clocks of 
$\mathcal{E}_{x'\to x}\setminus (x\to y)$ and before the fake-clock of $(x \to y)$. 
Similarly, if $\mathfrak{B}_{y \to y'}=1$, it means that the clock of $(y' \to y)$ will ring before the \mdla-clocks of 
$\mathcal{E}_{y \to y'}$. 
The evolution of \mdla{} simply ignores the fake-clocks. Since the fake-clocks and the \mdla-clocks are independent, there is no conflict 
with the independence of $\mathfrak{B}_{x' \to x}$ and $\mathfrak{B}_{y \to y'}$. Now we explain why this does not create other problems.

If the \mdla-clock of $(x \to y)$ rings, we say that a clock from $\mathcal{E}_{y\to y'}$ rings, whereas when the fake-clock of 
$(x \to y)$ rings, we say that a clock from $\mathcal{E}_{x'\to x}$ rings.
Assume that the first clock to ring among the \mdla-clocks and fake-clocks of $\mathcal{E}_{x\to y}$ is the 
\mdla-clock of $(x \to y)$. Let $s>t$ be the time at which that clock rings, and assume that this is the first clock to ring among the clocks of $\mathcal{E}_{x'\to x}$ and $\mathcal{E}_{y\to y'}$.
Note that in this case we have $\mathfrak{B}_{y \to y'}=0$.
Then, the hole that is in $x$ jumps to $y$, and we perform the steps described in Section~\ref{sec:backtrack} for the backtracking jump of $(y \to y')$ when $\mathfrak{B}_{y \to y'}=0$.
No action is taken with regards to the backtracking jump of $(x'\to x)$. In particular, we have that $(y\to y') \not\in \mathcal{H}(s)$ and, more crucially, we have that $(x'\to x) \in \mathcal{H}(s)$
even if there is no hole at $x$. (This is the reason why in~\eqref{eq:sh} we have not required $y$ to host a hole.)

The fact that $(x'\to x)$ remained in $\mathcal{H}(s)$ will not cause problems because the hole that was in $x$ jumped inside $h_s$ (because $y\in h_t \subseteq h_s$). 
So, in some sense, that hole did backtrack to the $h$-process. 
We will later still process the backtracking jump of $(x'\to x)$ even if there may not be a hole at $x$ (which just means that no hole will jump, but 
the $h$-process may still be updated according to the decision of a backtracking jump). 
For example, if $\mathfrak{B}_{x' \to x}=1$, we will assume that there is a backtracking jump over $(x'\to x)$ causing $x$ not to be added to the $h$-process, which remains to be true 
even if $x$ does not host a hole. 

The fake-clock of $(x\to y)$ will exist while $(x\to y)\in\mathcal{H}_B^*$, that is,
while both $(x'\to x)$ and $(y \to y')$ belong to $\mathcal{H}$. When this ceases to be true, the fake-clock of $(x \to y)$ is simply deleted and we will only keep track of its \mdla-clock.
\begin{align}
   \text{In the following, the term \emph{the clocks of $\mathcal{E}_{x\to y}$} means the \mdla-clocks and the fake-clocks}\nonumber\\
   \text{associated to the edges of $\mathcal{E}_{x \to y}$ when considering the backtracking jump of $(x\to y)$.}
   \label{eq:clockconvention}
\end{align} 


\subsubsection{Holes revisiting uninfected vertices}\label{sec:revisit}
Consider the setting in the previous section, where $(x'\to x)\in \mathcal{H}(s)$ but there is no hole at $x$. 
Note that if $\mathfrak{B}_{x'\to x}=1$ (so that $x\not\in h_{s}$), 
before the backtracking jump of $(x'\to x)$ is processed (that is, before the \mdla-clocks of $\mathcal{E}_{x'\to x}\setminus (x\to y)$ and the fake-clock of $(x\to y)$ ring), it could happen that a hole jumps from
a vertex $x''$ to $x$. 
Note that $x'' \neq x'$, because $x'$ remained unoccupied from the time the hole jumped from $x'$ to $x$ since $(x'\to x)$ is still in $\mathcal{H}$.
Since $(x''\to x)\not\in \mathcal{E}_{x'\to x}$, this jump does not affect the backtracking jump of $(x'\to x)$. 
But since $x$ is not infected, the hole just jumped out of the $h$-process. 
Suppose that this happens at some time $s''$. This is the situation explained in Section~\ref{sec:holeout}, when nothing needs to be done to the $h$-process. 
The reason is that this hole just occupied the place of the yet to be verified backtracking jump of $(x'\to x)$. Thus we only need to wait that backtracking to be processed.

Note that it can also happen that much later $x$ is still not infected and a hole jumps from $x'$ to $x$ again. 
But, as we explained above, this can only happen after the initial backtracking jump of $(x'\to x)$ has been proceed. 
Since we had that $\mathfrak{B}_{x'\to x}=1$ (so that $x$ remained uninfected), 
the second jump of a hole from $x'$ to $x$ can only happen after the clock of the edge $(x\to x')$ rings (because this happens before any edge from $\mathcal{E}_{x'}^{\leftarrow}$ rings, so $x'$ remained unoccupied).
When the edge $(x\to x')$ rings, the memoryless property of exponential random variables guarantees that the rings of the clocks in $\mathcal{E}_{x'\to x}$ are from this moment independent of the past.
Note that at this time it also happens that the value of $\mathfrak{B}_{x'\to x}$ is redrawn independently, so this new hole jumping to $x$ will have no correlation with the previous backtracking jump through $(x'\to x)$.

\subsection{Construction of the passage times of the $h$-process}\label{sec:coupling}
In the previous section we described how the $h$-process uses the \mdla-clocks to evolve. 
Here we will use the discussion from the previous section to describe the construction of the passage times of the $h$-process. 

For each (unoriented) edge $(x,y)$, we will construct an ordered list 
$\Pi_{x,y}=\big(\Pi_{x,y}^{(1)},\Pi_{x,y}^{(2)},\ldots,\Pi_{x,y}^{(\kappa_{x,y})}\big)$,
where the $\Pi_{x,y}^{(i)}$ will be independent random variables, and the number of elements $\kappa_{x,y}$ in $\Pi_{x,y}$ will also be a random variable.  
Recalling that $\tau_{x,y}$ is the passage time of the $h$-process through the edge $(x,y)$, we will construct the $h$-process so that 
\begin{equation}
   \tau_{x,y}=\sum_{i=1}^{\kappa_{x,y}} \Pi_{x,y}^{(i)}.
   \label{eq:tau}
\end{equation}

Suppose the $h$-process has been constructed up to time $t$, and assume that 
\begin{equation}
   \text{the set of holes discovered up to time $t$ is contained in $h_t\cup\bddo h_t$.}
   \label{eq:prophole}
\end{equation}
Assume also that 
\begin{align}
   &\text{for any $y\in \bddo h_t$ that hosts a hole, there is a unique $x$ such that $(x \to y)\in\mathcal{H}_t$.}\nonumber\\
   &\text{Moreover, in the past a hole jumped to $y$ from $x$, and the possibility of} \nonumber\\
   &\text{a backtracking jump of that hole has not yet been verified.}
   \label{eq:propbt}
\end{align}
The last property above means that if the hole jumped from $x$ to $y$ at some time $s\leq t$, then during $(s,t]$ the clocks of $\mathcal{E}_{x \to y}$ have not rung (but it could be the case that a fake-clock from 
$\mathcal{E}_{x\to y}$ has rung, see the discussions in Sections~\ref{sec:hbs} and~\ref{sec:revisit}).
We will use the convention that $\bddo h_t$ gives the outer boundary of $h_t$ that can be infected by the $h$-process. (Recall that vertex at the boundary of $h_t$ that is occupied by the aggregate 
does not get infected if the corresponding cluster of 
the $h$-process is separated from infinity by the aggregate of \mdla.)

We let \mdla{} evolve according to its clock until the first time $t+W$ at which either of two events happen:
\begin{enumerate}
   \item A birth operation takes place (see the definition in Section~\ref{sec:prepcoupling}). Note that in this case new vertices are added to the $h$-process, so $h_{t+W}\neq h_t$.
   \item A clock (regardless of whether it is a \mdla-clock or a fake-clock) from $\bdde h_t \cup \mathcal{H}_B(t)$ rings. 
       In this case, we will not yet observe which of these clocks rang. We will call this case a \emph{potential expansion operation}.
\end{enumerate}

\subsubsection{Birth Operation or addition of vertices to the $h$-process}\label{sec:ptbirth}
Suppose that at time $t+W$ an (unoriented) edge $(x,y)$ becomes part of $\bdde h_{t+W}$; assume without loss of generality that this is because $x$ got infected at time $t+W$. 
Then we create a variable $\Lambda_{x,y}$ whose initial value is 0.
The value of $\Lambda_{x,y}$ will be updated and will be used to add elements to the list 
$\Pi_{x,y}$.
Then we perform the following step:
\begin{equation}
   \text{ add $W$ to each $\Lambda_{u,v}$ for which $(u,v)\in \bdde h_t$}.
   \label{eq:passobirth}
\end{equation}

\subsubsection{Potential expansion operation of the $h$-process}\label{sec:ptexpand}
Suppose that at time $t+W$ an edge from $\bdde h_t \cup \mathcal{H}_B(t)$ rings.
The $h$-process will expand according to the description in Section~\ref{sec:hexp}, but we elaborate a little bit more here. 
We do not immediately observe which is the edge that rings, this edge is still random and could also correspond to a fake-clock from an edge of 
$\mathcal{H}_B^*(t)$.

Now we want to sample which clock from $\bdde h_t\cup \mathcal{H}_B(t)$ rings first. We will denote by  
$$
   (x \to y) \text{ the random edge whose clock rang}.
$$
However, we need to be a bit careful in the sampling of $(x\to y)$. First, let 
$$
   \text{$(x' \to y')$ be a uniformly random choice among the edges with clocks in $\bdde h_t \cup \mathcal{H}_B(t)$}, 
$$
where each edge 
in $\mathcal{H}_B^*(t)$ is counted with multiplicity $2$ (to account for its fake-clock and \mdla-clock) while the other edges have multiplicity $1$.
We would like to set $(x\to y)=(x'\to y')$. The problem is that, when a hole jumps over an edge $(u \to v)$ at some time $s$, where $v\not\in h_s$, 
we need to decide right away whether that hole will backtrack to $u$ later and 
this impacts which edge from $\mathcal{E}_{u\to v}$ rings first.
This decision is a function of the Bernoulli variable $\mathfrak{B}_{u\to v}$.
As explained in Section~\ref{sec:backtrack}, if $\mathfrak{B}_{u\to v}=1$, we know that the next 
clock to ring will be that of $(v\to u)$, whereas if $\mathfrak{B}_{u\to v}=0$ we have that the next clock to ring is chosen uniformly at random from the clocks of 
$\mathcal{E}_{u\to v}\setminus (v \to u)$. 

We can now state precisely how $(x\to y)$ is chosen. 
Recall that if $(x' \to y')\in \mathcal{H}_B^*(t)$ with $(x'\to y')\in \mathcal{E}_{x''\to x'}\cap \mathcal{E}_{y' \to y''}$, then we say that $(x' \to y')\in \mathcal{E}_{x'' \to x'}$ only if it was the fake-clock of 
$(x'\to y')$ that rang, and say that 
$(x' \to y')\in \mathcal{E}_{y' \to y''}$ if it was the \mdla-clock of 
$(x'\to y')$ that rang. Then, if $(x' \to y')\in \mathcal{E}_{u\to v}$ for some $(u\to v)\in\mathcal{H}(t)$ with $\mathfrak{B}_{u\to v}=1$, set $(x\to y)=(v \to u)$. 
If $(x' \to y')=(v\to u)$ for some $(u\to v)\in\mathcal{H}(t)$ with $\mathfrak{B}_{u\to v}=0$, then $(x\to y)$ is chosen uniformly at random from $\mathcal{E}_{u\to v}\setminus (v\to u)$.
In any other case, $(x\to y)=(x'\to y')$.

Now we describe the actions we need to do to compute the passage times as the $h$-process evolves according to the description in Section~\ref{sec:hexp}.
If an edge enters $\bdde h_{t+W}$ due to the ring of $(x\to y)$, we perform the steps described in Section~\ref{sec:ptbirth} for that edge.
Moreover, in any case, we
\begin{equation}
   \text{ add $W$ to each $\Lambda_{u,v}$ for which $(u,v)\in \bdde h_t$}.
   \label{eq:passo1}
\end{equation}
After the above, we do the following.
\begin{align}
   \text{if $(x,y)\in \bdde h_t$, add $\Lambda_{x,y}$ as a new element to the list $\Pi_{x,y}$,}\nonumber\\
   \text{and reset $\Lambda_{x,y}$ to 0}.
   \label{eq:Pi}
\end{align}

Note that if we always have $(x' \to y') =(x\to y)$, then the variables $\Lambda_{x,y}$ would be i.i.d.\ exponential random variables of rate $1$, since they are 
constructed exactly as described in Lemma~\ref{lem:decomp}.  
However, we will have cases when $(x'\to y' ) \neq (x\to y)$, which will still give rise to the variables $\Lambda_{x,y}$ being independent exponential random variables, but of possibly different rates. 
This will be explained in Section~\ref{sec:ptproperties}, after the passage times have been constructed.

Finally, if we have that $(x\to y) \in \bdde h_t$ and $y$ gets infected at time $t+W$, we close the list $\Pi_{x,y}$, and define $\tau_{x,y}$ as in~\eqref{eq:tau} (thereby concluding the construction of $\tau_{x,y}$). 
This means that nothing else will be added to the list 
$\Pi_{x,y}$. This can happen in the following cases:
\begin{itemize}
   \item $(x\to y) \in \bdde h_t \setminus \mathcal{H}_B(t)$, with $x$ not hosting a hole in \mdla{} or $y$ being occupied in \mdla. This was described in Section~\ref{sec:noholes}, where $y$ gets infected with 
   probability $\frac{M-1}{M}$.
   
   \item $(x\to y) \in \bdde h_t \setminus \mathcal{H}_B(t)$, with $x$ hosting a hole in \mdla, $y$ unoccupied, there exists no other edge $(\cdot \to y)\in \mathcal{H}(t)$, and $\mathfrak{B}_{x\to y}=0$. This is the case that a hole jumps out of $h_t$ (from $x$ to $y$) and does not 
   backtrack to $x$. This was described in Sections~\ref{sec:holeout} and~\ref{sec:backtrack}.
   
   \item $(x\to y) \in \mathcal{E}_{u\to v}$ for some $(u\to v)\in\mathcal{H}(t)$ with $\mathfrak{B}_{u\to v}=0$ (so $v\in h_t$ and $(x\to y)\neq (v\to u)$). 
      As mentioned in~\eqref{eq:case4}, the probability that $(x\to y)$ is equal to a given $(w\to z)\in \mathcal{E}_{u\to v}\setminus (v\to u)$ is $\frac{1}{M-1}$.
      If, in addition, we have that $(x\to y) \in \mathcal{E}_v^\rightarrow$, so $x=v$, we may infect $y$ if there is still a hole at $x=v$ and $\mathfrak{B}_{x\to y}=0$.
      (Note that this actually falls into the setting of the previous case, but we chose to highlight it here since a given edge from $\mathcal{E}_v^\rightarrow$ rings at rate $\frac{M}{M-1}$ instead of rate $1$, due to the 
      conditioning on $\mathfrak{B}_{x\to y}=0$.)
\end{itemize}

Iterating the above construction will produce the lists $\Pi_{x,y}$ for some edges $(x,y)$. 
For each edge that was not visited during this procedure, we sample an independent, exponential random variable 
of rate $\frac{M-1}{M}$ to be its passage time. 
For each edge $(u,v)$ whose construction was not completed, 
we add to $\Lambda_{u,v}$ an independent exponential random variable of rate $\frac{M-1}{M}$, add $\Lambda_{u,v}$ as a new element to $\Pi_{u,v}$, and complete the construction of $\tau_{u,v}$. 
Regardless of the value of these random variables, the evolution of the $h$-process will not change.
Also, the final passage times of those edges stochastically dominate exponential random variables of rate $\frac{M-1}{M}$.

\subsection{Properties of the passage times}\label{sec:ptproperties}
Before establishing properties of the passage times, we establishe~\eqref{eq:sh}.
\begin{lemma}\label{lem:sh}
   For any $t\geq0$,~\eqref{eq:sh} holds. 
\end{lemma}
\begin{proof}
   Clearly~\eqref{eq:sh} holds at time $0$ since $\mathcal{H}(0)=\emptyset$. Now assume that it holds during $[0,t)$, and that 
   at time $t$ a hole jumps out of $h_{t-}$, from $x$ to $y$; so $(x\to y)$ is added to $\mathcal{H}(t)$ via Case 1 (cf.\ Section~\ref{sec:holeout}) and $x\in h_t$. 
   Then, at time $t$, $x$ is not occupied by a hole or the aggregate in \mdla{} and belongs to $h_t$, while $y$ hosts a hole in \mdla{}.
   If $\mathfrak{B}_{x\to y}=0$, $y$ gets infected at time $t$ and~\eqref{eq:sh} continues to hold.
   If $\mathfrak{B}_{x\to y}=1$, $y$ remains uninfected, but $x$ remains unoccupied until at a time $s>t$ 
   the clock of an edge from $\mathcal{E}_{x}^\leftarrow$ rings, but that edge must be $(y\to x)$ since $\mathfrak{B}_{x\to y}=1$. 
   So the hole gets back to an infected vertex. 
   
   It remains to show that the edges in $\mathcal{H}(t)$ are disjoint. Assume that this is not the case; that is, there are edges $(x\to y),(u\to v)\in\mathcal{H}(t)$ with $|\{u,v\}\cap\{x,y\}|=1$.
   Assume that $(x\to y)$ and $(u\to v)$ were added to $\mathcal{H}$ at times $t_x$ and $t_u$, respectively, with $t_x>t_u$.
   If $u=x$, then at some time during $(t_u,t_x)$ a hole jumped into 
   $u=x$ in order to go to $y$ at time $t_x$. But this would cause $(u\to v)$ to be removed from $\mathcal{H}$ by Case 2, which would imply that $(u\to v)\not \in \mathcal{H}(t_x)$.
   If $v=y$, then at time $t_x$ we have $y\not\in h_{t_x}$; otherwise we would not add $(x\to y)$ to $\mathcal{H}(t_x)$. 
   In this case, since $y$ does not host a hole at time $t_x-$, Case 1 gives that $(x\to y)$ is not added to $\mathcal{H}(t_x)$ because there is already an edge $(\cdot \to y) \in \mathcal{H}(t_x-)$.
   The case $u=y$ cannot happen, because $(x\to y)\in \mathcal{E}_{u}^\leftarrow$, so Case 2 would remove $(u\to v)$ from $\mathcal{H}(t_x)$. 
   The case $v=x$ is similar, since $(x\to y)\in \mathcal{E}_{v}^\rightarrow$, so Case 2 would remove $(u\to v)$ from $\mathcal{H}(t_x)$. . 
\end{proof}

Let $Z^{(M)}$ be the following random variable. Take $Z'$ to be an exponential random variable of rate $M$, $Z''$ be an exponential random variable of rate $\frac{M-1}{M}$, and 
$Q$ be a Bernoulli random variable of parameter $\frac{M-1}{M}$, where $Z',Z''$ and $Q$ are independent of one another. Define
\begin{equation}
   Z^{(M)} = Z' + \ind{Q=1}Z''.
   \label{eq:zk}
\end{equation}
\begin{lemma}\label{lem:zk}
   For any $M> 0$, $Z^{(M)}$ stochastically dominates (strictly) an exponential random variable of rate $1$. 
   Moreover, $Z^{(M)}$ is stochastically dominated by an exponential random variable of rate $\frac{M-1}{M}$.
\end{lemma}
\begin{proof}
   For the first part, note that if we take~\eqref{eq:zk} and replace $Z''$ with $\hat Z$, where $\hat Z$ is an exponential random variable of rate $1$, then Lemma~\ref{lem:decomp} with $k=M$, $W=Z'$ and $X_1=\hat Z$ 
   (the values of $X_2,X_3,\ldots,X_M$ being irrelevant) gives that 
   $Z' + \ind{Q=1}\hat Z$ is an exponential random variable of rate $1$. Since $Z''$ stochastically dominates $\hat Z$, we obtain the first part of the lemma.
   
   For the second statement, we know from Lemma~\ref{lem:decomp} and the first part that 
   $Z' + \ind{Q=1}\hat Z$
   is an exponential random variable of rate $1$. Then, using Lemma~\ref{lem:minscale}, we have that 
   $\frac{M}{M-1}Z' + \ind{Q=1}\frac{M}{M-1}\hat Z$ is an exponential random variable of rate $\frac{M-1}{M}$. But that variable has the same distribution as 
   $\frac{M}{M-1}Z' + \ind{Q=1}Z''\geq Z^{(M)}$.
\end{proof}

Now we show that the passage times $\tau_{x,y}$ stochastically dominate i.i.d.\ random variables distributed as $Z^{(M)}$.
\begin{lemma}\label{lem:hp}
   The collection of passage times $\tau_{x,z}$, for each edge $(x,y)$, stochastically dominates an i.i.d.\ collection of random variables distributed as $Z^{(M)}$.
\end{lemma}
\begin{proof}
   If an edge $(w,z)$ is such that its passage time was not completed during the procedure above, then we know that its passage time stochastically dominates an independent, exponential random variable of 
   rate $\frac{M-1}{M}$, which in turn stochastically dominates $Z^{(M)}$ by the second part of Lemma~\ref{lem:zk}.
   
   So now we consider a given edge $(w,z)$, whose passage time was completely constructed during the procedure described in Sections~\ref{sec:hexp} and ~\ref{sec:coupling}. 
   Let $t_w<t_z$ be the times such that $w$ is infected at time $t_w$ and $z$ is infected at time $t_z$. The passage time of $(w,z)$ is then completed at time $t_z$ and is $t_z-t_w$.
   
   Now we split the proof into two parts. 
   In the first part, assume that $w$ does not host a hole at time $t_w$ (for example, because it was infected according to Case 3, Section~\ref{sec:noholes}). 
   The crucial property we will use in this case is that $(w\to z)$ cannot belong to $\mathcal{H}_B$ during $[t_w,t_z)$ 
   since an edge $(z\to \cdot)$ cannot be in $\mathcal{H}$ as $z$ is not infected and an edge 
   $(\cdot\to w)$ cannot be in $\mathcal{H}$ since $w$ was infected without a hole. 
   As the $h$-process evolves from $t_w$, at each step, we add $W$ to $\Lambda_{w,z}$, until we find a step where the clock that rings is the one of $(w\to z)$. 
   This is the procedure described in Lemma~\ref{lem:decomp} for the construction
   of independent exponential random variables of rate $1$. Hence, at the time the clock of $(w \to z)$ rings, call this time $s$, 
   we have that $\Lambda_{w,z}$ is an exponential random variable of rate 1, which is then added to $\Pi_{w,z}$.
   If at time $s$ we fall into the setting of Case $3$ (Section~\ref{sec:noholes}), then we only infect $z$ with probability
   $\frac{M-1}{M}$; otherwise we wait for the next time the clock of $(w\to z)$ rings, adding another exponential random variable of rate $1$ to the list $\Pi_{w,z}$, and iterating this procedure. 
   If at time $s$ we fall into the setting of Case $1$ (Section~\ref{sec:holeout}), we only 
   infect $z$ if $\mathfrak{B}_{w \to z}=0$, which occurs with probability $\frac{M-1}{M}$; otherwise we iterate this procedure since the hole will jump back to $w$ (or to the infected set before $(z\to w)$ rings).
   Therefore, each element of the list $\Pi_{w,z}$ is an independent random variable of rate $1$, and the number of elements is given by a geometric random variable of success probability $\frac{M-1}{M}$. 
   Lemma~\ref{lem:rsum} gives that $\tau_{w,z}$ is in this case an exponential random variable of rate $\frac{M-1}{M}$.
   Since this stochastically dominates $Z^{(M)}$ by Lemma~\ref{lem:zk}, this first part is completed.
   
   Now, for the second part, assume that $w$ hosts a hole at time $t_w$.
   Assume that that hole jumped to $w$ from a vertex $w'$.
   The first situation to imagine is that the hole jumped from $w'$ to $w$ at time $t_w$, which causes 
   $(w' \to w)$ to be added to $\mathcal{H}(t_w)$; thus $\mathfrak{B}_{w'\to w}=0$. 
   But, it could also be the case that $\mathfrak{B}_{w'\to w}=1$. For this to happen, the hole must have jumped from $w'$ to $w$ at a time $t_w'<t_w$, 
   which caused $(w' \to w)$ to be added to $\mathcal{H}(t_w')$ and $w$ not to be added to $h_{t_{w'}}$.
   Then, in a time $t_w$, the clock of an edge $(w'' \to w)$ (which is not part of $\mathcal{E}_{w'\to w}$) rings at time $t_w$ when $w$ is still occupied by a hole, triggering Case 3, which decides to infect $w$.  
   Regardless of which of the two situations above occurs, we know that 
   $(w \to z)\in \mathcal{E}_{w'\to w}$. 
   
   If $\mathfrak{B}_{w'\to w}=1$, then we know the hole will do a backtracking jump from $w$ to $w'$ before the clock of $(w\to z)$ rings, which will cause $(w'\to w)$ to be removed from $\mathcal{H}$. 
   At this point, $w$ will not have a hole anymore and we may proceed as in the first part of the proof, which implies that the passage time $\tau_{w,z}$ stochastically dominates an exponential 
   random variable of rate $\frac{M-1}{M}$.
   
   The most delicate case is when $\mathfrak{B}_{w'\to w}=0$. 
   Suppose that at time $s$ a clock from $\mathcal{E}_{w'\to w}$ rings. 
   Note that, since $|\mathcal{E}_{w'\to w}|=M$, we will have that $\Lambda_{w,z}=s-t_w$ is distributed as an 
   exponential random variable of rate $M$. Let $(x\to y)$ be the edge whose clock rang at time $s$.
   Then a few cases may happen.
   
   Case A: $(x \to y)=(w\to z)$. 
   This happens with probability $\frac{1}{M-1}$. 
   When this is the case, we will have to do the steps of Case 2 plus Case 1, which causes $(w\to z)$ to be added to $\mathcal{H}(s)$. Two subcases can then happen.
   
   Case A.1. With probability $\frac{M-1}{M}$ we have $\mathfrak{B}_{w\to z}=0$, so the hole at $z$ will not do a backtracking jump to $w$; hence we infect $z$ at time $s$. Therefore, with overall probability 
   $\frac{1}{M-1} \times \frac{M-1}{M}= \frac{1}{M}$ we obtain that the passage time of $(w,z)$ is completed at time $s$, which implies that $\tau_{w,z}$ is an exponential random variable of rate $M$. 
   
   Case A.2. With probability $\frac{1}{M}$ we have $\mathfrak{B}_{w\to z}=1$, so the hole will do a backtracking jump to $w$ and we will not infect $z$. 
   From this time onwards, the passage time of $(w,z)$ is computed exactly as 
   in the first part of the lemma. Therefore, from $s$, we will wait a time that is distributed according to an exponential random variable of rate $\frac{M-1}{M}$ to complete the passage time of 
   $(w,z)$. This gives that $\tau_{w,z}$ is the sum of an exponential random variable of rate $M$ plus an exponential random variable of rate $\frac{M-1}{M}$.
   
   Case B. $(x\to y)\neq (w\to z)$. This happens with probability $\frac{M-2}{M-1}$ and concludes the processing of the backtracking jump of $(w' \to w)$.
   From this time onwards, the passage time of $(w,z)$ is computed exactly as 
   in the first part of the proof. 
   This gives that $\tau_{w,z}$ is the sum of an exponential random variable of rate $M$ plus an exponential random variable of rate $\frac{M-1}{M}$.
   
   This concludes all cases in this second part. Then, the probability that $\tau_{w,z}$ is given by the sum of an exponential random variable of rate $M$ and an exponential random variable of rate $\frac{M-1}{M}$ is 
   the probability that case $A.1$ does not happen, which is $1-\frac{1}{M}$. Therefore, $\tau_{w,z}$ is distributed as $Z^{(M)}$ from~\eqref{eq:zk}.

   The independence of $\tau$ across different edges follows from the fact that elements are added to each list $\Pi_{x,y}$ only when the edge $(x,y)$ rings, and edges ring independently of one another.
\end{proof}
\subsection{Concluding the proof of Theorem~\ref{thm:mdla}}\label{sec:endproof}
The proof will rely on a result by van den Berg and Kesten~\cite{BK} about strict inequalities for first passage percolation. 
We will state here a version that is adapted to our needs. In the one hand, it is a special case of the main result in~\cite{BK}, as we explain in Remark~\ref{rmk:bk} below. 
On the other hand, the result we need does not follow from the theorems in~\cite{BK}, but follows from the proof there, without changing a single word.
We will not repeat the full proof of~\cite{BK} here, but will give a description of the main steps. 

Recall that $\mathcal{B}_\Q=\ballQ{1}$ denotes the ball of radius 1 according to the norm given by a first passage percolation with passage times given by i.i.d.\ random variables of distribution $\Q$.
Recall also that we drop the subscript $\Q$ when $\Q$ is the exponential distribution of rate $1$.
\begin{proposition}\label{pro:bk}
   If $\Q$ stochastically dominates and is not equal to the exponential distribution of rate 1, then there exists $\epsilon>0$ such that 
   $$
      \mathcal{B}_\Q \subset \mathcal{B}(1-\epsilon).
   $$
\end{proposition}
\begin{proof}
   Let $\FPP_1$ stands for a first passage percolation process with i.i.d.\ exponential passage times of rate $1$, and let $\FPP_\Q$ stands for a first passage percolation
   with i.i.d.\ passage times of distribution $\Q$.
   This proposition is implicitly proven in~\cite{BK}, but the theorem in~\cite{BK} states the above result only in the axial direction.
   To do this, let $x_n=(n,0,0,\ldots,0)\in\mathbb{Z}^d$, and for any $x\in \mathbb{Z}^d$, define $T(x)$ and $T_\Q(x)$ to be the time that $\FPP_1$ and $\FPP_\Q$, respectively, take to occupy $x$.
   By Kingman's subadditive ergodic theorem~\cite{Kingman} 
   we have that the following limits exists almost surely:
   \begin{equation}
      \nu = \lim_{n\to \infty} \frac{T(x_n)}{n}
      \quad \text{ and }\quad
      \nu_\Q= \lim_{n\to \infty} \frac{T_\Q(x_n)}{n}.
      \label{eq:nu}
   \end{equation}
   By monotonicity and stochastic domination, we obtain that $\nu\leq \nu_\Q$, and the main result of~\cite{BK} is to establish the strict inequality $\nu< \nu_\Q$. 
   
   The proof in~\cite{BK} goes via a renormalization argument. 
   First define a fixed, but large value $\ell$ and partition $\mathbb{Z}^d$ into cubes of side-length $\ell$. 
   Then they say that a cube $R$ is \emph{good} if a certain ``good'' event happens. The good event is such that for any given path $P$
   of $\FPP_\Q$ inside $P$, there is a positive probability (uniformly over $P$ and $R$) such that 
   can find an alternative path $P'$ which differs very little from $P$ 
   and such that the time $\FPP_1$ takes to traverse $P'$ is at most the time that $\FPP_\Q$ takes to traverse $P$ minus a fixed value $\delta>0$. 
   Then a percolation argument (which is by now quite standard) gives that the set of 
   good cubes percolates on $\mathbb{Z}^d$. 
   This means that any long enough path on $\mathbb{Z}^d$, say of size $n$, must pass through a number of good cubes of order $n$. (Here $\ell$ and $\delta$ are fixed, while $n$ 
   can grow.) Then, we can consider
   the geodesic path $P$ that $\FPP_\Q$ takes to go from $0$ to $x_n$, and using the above reasoning we obtain an order of $n$ good cubes $R$ such that $P$ has a long piece inside $R$. 
   Then, for each good cube, with positive probability we can replace the 
   long pieces of $P$ within the good cube with the alternative path provided by the definition of good cubes, 
   which produces a path from $0$ to $x_n$ whose passage time in $\FPP_1$ is faster than that of $\FPP_\Q$ by a factor of order $n$.
   This implies, for example, that one obtains a value $\delta'>0$, depending only on $\Q$, for which $\nu \leq (1-\delta) \nu_\Q$. 
   This is the argument in~\cite{BK}.
   
   An important feature of the proof in~\cite{BK} is that it does not depend on the direction; this fact was already observed in~\cite{Marchand}. 
   In other words, instead of only considering the sequence of vertices $(x_n)_n$ as defined above, 
   we can consider any rational value $q\in\mathbb{Q}^d$ and a sequence 
   $x_n = q\, n$. 
   Then, for each $x\in\mathbb{R}^d$, associate $x$ to the integer point $y\in\mathbb{Z}^d$ such that $x\in y+[-1/2,1/2)^d$, and 
   generalize $T(x)$ and $T_\Q(x)$ to be the time that $\FPP_1$ and $\FPP_\Q$, respectively, take to occupy the integer point that is associated to $x$. 
   Then we can define $\nu(q)$ and $\nu_\Q(q)$ as in~\eqref{eq:nu}: 
   $$
      \nu(q) = \lim_{n\to \infty} \frac{T(x_n)}{n}
      \quad \text{ and }\quad
      \nu_\Q(q) = \lim_{n\to \infty} \frac{T_\Q(x_n)}{n}.
   $$
   The very same proof in~\cite{BK} gives that one can find $\delta''>0$ depending only on $\Q$ such that, 
   \emph{uniformly} over all $q\in \mathbb{Q}^d$, one has 
   $\nu(q) \leq (1-\delta'')\nu_\Q(q)$. 
   
   Then one can obtain two continuous functions $\bar\nu,\bar\nu_\Q$ from $\mathbb{R}^d$ to $\mathbb{R}_+$ by taking the unique continuous extension of $\nu,\nu_\Q$. 
   Since $\delta''>0$ uniformly on the choice of $q$, we have the existence of $\epsilon>0$ for which $\mathcal{B}_\Q \subset \mathcal{B}(1-\epsilon)$.
\end{proof}

\begin{remark}~\label{rmk:bk}
   The result in~\cite{BK} is in some sense more general than stated above since instead of requiring stochastic domination, it just requires that $\Q$ 
   is \emph{less variable} than an exponential random variable of rate $1$; but we will not require such level of generality. 
\end{remark}

\begin{proof}[Proof of Theorem~\ref{thm:mdla}]
   Using Lemma~\ref{lem:hp}, we have that the $h$-process grows slower than if the red seeds grows clusters of first passage percolation with passage times distributed as $Z^{(M)}$. 
   Let $\Q$ be the distribution of $Z^{(M)}$. Then, 
   by Proposition~\ref{pro:bk}, we have the existence of $\lambda<1$ such that $\mathcal{B}_\Q \subseteq \mathcal{B}(\lambda)$.
   Then, performing the whole multi-scale procedure described in Section~\ref{sec:proof} and using the encapsulation procedure of Section~\ref{sec:hp}, we obtain that with positive probability 
   the set of sites that are not infected by the $h$-process and that are occupied by the aggregate grows indefinitely and contains a ball (up to regions separated from infinity by this set of sites).
   
   Note that, in the $h$-process, whole clusters of seeds are added when the aggregate tries to occupy a seed, while in \fpphe{} seeds are activated one by one. 
   This is not a problem. In fact, the same proof works if in \fpphe{} the activation of a seed implies the activation of its whole cluster of seeds. 
   The reason is that, in the encapsulation procedure of Proposition~\ref{pro:encapsulate2}, we already assumed that $\xi^2(0)$ starts from time $0$ from any subset of a ball $\ball{r}$ of radius $r$. 
\end{proof}

\appendix
\section{Proof of Propositions~\ref{pro:encapsulate} and~\ref{pro:encapsulate2}}\label{sec:HPproof}
We first establish Proposition~\ref{pro:encapsulate} and at the end, in Section~\ref{sec:encapsulate2}, we discuss how the proof can be changed to establish 
Proposition~\ref{pro:encapsulate2}.
We start describing the overall strategy of the proof, and setting up the notation.
The main intuition behind the proof is that since $\xi^2(0)$ is initially inside $\ball{r}$ and $\xi^1(0)$ is 
outside $\ball{\alpha r}$, with $\alpha>1$ being large enough, there is enough space for $\xi^1$ to start growing before 
noticing the presence of $\xi^2$. This gives enough time for $\xi^1$ and $\xi^2$ to get closer to the set predicted by the shape theorem.
Then we can guarantee that $\xi^1$ can encapsulate $\xi^2$ by letting $\xi^1$ occupy a sequence of growing annulus sectors centered at the origin.
This is illustrated in Figure~\ref{fig:encapsulate}.
\begin{figure}[htbp]
   \begin{center}
      \includegraphics[scale=.8]{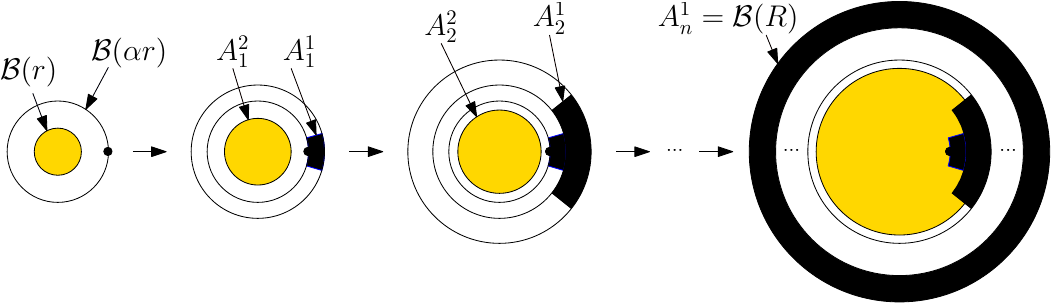}
   \end{center}
   \caption{A schematic view of the encapsulation procedure. The left picture shows the ball $\ball{r}$ that contains $\xi^2(0)$, and we assume that 
   $\xi^1(0)$ is at the boundary of $\ball{\alpha r}$. The proof goes by establishing that $\xi^1$ occupies the regions illustrated in black, while 
   $\xi^2$ is contained in the yellow regions. The growing sequence of annulus sectors that are occupied by $\xi^1$ grows until it reaches a full 
   annulus, as illustrated in the rightmost picture. At this moment, $\xi^2$ has been encapsulated and is confined inside a ball.} 
   \label{fig:encapsulate}
\end{figure}

We now turn to the details of the construction of the annulus sectors. Set 
$$
   \delta=\frac{1-\lambda}{10}.
$$ 
Note that $\delta \leq \frac{1}{\lambda^{1/10}}-1 \land \frac{1}{10}$.
Define
$$
   C_0=\{x/|x|\}
   \quad\text{and}\quad
   C_n = \lrc{y\in\mathbb{R}^d \colon |y|=1 \text{ and } \inf_{z\in C_{n-1}}|y-z|\leq \delta^2/2} \quad\text{for all $n\geq1$}.
$$
The value of $C_n$ is related to the angle of the annulus sector at step $n$, which starts from the angle related to position $x$ and increases until $C_n$ 
is the full unit circle, according to the norm $|\cdot|$. Let $N$ be the step where we obtain the unit circle; i.e., $N$ is the smallest integer so that 
$C_N=C_{N+1}$.   
Note that 
$$
   d_H(C_n,C_{n-1})\leq \delta^2/2,
$$ 
where $d_H$ stands for the Hausdorff distance.
Let 
$$
   A_0^1=\{x\}
   \quad\text{and}\quad
   A_n^1 = \lrc{y\in\mathbb{R}^d \colon (1+\delta)^{n-1}\alpha r< |y|\leq (1+\delta)^{n}\alpha r \text{ and } \tfrac{y}{|y|}\in C_n }
   \quad\text{for all $n\geq 1$}.
$$
The goal is to show that, for each $n=1,2,\ldots,N$, $\xi^1$ completely occupies $A_n^1$ after step $n$.
Hence $\xi^1$ will encapsulate $\xi^2$ when it occupies $A_N^1$.

As in~\eqref{eq:cfpp}, we let $C_\FPP$ be a constant such that $\ball{r}\supset [-C_\FPP r, C_\FPP r]^d$ for all $r>0$, which gives that 
$\ball{\frac{3}{2C_\FPP}}\supset [-3/2,3/2]^d$.
Since for any point $w\in \mathbb{R}^d$ we have that $w+[-1/2,+1/2]^d$ contains a vertex $\bar w \in \mathbb{Z}^d$, and at least one vertex in $\bar w + [-1,1]^d$ must have 
norm $|\cdot|$ smaller than that of $\bar w$, we obtain that 
\begin{equation}
   \text{for all $w\in\mathbb{R}^d$, there exists $w'\in\mathbb{Z}^d \cap \left(w+\ball{\tfrac{3}{2C_\FPP}}\right)$ such that $|w'|\leq |w|$}.
   \label{eq:norm}
\end{equation}

In order to show that $\xi^1$ occupies $A_n^1$ for all $n$, we need to bound the distance between $A_n^1$ and $A_{n-1}^1$. 
Given $y\in A_{n}^1$, let $\hat y$ be the closest vertex of $\mathbb{Z}^d$ to the point $\frac{y}{|y|}(1+\delta)^{n-1}\alpha r$.
Using the triangle inequality we have 
\begin{align}
   \sup_{y\in A_n^1}\inf_{z\in A_{n-1}^1} |y-z|
   &\leq \sup_{y\in A_n^1}\left(|y-\hat y|+\inf_{z\in A_{n-1}^1} |z-\hat y|\right).\nonumber
\end{align}
Since the ball in $\mathbb{R}^d$ centered at the point $\frac{y}{|y|}(1+\delta)^{n-1}\alpha r$ and of radius $(1+\delta)^{n-1}\alpha r\,d_H(C_{n-1},C_n)$ must contain a point $w\in\mathbb{R}^d$ with $\frac{w}{|w|}\in C_{n-1}$
and $|w|=(1+\delta)^{n-1}\alpha r$, we obtain that 
$$
   \inf_{z\in A_{n-1}^1} |z-\hat y| \leq (1+\delta)^{n-1}\alpha r \,d_H(C_{n-1},C_n) + \frac{3}{C_\FPP}.
$$
Thus, using that $|y-\hat y|\leq \left((1+\delta)^{n}-(1+\delta)^{n-1}\right)\alpha r+\frac{3}{2 C_\FPP}$, we obtain
\begin{align}
   \sup_{y\in A_n^1}\inf_{z\in A_{n-1}^1} |y-z|
   &\leq (1+\delta)^{n-1}\alpha r \frac{\delta^2}{2} + \left((1+\delta)^{n}-(1+\delta)^{n-1}\right)\alpha r+ \frac{9}{2C_\FPP}\nonumber\\
   &\leq (1+\delta)^{n-1}(1+\delta/2)\delta \alpha r+ \frac{9}{2C_\FPP}\nonumber\\
   &\leq (1+\delta)^{n}\delta \alpha r,
   \label{eq:distan}
\end{align}
where the last step holds by choosing $c_1$ large enough in the condition of $\alpha$.
This leads us to define
$$
   t_n = (1+\delta)^{n+1}\delta\alpha r
   \quad\text{and}\quad
   T_n = \sum_{i=1}^n t_n
   = \lr{1-(1+\delta)^{-n}}(1+\delta)^{n+2}\alpha r.
$$
The value $t_n$ represents the time we will wait so that $\xi^1$ grows from $A^1_{n-1}$ until it contains $A^1_{n}$.
For all $n\geq 1$, define also the sets 
$$
   A_n^2 = \ball{r+\lambda \lr{1-(1+\delta)^{-n}}(1+\delta)^{n+3}\alpha r}
$$
and
$$
   \hat A_n^1 = U_n \cup \bddo U_n,
$$
where
$$
   U_n = \lrc{x\in\mathbb{Z}^d \colon \exists y\in A_{n-1}^1 \text{ and } z\sim x\text{ for which } |z-y| \leq (1+\delta)^{n+3}\delta\alpha r} \supset A_{n-1}^1 \cup A_{n}^1
$$
Note that $A_n^2 = \ball{r+\lambda (1+\delta) T_n}$ and the distance between $A_{n-1}^1$ and $\bddi \hat A_n^1$ is at least $(1+\delta)t_n$; we have chosen to define $A_n^2$ and $\hat A_n^1$ independently of $t_n$ because later 
we will apply Proposition~\ref{pro:encapsulate} with a time scaling, which will only cause a change in the definition of $t_n$ in this proof.

The idea is that after occupying $A_{n-1}^1$, $\xi^1$ will occupy $A_n^1$ after a time interval of length $t_n$, and this is achieved 
only using passage times inside $\hat A_{n}^1$. At the same time, for each $n$, after time $T_n$, $\xi^2$ will be contained 
inside $A_n^2$. The crucial part of the construction is that $\hat A_n^1 \cap A_n^2=\emptyset$. This means that the passage times inside 
$\hat A_n^1$ that we use to guarantee that $\xi^1$ grows from $A^1_{n-1}$ to $A^1_{n}$ do not intersect $\xi^2$. 

\subsection{The spread of $\xi^2$}
In this part we show that the growth of $\xi^2$ is not too fast, so that $\xi^2$ is contained 
inside $A_n^2$. Recall that, by the conditions in Proposition~\ref{pro:encapsulate}, $\alpha$ is assumed to be large enough; in 
particular, there exists a large $c_1$ such that $\alpha > \left(\frac{1}{\lambda(1-\lambda)}\right)^{c_1}$.
\begin{lemma}\label{lem:eta2}
   Assume that $\alpha$ is large enough, as described above. 
   Then there exists a constant $c>0$ such that, for any $n\geq 1$, we have 
   $$
      \PQ_{x,\ball{r}}^\Q\lr{\xi^2(T_n) \not \subseteq A_n^2}
      \leq \exp\lr{-c (\lambda T_n)^{\frac{d+1}{2d+4}}}.
   $$
\end{lemma}
\begin{proof}
   Let $\lambda \Q$ denote the distribution of $\lambda\zeta^2_{x,y}$.
   By time scaling, we have that 
   \begin{align*}
      \PQ_{x,\ball{r}}^\Q\lr{\xi^2(T_n) \not \subseteq A_n^2}
      \leq \PQ^{\lambda\Q}\lr{D(\ball{r},\bddo A_n^2; \zeta) \leq \lambda T_n}
      \leq \sum_{x\in \ball{r}}\PQ^{\lambda\Q}\lr{D(x,\bddo A_n^2; \zeta) \leq \lambda T_n}.
   \end{align*}
   By choosing $c_1$ large enough in the condition of $\alpha$, we can guarantee that $\delta\geq (\lambda T_1)^{-\frac{1}{2d+4}} \geq (\lambda T_n)^{-\frac{1}{2d+4}}$, which allows us to apply 
   Proposition~\ref{pro:kesten}. Then using that 
   \begin{equation}
      \inf_{x\in \ball{r},y\in \bddo A_n^2} |x-y| \geq (1+\delta)\lambda T_n,
      \label{eq:appdist2}
   \end{equation}
   we obtain 
   \begin{equation}
      \PQ^{\lambda\Q}\lr{D(x,\bddo A_n^2; \zeta) \leq \lambda T_n}
      \leq \PQ^{\lambda\Q}\lr{S_{\lambda T_n}^\delta(\zeta)}
      \leq c'\exp\lr{-c'' (\lambda T_n)^{\frac{d+1}{2d+4}}},
      \label{eq:k1}
   \end{equation}
   for positive constants $c',c''$.
   Since the number of vertices in $\ball{r}$ is of order $r^d$, and $T_n\geq T_1$ is large enough by the condition in $\alpha$, the lemma follows.
\end{proof}

\subsection{The spread of $\xi^1$}
Here we show that the growth of $\xi^1$ is fast enough, so that $\xi^1$ occupies $A_n^1$ at time $T_n$.
\begin{lemma}\label{lem:eta1}
   Assume that $\alpha$ is large enough, as in the statement of Proposition~\ref{pro:encapsulate}.
   There exists a positive constant $c$ such that, for any $n\geq 1$, we have 
   $$
      \PQ_{x,\ball{r}}\lr{\sup_{x\in A_n^1} D(A_{n-1}^1,x;\zeta^1) > t_n}
      \leq \exp\lr{-c t_n^\frac{d+1}{2d+4}}.
   $$
\end{lemma}
\begin{proof}
   We start writing
   $$
      \PQ_{x,\ball{r}}\lr{\sup_{x\in A_n^1} D(A_{n-1}^1,x;\zeta^1) > t_n}
      \leq \sum_{x\in A_n^1}  \PQ\lr{D(A_{n-1}^1,x;\zeta) > t_n}.
   $$
   From~\eqref{eq:distan} we have that
   \begin{equation}
      \sup_{x\in A_n^1}\inf_{y\in A_{n-1}^1} |x-y| \leq \frac{t_n}{1+\delta}
      =t_n\lr{1 - \frac{\delta}{1+\delta}}.
      \label{eq:appdist1}
   \end{equation}
   Applying Proposition~\ref{pro:kesten} and using that the number of vertices in $A_n^1$ is bounded above by $c' (1+\delta)^{n-1}\delta\alpha r$, we have 
   \begin{align*}
      \PQ_{x,\ball{r}}\lr{\sup_{x\in A_n^1} D(A_{n-1}^1,x;\zeta^1) > t_n}
      \leq \sum_{x\in A_n^1}  \PQ\lr{S_{t_n}^\frac{\delta}{1+\delta}}
      \leq c'' (1+\delta)^{n-1}\delta\alpha r t_n^d\exp\lr{-c''' t_n^\frac{d+1}{2d+4}}.
   \end{align*}
   The lemma then follows since $\alpha> \left(\frac{1}{\lambda(1-\lambda)}\right)^{c_1}$ for some large enough $c_1$.
\end{proof}

The lemma below gives that it is unlikely that $\xi^1$ leaves the set $\hat A_n^1$ in the $n$th step. 
It may sound a bit counterintuitive that we need to guarantee that 
$\xi^1$ is not too fast, but this lemma is required to ensure that $\xi^1$ occupies $A_n^1$ regardless of the passage times outside 
$\hat A_n^1$.
\begin{lemma}\label{lem:eta1b}
   Assume that $\alpha$ is large enough, as in the statement of Proposition~\ref{pro:encapsulate}.
   There is a positive constant $c$ such that, for any $n\geq 1$, we have 
   $$
      \PQ_{x,\ball{r}}\lr{D(A_{n-1}^1,\bddi \hat A_{n}^1;\zeta^1) \leq t_n}
      \leq \exp\lr{-c t_n^{\frac{d+1}{2d+4}}}.
   $$
\end{lemma}
\begin{proof}
   Recall that 
   \begin{equation}
      \inf_{x\in A_{n-1}^1,y\in\bddi\hat A_n^1} |x-y|
      \geq (1+\delta)t_n.
      \label{eq:appdist1b}
   \end{equation}
   Then, since the number of vertices in $A_{n-1}^1$ can be bounded above by $c' (1+\delta)^{n-2}\delta\alpha r$ for some constant $c'>0$, 
   applying Proposition~\ref{pro:kesten} we obtain
   $$
      \PQ_{x,\ball{r}}\lr{D(A_{n-1}^1,\bddi \hat A_{n}^1;\zeta^1) \leq t_n}
      \leq \sum_{x\in A_{n-1}^1}  \PQ\lr{S_{t_n}^\delta}
      \leq c' (1+\delta)^{n-2}\delta\alpha r t_n^{2d}\exp\lr{-c'' t_n^{\frac{d+1}{2d+4}}}.
   $$
   The lemma then follows since $\alpha> \left(\frac{1}{\lambda(1-\lambda)}\right)^{c_1}$ for some large enough $c_1$.
\end{proof}

\subsection{Completing the Proof of Proposition~\ref{pro:encapsulate}}
\begin{proof}[Proof of Proposition~\ref{pro:encapsulate}]
   For any $X\subset\mathbb{Z}^d$, let $\zeta^1|_X$ be a set of passage times to the edges of $\mathbb{Z}^d$ such that 
   they are equal to $\zeta^1$ for any edge whose both endpoints belong to $X$ and are equal to infinity everywhere else.
   For each integer $n$, define the events
   \begin{align*}
      E_n^{(1)} = \Big\{D(\ball{r},\bddo A_n^2;\zeta^2)\leq T_n\Big\}, \quad
      E_n^{(2)} = \Big\{\sup\nolimits_{x\in A_n^1} D\lr{A_{n-1}^1,x;\zeta^1|_{\hat A_n^1}}> t_n\Big\}\\
      \quad\text{and}\quad
      E_n = E_n^{(1)} \cup E_n^{(2)}.
   \end{align*}
   We define the event $F$ in the proposition by $F= \bigcap_{n=1}^N E_n^\compl$. 
   We also define $R=(1+\delta)^{N}\alpha r$ and $T=T_N$.
   By Lemmas~\ref{lem:eta2}, \ref{lem:eta1} and \ref{lem:eta1b}, we have 
   $$
      \PR\lr{F^\compl}
      \leq \sum_{n\geq 1} \exp\lr{-c (\lambda t_n)^\frac{d+1}{2d+4}}
      \leq 2\exp\lr{-c (\lambda t_1)^\frac{d+1}{2d+4}}
      \leq \exp\lr{-c' (\lambda(1-\lambda) \alpha r)^\frac{d+1}{2d+4}},
   $$
   where the last inequality follows since $\lambda t_1 = (1+\delta)^2\lambda \delta \alpha r$.
   This establishes the bound in the probability appearing in Proposition~\ref{pro:encapsulate}.
   
   Since $E_n^{(1)}$ is measurable with respect to $A_n^2$ and 
   $E_n^{(2)}$ is measurable with respect to $\hat A_n^1$, we have that $F$ is measurable with respect to the passage times inside 
   $$
      \bigcup_{n=1}^N \lr{\hat A_n^1 \cup A_n^2}
      \subset \ball{(1+\delta)^{N+2}\alpha r}
      = \ball{\lr{\frac{11-\lambda}{10}}^2(1+\delta)^{N}\alpha r}
      = \ball{\lr{\frac{11-\lambda}{10}}^2 R}.
   $$
   Note that $F$ is increasing with respect to $\zeta^2$ and decreasing with respect to $\zeta^1|_{\ball{\lr{\frac{11-\lambda}{10}}^2 R}}$.
   
   To conclude the proof of the proposition, it suffices to show that 
   $F$ implies that $\xi^1$ occupies all vertices in $\bigcup_{n=1}^N A_n^1$, since 
   $$
      A^1_N = \ball{R}\setminus \ball{\frac{R}{1+\delta}}
      = \ball{R}\setminus \ball{\frac{10R}{11-\lambda}}.
   $$
   We will use induction on $n$ 
   to establish a stronger result by showing that, for each $n$, given that $\xi^1(T_{n-1})\supset A_{n-1}^1$,
   the event $E_n^\compl$ implies that $\xi^1(T_{n})\supset A_{n}^1$.
   First, for $n=0$ we have from the initial condition that 
   $\xi^1(0) \supset A_0^1$. 
   Now assume that $\xi^1(T_{n-1})\supset A_{n-1}^1$. Since $E_n^{(1)}$ does not hold, we have that 
   \begin{align*}
      \xi^2(T_n)
      \subset A_n^2
      &= \ball{r+\lambda \lr{1-(1+\delta)^{-n}}(1+\delta)^{n+3}\alpha r}\\
      &\subset \ball{r+(1+\delta)^{n+3}\lambda \alpha r}\\
      &\subset \ball{r+(1+\delta)^{n-7}\alpha r}\\
      &\subset \ball{(1+\delta)^{n-6}\alpha r},
   \end{align*}
   where the second-to-last step follows since $1+\delta=\frac{11-\lambda}{10}$ and the function 
   $\lr{\frac{11-x}{10}}^{10}x\leq 1$ for all $x\in [0,1]$. The last step follows since $\alpha$ is large enough.
   Now we note that $\hat A_n^1$ does not intersect 
   \begin{align}
      \ball{(1+\delta)^{n-2}\alpha r-(1+\delta)^{n+4} \delta\alpha r}
      &= \ball{(1+\delta)^{n-2}\alpha r(1-(1+\delta)^6\delta)}\nonumber\\
      &\supset \ball{(1+\delta)^{n-6}\alpha r}
      \supset A_n^2.
      \label{eq:anhat}
   \end{align}
   Since $E_n^{(2)}$ does not happen, and $\xi^2(T_n)\cap \hat A_n^1=\emptyset$, 
   the passage times inside $\hat A_n^1$ guarantee that $\xi^1(T_n)\supset A_n^1$, concluding the proof.
\end{proof}

\subsection{Proof of Proposition~\ref{pro:encapsulate2}}\label{sec:encapsulate2}
\begin{proof}[Proof of Proposition~\ref{pro:encapsulate2}]
   The proof of Proposition~\ref{pro:encapsulate} goes by showing that, for each $n\geq 1$, given that $\xi^1$ occupies $A_{n-1}^1$, $\xi^1$ will occupy $A_{n}^1$ after time $t_n$, $\xi^2$ will be 
   confined to $A_n^2$, and these events are measurable with respect to $\hat A_n^1\cup A_n^2$.
   Moreover, we have by~\eqref{eq:anhat} that $A_n^2$ does not intersect $\hat A_n^1$, which guarantees that for each $n$ the events for $\xi^1$ and $\xi^2$ at the $n$th step are independent. 
   In the setting of Proposition~\ref{pro:encapsulate2}, the main difference is that, with the presence of the sets $\Pi_i$, some parts of $A_n^1,A_n^2,\hat A_n^1$ may intersect $\bigcup_i\Pi_i$.
%
%
   Below, when we refer to the properties of $\Pi_i$, we mean the enumerated properties \ref{it:pi1}--\ref{it:pi3} described for $\Pi_i$ right before Proposition~\ref{pro:encapsulate2}.
 
   First we consider the effect of the sets $\Pi_i$ for the spread of $\xi^2$.  
   Note that, for each $i$, if $\xi^2(0)$ does not intersect $\Pi_i$, then due to properties~\ref{it:pi1} and~\ref{it:pi3}, the spread of $\xi^2$ inside $\Pi_i$ is slower than that given by $\zeta^2$, just as in Proposition~\ref{pro:encapsulate}. 
   On the other hand, if $\xi^2(0)$ intersects some $\Pi_i$, then $\xi^2$ may benefit from the set of vertices in $\Pi_i$ that is occupied by type 2. 
   This decreases the distance between $\ball{r}$ and $A_n^2$ by at most $\gamma r$. 
   Moreover, to ensure that the event for $\xi^2$ at the $n$th step is measurable with respect to $A_n^2$, we will consider the event 
   that $\xi^2$ is confined to $A_n^2\setminus \bddi_{\gamma r} A_n^2$, where 
   $\bddi_{\gamma r} A_n^2$ stands for the vertices of $A_n^2$ within distance $\gamma r$ from $\bddi A_n^2$.
   Therefore,
   we define 
   $$
      E_n^{(1)} = \Big\{D(\ball{(1+\gamma)r},\bddo (A_n^2\setminus \bddi_{\gamma r} A_n^2);\zeta^2)\leq T_n\Big\}.
   $$
   For Proposition~\ref{pro:encapsulate}, the probability of the corresponding event is given in Lemma~\ref{lem:eta2}, and follows from inequality~\eqref{eq:appdist2}.
   Here,~\eqref{eq:appdist2} translates to
   \begin{align*}
      \inf_{x\in \ball{(1+\gamma)r},y\in \bddo (A_n^2\setminus \bddi_{\gamma r} A_n^2)} |x-y| 
      &\geq \lambda\left(1-(1+\delta)^{-n}\right)(1+\delta)^{n+3}\alpha r - 2\gamma r \\
      &\geq \lambda T_n\left(1+\delta - \frac{2\gamma}{\alpha}\right) \\
      &\geq (1+\delta/2)\lambda T_n,
   \end{align*}
   and Lemma~\ref{lem:eta2} follows by adjusting the constant $c$.
   
   Now we turn to the effect of the $\Pi_i$ on the spread of $\xi^1$.
   There are two aspects regarding the spread of $\xi^1$: the time to spread from $A_{n-1}^1$ to $A_n^1$ (handled in Lemma~\ref{lem:eta1}) and the measurability of this event (handled in Lemma~\ref{lem:eta1b}).
   Regarding Lemma~\ref{lem:eta1}, the crucial inequality is~\eqref{eq:appdist1}. 
   But since $A_{n-1}^1$ and $A_n^1$ may intersect $\bigcup_i\Pi_i$, to ensure that 
   $\xi^1$ can spread in the $n$th step, we need to change $E_n^{(2)}$ to 
   $$
      E_n^{(2)} = \Big\{\sup\nolimits_{x\in A_n^1(\Pi)} D\lr{A_{n-1}^1\setminus \bddi_{\gamma r}A_{n-1}^1,x;\zeta^1|_{\hat A_n^1}}> t_n\Big\},
   $$ 
   where $A_n^1(\Pi)$ is the set $A_n^1$ plus all sets $\Pi_i$ that intersect $A_n^1$; that is, 
   $A_n^1(\Pi)= A_n^1 \cup \left(\bigcup_{i\colon \Pi_i\cap A_n^1\neq \emptyset}\Pi_i \right)$.
   Then, using property~\ref{it:pi2},~\eqref{eq:appdist1} translates to 
   \begin{align*}
   \sup_{x\in A_n^1(\Pi)}\inf_{y\in A_{n-1}^1\setminus \bddi_{\gamma r}A_{n-1}^1} |x-y| 
      &\leq (1+\delta)^n\delta\alpha r+2\gamma r\\
      &\leq t_n\lr{1 - \frac{\delta}{1+\delta}+\frac{2\gamma}{\delta \alpha}}\\
      &\leq t_n\lr{1 - \frac{\delta}{2(1+\delta)}},
   \end{align*}
   and Lemma~\eqref{lem:eta1} follows by adjusting $c$.
   
   Regarding Lemma~\ref{lem:eta1b}, $\hat A_n^1$ will not need to be changed, and~\eqref{eq:appdist1b} is replaced with 
   $$
      \inf_{x\in A_{n-1}^1,y\in\bddi_{\gamma r}\hat A_n^1} |x-y|
      \geq (1+\delta)^{n+3}\delta\alpha r -\gamma r- \frac{2}{C_\FPP}
      \geq (1+\delta)t_n,
   $$ 
   where $\frac{2}{C_\FPP}$ accounts for the fact that the definition of $\hat A_n^2$ includes neighborhoods in $\mathbb{Z}^d$ and $\ball{2/C_\FPP}$ contains all vertices in $[-2,2]^d$. 
   Then Lemma~\ref{lem:eta1b} holds as it is.
   Finally, we just need to ensure that the events $E_n^{(1)}$ and $E_n^{(2)}$ are independent. In other words, that $\hat A_n^1$ and $A_n^2$ do not intersect. But this is true since their definition did not change; hence~\ref{eq:anhat}
   holds as it is. 
\end{proof}
\section{Standard properties of exponential random variables}\label{sec:exp}

Here we state some properties of exponential random variables that we will use in the paper.
\begin{lemma}[Random sum]\label{lem:rsum}
   Fix any $q\in(0,1)$. Let $L$ be a geometric random variable of success probability $q$. 
   Let $X_1,X_2,\ldots$ be an i.i.d.\ sequence of exponential random variables of rate $1$.
   Hence,
   $$
      \sum_{i=1}^L X_i \text{ is an exponential random variable of rate $q$}.
   $$
\end{lemma}

\begin{lemma}[Scaling and minimum]\label{lem:minscale}
   Let $X$ be an exponential random variable of rate $\theta$.
   Then, for any $M>0$, we have 
   $$
      \frac{X}{M} \text{ is an exponential random variable of rate $M\theta$}.
   $$
   Moreover, for integer $M$, $\frac{X}{M}$ has the same distribution of the minimum of $M$ independent, exponential random variables of rate $\theta$. 
\end{lemma}


In the lemma below we show that a collection of exponential random variables $Z_1,Z_2,\ldots,Z_{k}$ can be sampled by first sampling the minimum value among all of them, which is the variable $Z_I$ whose value is $W$, and then 
using the memoryless property of exponential random variables to say that the other ones are equal to $W$ plus an exponential random variable of the same rate. 
\begin{lemma}[Decomposition on the minimum]\label{lem:decomp}
   Fix any integers $k\geq 1$. Let $X_i$, $i=1,2,\ldots, k$, be independent exponential random variables of rate $\theta_i$. 
   Let $I$ be a random variable in $\{1,2,\ldots,k\}$ which has value $i$ with probability $\frac{\theta_i}{\sum_{j=1}^k\theta_j}$.
   Let $W$ be an independent, exponential random variable of rate $\sum_{j=1}^k\theta_j$.
   Thus, if we set $Z_i$, $i=1,2,\ldots,j$, as 
   $$
      Z_i = W + \ind{I\neq i} X_i,
   $$
   we obtain that the $Z_i$ are independent exponenential random variables of rate $\theta_i$.  
\end{lemma}

Note that the above lemma can be iterated. That is, after we see that $Z_I=W$ is the minimum among the $Z_i$, then the value of $Z_i$ for $i\neq I$ can be sampled by 
first sampling the minimum among the $X_i$ with $i\neq I$. Thus we obtain new random variables $I'$ and $W'$ so that $X_{I'}=W'$ and $Z_{I'}=W+W'$, while the other values of $Z_i$ with $i\neq I,I'$ are equal to $W+W'$ 
plus an independent exponential random variable of the same rate. 
Then, after having sampled the values of $Z_I$ and $Z_{I'}$, we can iterate the above procedure with the $Z_i$ that were not yet sampled, 
for $k-2$ iterations, until all the $Z_i$'s have been obtained.

\bibliographystyle{plain}
\bibliography{mdla}

\begin{thebibliography}{10}

\bibitem{ADH}
A.~{Auffinger}, M.~{Damron}, and J.~{Hanson}.
\newblock {\em {50 years of first passage percolation}}.
\newblock American Mathematical Society, 2017.

\bibitem{Barlow}
M.T. Barlow.
\newblock Fractals, and diffusion-limited aggregation.
\newblock {\em Bull. Sci. Math.}, 117(1):161--169, 1993.

\bibitem{BPP}
M.T. Barlow, R.~Pemantle, and E.A. Perkins.
\newblock Diffusion-limited aggregation on a tree.
\newblock {\em Probab. Theory Related Fields}, 107(1):1--60, 1997.

\bibitem{Procaccia}
F.~Barra, B.~Davidovitch, A.~Levermann, and I.~Procaccia.
\newblock Laplacian growth and diffusion limited aggregation: different
  universality classes.
\newblock {\em Physical Review Letters}, 87(13):134501, 2001.

\bibitem{BY}
I.~Benjamini and A.~Yadin.
\newblock Diffusion limited aggregation on a cylinder.
\newblock {\em Comm. Math. Phys.}, 279(1):187--223, 2008.

\bibitem{CSFPPHE}
E.~Candellero and A.~Stauffer.
\newblock Coexistence of competing first passage percolation on hyperbolic
  graphs, 2018.
\newblock Available at arXiv:1810.04593.

\bibitem{CM}
L.~Carleson and N.~Makarov.
\newblock Aggregation in the plane and {L}oewner's equation.
\newblock {\em Comm. Math. Phys.}, 216(3):583--607, 2001.

\bibitem{CSwindle}
L.~Chayes and G.~Swindle.
\newblock Hydrodynamic limits for one-dimensional particle systems with moving
  boundaries.
\newblock {\em Ann. Probab.}, 24(2):559--598, 1996.

\bibitem{CoxDurrett}
J.T. Cox and R.~Durrett.
\newblock Some limit theorems for percolation processes with necessary and
  sufficient conditions.
\newblock {\em Ann. Probab.}, 9(4):583--603, 1981.

\bibitem{Eldan}
R.~Eldan.
\newblock Diffusion-limited aggregation on the hyperbolic plane.
\newblock {\em Ann. Probab.}, 43(4):2084--2118, 2015.

\bibitem{GK}
Geoffrey Grimmett and Harry Kesten.
\newblock First-passage percolation, network flows and electrical resistances.
\newblock {\em Zeitschrift f{\"u}r Wahrscheinlichkeitstheorie und Verwandte
  Gebiete}, 66(3):335--366, Aug 1984.

\bibitem{HP2000}
O.~H{\"a}ggstr{\"o}m and R.~Pemantle.
\newblock Absence of mutual unbounded growth for almost all parameter values in
  the two-type {R}ichardson model.
\newblock {\em Stochastic Process. Appl.}, 90(2):207--222, 2000.

\bibitem{HL}
M.B. Hastings and L.S. Levitov.
\newblock Laplacian growth as one-dimensional turbulence.
\newblock {\em Physica D}, 116:244--252, 1998.

\bibitem{WS}
T.A.~Witten Jr. and L.M. Sander.
\newblock Diffusion-limited aggregation, a kinetic critical phenomenon.
\newblock {\em Physical Review Letters}, 47(19):1400--1403, 1981.

\bibitem{Kassner}
K.~Kassner.
\newblock {\em Pattern formation in diffusion-limited crystal growth}.
\newblock World Scientific, 1996.

\bibitem{Kesten1}
H.~Kesten.
\newblock How long are the arms in {DLA}?
\newblock {\em J. Phys. A}, 20(1):L29--L33, 1987.

\bibitem{Kesten2}
H.~Kesten.
\newblock Upper bounds for the growth rate of {DLA}.
\newblock {\em Phys. A}, 168(1):529--535, 1990.

\bibitem{Kesten93}
H.~Kesten.
\newblock On the speed of convergence in first-passage percolation.
\newblock {\em Ann. Appl. Probab.}, 3(2):296--338, 1993.

\bibitem{KS_dla2}
H.~Kesten and V.~Sidoravicius.
\newblock Positive recurrence of a one-dimensional variant of diffusion limited
  aggregation.
\newblock In {\em In and out of equilibrium. 2}, volume~60 of {\em Progr.
  Probab.}, pages 429--461. Birkh\"auser, Basel, 2008.

\bibitem{KS_dla}
H.~Kesten and V.~Sidoravicius.
\newblock A problem in one-dimensional diffusion-limited aggregation ({DLA})
  and positive recurrence of {M}arkov chains.
\newblock {\em Ann. Probab.}, 36(5):1838--1879, 2008.

\bibitem{Kingman}
J.~F.~C. Kingman.
\newblock Subadditive ergodic theory.
\newblock {\em Ann. Probab.}, 1(6):883--899, 1973.

\bibitem{Marchand}
R.~Marchand.
\newblock Strict inequalities for the time constant in first passage
  percolation.
\newblock {\em Ann. Appl. Probab.}, 12(3):1001--1038, 08 2002.

\bibitem{Martineau}
S\'{e}bastien Martineau.
\newblock Directed diffusion-limited aggregation.
\newblock {\em ALEA Lat. Am. J. Probab. Math. Stat.}, 14(1):249--270, 2017.

\bibitem{Richardson}
D.~Richardson.
\newblock Random growth in a tessellation.
\newblock {\em Proc. Cambridge Philos. Soc.}, 74:515--528, 1973.

\bibitem{RM}
H.~Rosenstock and C.~Marquardt.
\newblock Cluster formation in two-dimensional random walks: application to
  photolysis of silver halides.
\newblock {\em Physical Review B}, 22(12):5797--5809, 1980.

\bibitem{ST}
P.G. Saffman and G.I. Taylor.
\newblock The penetration of a fluid into a porous medium or hele-shaw cell
  containing a more viscous liquid.
\newblock {\em Proceedings of the Royal Society London A}, 245:312--329, 1958.

\bibitem{Silvestri}
Vittoria Silvestri.
\newblock Fluctuation results for hastings--levitov planar growth.
\newblock {\em Probability Theory and Related Fields}, 167(1):417--460, Feb
  2017.

\bibitem{BK}
J.~van~den Berg and H.~Kesten.
\newblock Inequalities for the time constant in first-passage percolation.
\newblock {\em Ann. Appl. Probab.}, 3(1):56--80, 02 1993.

\bibitem{Voss}
R.~Voss.
\newblock Multiparticle fractal aggregation.
\newblock {\em Journal of Statistical Physics}, 36(5/6):861--872, 1984.

\end{thebibliography}

\end{document}